\documentclass[reqno,final,a4paper]{amsart}
\usepackage{color}
\usepackage[colorlinks=true,allcolors=blue
,backref=page
]{hyperref}
\usepackage{amsmath, amssymb, amsthm}
\usepackage{mathrsfs}
\usepackage{mathtools}
\usepackage[noabbrev,capitalize,nameinlink]{cleveref}
\crefname{equation}{}{}
\usepackage{fullpage}
\usepackage[noadjust]{cite}
\usepackage{graphics}
\usepackage{pifont}
\usepackage{tikz}
\usepackage{bbm}
\usepackage[T1]{fontenc}
\usetikzlibrary{arrows.meta}

\usepackage{environ}
\usepackage{framed}
\usepackage{url}
\usepackage[linesnumbered,ruled,vlined]{algorithm2e}
\usepackage[noend]{algpseudocode}
\usepackage[labelfont=bf]{caption}
\usepackage{framed}
\usepackage[framemethod=tikz]{mdframed}
\usepackage{pgfplots}
\usepackage{appendix}
\usepackage{graphicx}
\usepackage[textsize=tiny]{todonotes}
\usepackage{tcolorbox}
\usepackage[shortlabels]{enumitem}
\crefformat{enumi}{#2#1#3}
\crefrangeformat{enumi}{#3#1#4 to~#5#2#6}
\crefmultiformat{enumi}{#2#1#3}
{ and~#2#1#3}{, #2#1#3}{ and~#2#1#3}
\allowdisplaybreaks

\let\originalleft\left
\let\originalright\right
\renewcommand{\left}{\mathopen{}\mathclose\bgroup\originalleft}
\renewcommand{\right}{\aftergroup\egroup\originalright}

\crefname{algocf}{Algorithm}{Algorithms}

\crefname{equation}{}{} 
\AtBeginEnvironment{appendices}{\crefalias{section}{appendix}} 

\usepackage[color]{showkeys} 

\colorlet{refkey}{orange!20}
\colorlet{labelkey}{blue!30}

\crefname{algocf}{Algorithm}{Algorithms}

\numberwithin{equation}{section}
\newtheorem{theorem}{Theorem}[section]
\newtheorem{proposition}[theorem]{Proposition}
\newtheorem{lemma}[theorem]{Lemma}
\newtheorem{claim}[theorem]{Claim}
\newtheorem*{claim*}{Claim}
\crefname{claim}{Claim}{Claims}

\newtheorem{corollary}[theorem]{Corollary}
\newtheorem{conjecture}[theorem]{Conjecture}
\newtheorem*{question*}{Question}
\newtheorem{fact}[theorem]{Fact}
\crefname{conjecture}{Conjecture}{Conjectures}
\crefname{fact}{Fact}{Facts}

\theoremstyle{definition}
\newtheorem{definition}[theorem]{Definition}

\newtheorem{question}[theorem]{Question}
\newtheorem*{definition*}{Definition}

\theoremstyle{remark}
\newtheorem{remark}[theorem]{Remark}
\newtheorem*{remark*}{Remark}

\DeclareMathOperator{\rank}{rank}

\newcommand{\mb}{\mathbb}

\newcommand{\mbm}{\mathbbm}
\newcommand{\mc}{\mathcal}

\newcommand{\mr}{\mathrm}

\newcommand{\on}{\operatorname}

\newcommand{\eps}{\varepsilon}
\renewcommand{\Pr}{\mb P}
\newcommand{\Ex}{\mb E}

\newcommand{\pvec}[1]{\vec{#1}\mkern2mu\vphantom{#1}'}

\newcommand*{\claimproofname}{Proof of claim}
\newenvironment{claimproof}[1][\claimproofname]{\begin{proof}[#1]}{\end{proof}}

\setlist{itemsep=0.3em} 

\allowdisplaybreaks

\title{Algebraic aspects of the polynomial Littlewood--Offord problem}

\author[Jin]{Zhihan Jin}

\address{Department of Mathematics, ETH, Z\"urich, Switzerland}
\email{zhihan.jin@math.ethz.ch}

\author[Kwan]{Matthew Kwan}

\address{Institute of Science and Technology Austria (ISTA). Am Campus 1, 3400 Klosterneuburg, Austria.}
\email{matthew.kwan@ist.ac.at}

\author[Sauermann]{Lisa Sauermann}

\address{Institute for Applied Mathematics, University of Bonn. Endenicher Allee 60, 53115 Bonn, Germany.}
\email{sauermann@iam.uni-bonn.de}

\author[Wang]{Yiting Wang}

\address{Institute of Science and Technology Austria (ISTA). Am Campus 1, 3400 Klosterneuburg, Austria.}
\email{yiting.wang@ist.ac.at}
\thanks{
The first author was supported by SNSF grant 200021-19696.
The second and fourth authors were supported by ERC Starting Grant ``RANDSTRUCT'' No.~101076777. The third author was supported by the DFG Heisenberg Program.
}

\makeatletter
\let\@wraptoccontribs\wraptoccontribs
\makeatother

\contrib[with an appendix by]{Matthew Kwan, Ashwin Sah and Mehtaab Sawhney}

\begin{document}

\begin{abstract}
Consider a degree-$d$ polynomial $f(\xi_{1},\dots,\xi_{n})$ of independent
Rademacher random variables $\xi_{1},\dots,\xi_{n}$. To what extent
can $f(\xi_{1},\dots,\xi_{n})$ concentrate on a single point? This
is the so-called \emph{polynomial Littlewood--Offord problem}. A
nearly optimal bound was proved by
Meka, Nguyen and Vu: the point probabilities are always at most about
$1/\sqrt{n}$, unless $f$ is ``close to the zero polynomial'' (having
only $o(n^{d})$ nonzero coefficients).

In this paper we prove several results supporting the general philosophy
that the Meka--Nguyen--Vu bound can be significantly improved unless
$f$ is ``close to a polynomial with special algebraic structure'',
drawing some comparisons to phenomena in analytic number theory. In
particular, one of our results is a corrected version of a conjecture
of Costello on multilinear forms (in an appendix with Ashwin
Sah and Mehtaab Sawhney, we disprove Costello's original conjecture).
\end{abstract}

\maketitle
\section{Introduction}

\emph{Anticoncentration} inequalities are an important class of probabilistic inequalities, giving upper bounds on the probability that a
random variable falls in a small interval or is equal to a particular
value (i.e., they provide limits on the extent to which a random variable
``concentrates''). Such inequalities are ubiquitous, playing an
essential behind-the-scenes role in a wide range of different types
of probabilistic arguments.

In particular, an important direction is the
(polynomial) \emph{Littlewood--Offord problem}, which concerns anticoncentration
of low-degree polynomials of independent random variables (making
minimal assumptions about the structure of the polynomials in question).
A representative theorem, due to Meka, Nguyen and Vu~\cite{MNV16}, is as follows. (The \emph{Rademacher} distribution
is the uniform distribution on $\{-1,1\}$).
\begin{theorem} \label{thm:polynomial-LO}
    Fix $d\ge 1$ and $\varepsilon>0$,
    and let $\mathbb{F}\in\{\mathbb{R},\mathbb{C}\}$ be the field of
    real or complex numbers. Let $n$ be sufficiently large (in terms
    of $\varepsilon,d$). Let $f\in\mathbb{F}[x_{1},\dots,x_{n}]$ be
    an $n$-variable polynomial of degree $d$. Then, at least
    one of the following holds:
    \begin{enumerate}[{\bfseries{A\arabic{enumi}}}]
    \item\label{A1} $f$ has at most $\varepsilon n^{d}$ nonzero coefficients,
    or
    \item\label{A2} letting $\xi_{1},\dots,\xi_{n}\in\{-1,1\}$ be i.i.d.\ Rademacher
    random variables, we have
    \[
    \sup_{z\in\mathbb{F}}\Pr[f(\xi_{1},\dots,\xi_{n})=z]\le n^{-1/2+\varepsilon}.
    \]
    \end{enumerate}
\end{theorem}

In other words, polynomials of $n$ independent random variables have their point
probabilities bounded by about $1/\sqrt{n}$, unless they are ``close
to the zero polynomial'' (note that we may take $\varepsilon>0$
to be arbitrarily small\footnote{There are more precise quantitative questions in this direction, which we will not pursue in this paper.}, strengthening both \cref{A1} and \cref{A2}, though this forces us to take large $n$). To see
that one cannot hope for a bound stronger than about $1/\sqrt{n}$,
consider the polynomial $f(x_{1},\dots,x_{n})=(x_{1}+\dots+x_{n})^{d}$. Indeed, then $\sup_{z\in\mathbb{F}}\Pr[f(\xi_{1},\dots,\xi_{n})=z]\ge \sup_{z\in\mathbb{F}}\Pr[x_{1}+\dots+x_{n}=z]=\binom{n}{\lfloor n/2\rfloor}\cdot 2^{-n}$ is on the order of $1/\sqrt n$ (up to constant factors).
\begin{remark}
Littlewood--Offord theorems are usually stated with $\xi_{1},\dots,\xi_{n}$
being i.i.d.\ Rademacher (as in \cref{thm:polynomial-LO} above),
but it is often not hard to deduce analogous results where $\xi_{1},\dots,\xi_{n}$
are allowed to have arbitrary distributions, as long as they are independent.
Indeed, in some sense the Rademacher case is the ``hardest case'':
other cases can be reduced to this case with simple coupling/conditioning
arguments (see for example \cite[Section~6]{MNV16}).
\end{remark}

To give some history for \cref{thm:polynomial-LO} (more can be found in the surveys \cite{NV13,Vu17}): the
$d=1$ (linear) case of \cref{thm:polynomial-LO} first appeared as
a lemma in a 1943 paper of Littlewood and Offord~\cite{LO43}\footnote{A more general result was claimed in a 1939 paper by Doeblin~\cite{Doe39},
though this is also the subject of a 1958 paper of Kolmogorov~\cite{Kol58},
which claims that Doeblin's paper did not provide a full proof.} on random algebraic equations. Their bound was improved by Erd\H os~\cite{Erd45} in 1945, obtaining \emph{exact} bounds for the $d=1$ case (the result of this latter paper is now usually called the \emph{Erd\H os--Littlewood--Offord theorem}, and has
been very influential in combinatorics and random matrix
theory). Both these papers were for linear forms with real coefficients;
complex coefficients were first explicitly considered by Kleitman~\cite{Kle70}
(though, in the way we have stated \cref{thm:polynomial-LO}, the real
and complex cases are equivalent).

The higher-degree ($d\ge2$) Littlewood--Offord problem was first
considered in a 1996 paper of Rosi\'nski and Samorodnitsky~\cite{RS96}
(they proved a weaker form of \cref{thm:polynomial-LO} for applications
to L\'evy chaos), and this direction of study rose to the
forefront after a 2006 paper of Costello, Tao and Vu~\cite{CTV06}
(who proved a weak form of the $d=2$ case of \cref{thm:polynomial-LO},
as a key ingredient in their work on random symmetric matrices), and
a 2013 paper of Razborov and Viola \cite{RV13} (who proved a weak form
of \cref{thm:polynomial-LO} as a tool to study correlation bounds
between Boolean functions and polynomials). The $d=2$ case of \cref{thm:polynomial-LO}
was proved by Costello~\cite{Cos13}, and \cref{thm:polynomial-LO} for general $d$ was proved by Meka, Nguyen
and Vu~\cite{MNV16}, using a result of Kane~\cite{Kan14}.
Very recently, an optimal $O(1/\sqrt n)$ bound in \cref{A2} was obtained in the case $d=2$ by the second and third authors~\cite{KS}.

It remains a very interesting open problem to obtain sharp $O(1/\sqrt{n})$
bounds in \cref{A2} for $d\ge3$, but another important question is to study conditions
under which one can prove \emph{much stronger} bounds. A natural conjecture
along these lines was first made by Costello~\cite[Conjecture~3]{Cos13},
as follows (his paper was about polynomials with complex coefficients,
but we have taken the liberty to state his conjecture in both the
real and complex cases\footnote{Unlike \cref{thm:polynomial-LO}, here there is a genuine difference
between the real and complex cases, as there are real polynomials
that factorise over $\mb C$ but not $\mb R$.}).
\begin{conjecture} \label{conj:costello}
    Fix $d\ge 2$ and $\varepsilon>0$,
    and let $\mathbb{F}\in\{\mathbb{R},\mathbb{C}\}$. Let $n$ be sufficiently
    large (in terms of $\varepsilon,d$), and let $f\in\mathbb{F}[x_{1},\dots,x_{n}]$
    be an $n$-variable polynomial of degree $d$. Then, at
    least one of the following holds:
    \begin{enumerate}[{\bfseries{B\arabic{enumi}}}]
        \item \label{B1}there is a reducible\footnote{A polynomial is reducible if it can be factored as the product of two non-constant polynomials. In this paper, we use the convention that the zero polynomial is reducible.} polynomial $g\in\mathbb{F}[x_{1},\dots,x_{n}]$ with degree at most $d$, such that $f-g$ has at most
        $\varepsilon n^{d}$ nonzero coefficients, or
        \item \label{B2}letting $\xi_{1},\dots,\xi_{n}\in\{-1,1\}$ be i.i.d.\ Rademacher
        random variables, we have
        \[
        \sup_{z\in\mathbb{F}}\Pr[f(\xi_{1},\dots,\xi_{n})=z]\le n^{-d/2+\varepsilon}.
        \]
    \end{enumerate}
\end{conjecture}

To paraphrase, Costello conjectured that either the point probabilities
are bounded by about $n^{-d/2}$, or $f$ is close to a product of
two lower-degree polynomials.
To motivate the bound ``$n^{-d/2}$'', note that if the coefficients of $f$ are bounded integers (and at least a constant fraction of them are nonzero), then the variance of $f(\xi_{1},\dots,\xi_{n})$ has order of magnitude $n^d$, so its standard deviation has order of magnitude $n^{d/2}$. It is reasonable to imagine that, ``generically'', the probability mass is roughly evenly spread out on the integers within standard-deviation range of the mean.
The most obvious way for this ``generic'' bound to fail is if $f$ factorises into lower-degree polynomials, since when one of the factors is zero, the whole polynomial is zero. Costello conjectured that that a bound of roughly $n^{-d/2}$ is achieved as long as the polynomial is far from factorising (i.e., if it is ``robustly irreducible'' in the sense of \cref{B1}).

We remark that \cref{conj:costello} falls in the general direction
of the \emph{inverse} Littlewood--Offord problem, which is the study
of the structural aspects of $f$ that control the concentration behaviour
of $f(\xi_{1},\dots,\xi_{n})$ (this problem has a long history, which
we discuss in more detail in \cref{subsubsec:inverse}). Comparisons can also be made
to the \emph{Elekes--R\'onyai--Szab\'o problem} in combinatorial
geometry (see \cite{dZe18} for a survey), and the problem of counting
integral points on varieties (which is intensively studied in analytic
number theory; see e.g.~\cite{Bro09}). Indeed, one can consider the general problem
of counting the number of zeroes of an $n$-variable polynomial $f$
which lie in a ``combinatorial box'' $\prod_{i=1}^{n}A_{i}$. The polynomial Littlewood--Offord problem is concerned with the case where each $A_{i}$ has size 2 and the dimension $n$ is large, whereas the Elekes--R\'onyai--Szab\'o problem is concerned with the case where $n$ is small and each $A_{i}$ is large, and in the study of integral points on varieties one is concerned with the specific case where each $A_{i}$ is a long interval of integers of the form $\{-B,\dots,B\}$. In the latter two settings, the fundamental question is how the algebraic properties of the polynomial $f$ control the asymptotics of the number of zeroes in a large box.

The $d=1$ case of \cref{conj:costello} trivially coincides with the
$d=1$ case of \cref{thm:polynomial-LO}, and we are also inclined
to believe the $d=2$ case (we will justify this in \cref{subsec:further-directions}, by comparison
to related work in analytic number theory), but unfortunately \cref{conj:costello}
is too optimistic for $d\ge 3$, as shown by the following proposition.

\begin{proposition}\label{prop:costello-false}
Fix an integer constant $d\ge 1$. Consider $2d$ disjoint subsets $I_{1},\dots,I_{2d}\subseteq\{1,\dots,n\}$,
each of size exactly $2\lfloor n/(4d)\rfloor$. For $j\in\{1,\dots,2d\}$ let
$L_{j}(x_{1},\dots,x_{n})=\sum_{i\in I_{j}}x_{i}$, and define the
degree-$d$ polynomial
\[
f=L_{1}\dots L_{d}-L_{d+1}\dots L_{2d}\in\mb R[x_{1},\dots,x_{n}].
\]
Then, there is $\varepsilon>0$ (depending only on $d$) such that:
\begin{itemize}
    \item \cref{B1} does not hold, even if we take $\mathbb{F}=\mathbb{C}$ (i.e., the real polynomial $f$ is ``robustly irreducible'', even if we allow factorisations over the complex numbers), and
    \item for i.i.d.\ Rademacher
        random variables $\xi_{1},\dots,\xi_{n}\in\{-1,1\}$, we have $\Pr[f(\xi_{1},\dots,\xi_{n})=0]\ge \varepsilon/n$ (so if $d\ge 3$, then \cref{B2} does not hold either).
\end{itemize}
\end{proposition}

\cref{prop:costello-false} was observed in a conversation between
Ashwin Sah, Mehtaab Sawhney and the second author of this paper.
The proof proceeds via a random sampling argument restricting to a small subset of the variables where the polynomial essentially factorises, with some careful Ramsey-theoretic arguments to handle non-multilinear terms. The details are provided in \cref{sec:counterexample}.
\begin{remark}
\label{rem:multilinear}Costello~\cite{Cos13} also highlighted a
special case of \cref{conj:costello} (appearing as \cite[Conjecture~2]{Cos13})
for $d$-multilinear forms (i.e., degree-$d$ polynomials for which
the variables are partitioned into $d$ different ``types'', and
every monomial contains exactly one variable of each type). Costello
was able to prove this specialised conjecture for bilinear forms (2-multilinear forms),
but \cref{prop:costello-false} shows that when $d=3$ even this
specialised conjecture is false.
\end{remark}
\begin{remark}
\label{rem:maybe}It seems plausible that with similar methods as in the proof of \cref{prop:costello-false} (but more technical details) one can actually prove that for \emph{any} irreducible polynomial $F\in \mb F[X_1,\dots,X_k]$, polynomials of the form $F(L_1,\dots,L_k)$ are robustly irreducible in the sense of \cref{B1} (where we fix disjoint sets $I_1,\dots,I_k\subseteq\{1,\dots,n\}$ of size about $n/k$, and take $L_j(x_1,\dots,x_n)=\sum_{i\in I_j}x_i$ for $j=1,\dots,k$). If the polynomial $F$ has integer coefficients and has many zeroes in  $([-\sqrt{n},\sqrt{n}]\cap \mathbb{Z})^k$, then the point probability $\Pr[F(L_1(\xi_{1},\dots,\xi_{n}),\dots,L_k(\xi_{1},\dots,\xi_{n}))=0]$ is large (see also the discussion in \cref{subsubsec:analytic}). So, this would give a large family of counterexamples to \cref{conj:costello}.
\end{remark}

The purpose of this paper is to open a systematic investigation
into the algebraic aspects of the Littlewood--Offord problem, in
accordance with the philosophy of \cref{conj:costello}. We see there
being two main questions to investigate. First, we believe that if
\cref{B1} does not hold (i.e., if $f$ is ``robustly irreducible''), it should still be possible to at least prove
some ``power-saving'' beyond the general bound in \cref{thm:polynomial-LO}.
\begin{conjecture}
\label{conj:stability}In the setting of \cref{conj:costello}, either
\cref{B1} holds, or else
\begin{enumerate}[{\bfseries{B\arabic{enumi}'}}]
  \setcounter{enumi}{1}
    \item\label{B2'} letting $\xi_{1},\dots,\xi_{n}\in\{-1,1\}$ be i.i.d.\ Rademacher
    random variables, we have
    \[
    \sup_{z\in\mathbb{F}}\Pr[f(\xi_{1},\dots,\xi_{n})=z]\le n^{-c+\varepsilon}.
    \]
    for some $c>1/2$ (which does not depend on $\varepsilon$, $n$
    or $f$).
\end{enumerate}
\end{conjecture}

It seems plausible that in fact \cref{conj:stability} holds with $c=1$
(this is the best one could hope for, given \cref{prop:costello-false}, and in the $d=2$ case it would be equivalent to \cref{conj:costello}).
However, this is probably a very difficult problem: as we will discuss in \cref{subsubsec:analytic}, the general $c=1$ case of \cref{conj:stability} seems to be at least as hard as the so-called \emph{affine dimension growth conjecture} in analytic number theory.

Second, it would be of interest to identify conditions under which
\cref{B2} can be salvaged.
\begin{question}
\label{qu:dream}In the setting of \cref{conj:costello}, what natural
assumptions on $f$ are sufficient to guarantee that \cref{B2} holds?
\end{question}

As our first main result, we prove an optimal version of \cref{conj:stability}
for $d$-multilinear forms, which can be viewed as a ``repaired''
version of \cite[Conjecture~2]{Cos13} (recalling the discussion in
\cref{rem:multilinear}).
\begin{theorem}
\label{thm:multilinear}
    In the setting of \cref{conj:costello,conj:stability}, if $f$ is a $d$-multilinear form, then either \cref{B1} or \cref{B2'} holds with $c=1$.
\end{theorem}

As our next main result, we prove \cref{conj:stability} in the complex
quadratic case ($d=2$ and $\mathbb{F}=\mathbb{C}$).

\begin{theorem}
\label{thm:stability}When $d=2$ and $\mathbb{F}=\mb C$, \cref{conj:stability}
holds with $c=13/24$.
\end{theorem}

As our final main result, we show that in the quadratic case ($d=2$),
one can approach the ``generic bound'' $1/n=n^{-d/2}$ with an assumption
that the quadratic part of $f$ has high rank (this provides a partial
answer to \cref{qu:dream}).
\begin{theorem} \label{thm:rank}
    Fix any $\varepsilon>0$ and $k\ge 1$, and let $\mathbb F\in \{\mb R,\mb C\}$. Let $n$ be sufficiently
    large (in terms of $\varepsilon,k$), and let $f\in\mathbb{F}[x_{1},\dots,x_{n}]$
    be a quadratic $n$-variable polynomial. Then, at
    least one of the following holds:
    \begin{enumerate}[{\bfseries{C\arabic{enumi}}}]
        \item \label{C1} there is a quadratic form $g\in\mathbb{F}[x_{1},\dots,x_{n}]$ with rank\footnote{Recall that a \emph{quadratic form} is a homogeneous quadratic polynomial $h\in \mb F[x_1,\dots,x_n]$.  Any quadratic form can be expressed in matrix form as $h(x)=\vec x^T Q \vec x$ for some symmetric matrix $Q\in \mb F^{n\times n}$. The \emph{rank} of $h$ is defined to be the rank of this matrix $Q$. Equivalently, the rank of $h$ is the minimum $r$ such that there is a representation $h=\lambda_1 h_1^2+\dots+\lambda_r h_r^2$ as a linear combination of squares of homogeneous linear polynomials $h_1,\dots,h_r\in \mb F[x_1,\dots,x_n]$ with coefficients $\lambda_1,\dots,\lambda_r\in \mb F$.} less than $2k^2$, such that $f-g$ has at most $\varepsilon n^{2}$ nonzero coefficients, or 
        \item \label{C2} letting $\xi_{1},\dots,\xi_{n}\in\{-1,1\}$ be i.i.d.\ Rademacher
        random variables, we have
        \[
            \sup_{z\in\mathbb{F}}\Pr[f(\xi_{1},\dots,\xi_{n})=z]\le n^{-1+2/k}.
        \]
    \end{enumerate}
    
\end{theorem}
In other words, for any $\alpha>0$, we can prove a bound of the form $n^{-1+\alpha}$ (for large $n$), as long as $f$ is not close to a quadratic form with rank less than about $\alpha^{-2}$.
\begin{remark}
It is worth remarking that \cref{thm:rank} is reminiscent of a theorem of Gowers and Karam~\cite{GK22}, which (rephrased in the language of this paper) implies that for a fixed prime $p$, any degree-$d$ polynomial $f\in \mb Z[x_1,\dots,x_n]$, and i.i.d. Rademacher $\xi_1,\dots,\xi_n$, if $f(\xi_1,\dots,\xi_n)$ has ``a significant bias mod $p$'', then in a certain sense it is possible to approximate $f$ by a degree-$d$ polynomial of ``low rank''. This can be interpreted as a ``restricted-domain'' version of a similar theorem of Green and Tao~\cite{prank5} (qualitatively improved by Janzer~\cite{prank7} and Mili\'cevi\'c~\cite{prank9}), which roughly speaking considers the case where $\xi_1,\dots,\xi_n$ are uniform mod $p$. However, we do not see a formal connection between \cref{thm:rank} and this direction of study.
\end{remark}

\begin{remark}
    The second and third author~\cite{KS20} previously proved versions of \cref{thm:stability,thm:rank} in the setting where $f$ has bounded coefficients, and instead of measuring the ``smallness'' of a polynomial by counting nonzero coefficients, one considers the sum of the absolute values of the coefficients. This setting makes the problem very different in character (it makes the problem much more analytic than algebraic), and the proofs in \cite{KS20} have almost nothing in common with the proofs of \cref{thm:stability,thm:rank}.
\end{remark}

We outline the ideas in the proofs of \cref{thm:multilinear,thm:rank,thm:stability} in \cref{sec:outline}. The proofs contain a number of ideas
of general interest, including some lemmas related to \emph{property
testing} of low-rank tensors and symmetric matrices.

\subsection{Further directions, and speculations\label{subsec:further-directions}}It seems that for further progress on \cref{conj:stability,qu:dream} (and the $d=2$ case of \cref{conj:costello}), it might be necessary to integrate the ``combinatorial'' ideas in this paper with ideas from analytic number theory. Here, we discuss the connection to analytic number theory and also suggest some other directions for further research.

\subsubsection{Connections to analytic number theory}\label{subsubsec:analytic}
\cref{prop:costello-false} (our counterexample to Costello's conjecture) is really a special case of a general observation
that if $P\in\mathbb{F}[y_{1},\dots,y_{k}]$ is any polynomial with $N_P(B)$ zeroes in $\{-B,\dots,B\}^k$, then one can construct a polynomial $f_P\in\mathbb{F}[x_{1},\dots,x_{n}]$
of the form $f_P=P(L_{1},\dots,L_{k})$ for linear forms $L_{1},\dots,L_{k}\in \mathbb{F}[x_{1},\dots,x_{n}]$ such that $\Pr[f_P(\xi_1,\dots,\xi_n)=0]$ scales roughly like $N_P(\sqrt n) n^{-k/2}$. So, in order to fully understand the algebraic aspects of the
polynomial Littlewood--Offord problem, it seems to be necessary to understand how the algebraic features of a polynomial $P$ affect the rate of growth of $N_P(B)$ (which can be interpreted as a certain ``density of integral points'' on the variety defined by $P$). To this end, one is forced to confront
deep questions in analytic number theory.

Specifically, if we restrict our attention to polynomials of the form $f_P$, then the bound of shape approximately $1/\sqrt n$ in \cref{thm:polynomial-LO} corresponds to the ``trivial bound'' $N_P(B)=O(B^{k-1})$, which holds for any nonzero polynomial $P\in\mathbb{F}[y_{1},\dots,y_{k}]$ of fixed degree (see for example \cite[Equation~(2.3)]{BH05}). It seems that in order to prove \cref{conj:stability} for polynomials of the form $f_P$, one would need a power-saving improvement over this ``trivial bound'' when $P$ is irreducible. 
Such estimates are indeed available: for an irreducible $d$-form $P$, Cohen~\cite{Coh81} proved that $N_P(B)\le B^{k-3/2+o(1)}$, and this was later improved by Pila~\cite{Pil95} to $N_P(B)\le B^{k-2+1/d+o(1)}$ (see also \cite{Wal15}). It is conjectured that $N_P(B)\le B^{k-2+o(1)}$ for all $d\ge 2$ (which would be best-possible, if true; this is an affine version of the so-called \emph{dimension growth conjecture} of Heath-Brown~\cite{HB02} for projective varieties). The affine dimension growth conjecture was very recently proved when $d\ne 3$ by Vermeulen~\cite{Ver24}, building on earlier work of Browning, Heath-Brown and Salberger~\cite{BHS06} and Salberger~\cite{Sal23}.
The methods in all these works are quite different, and it would be very interesting to investigate if any of them can be extended to a full proof of \cref{conj:stability} (not just for polynomials of the form $f_P$).

Regarding \cref{qu:dream}, if $f_P$ has ``generic'' anticoncentration $n^{-d/2+o(1)}$, this corresponds to a bound of the form $N_P(B)\le B^{k-d+o(1)}$. This is sometimes called the ``probabilistic heuristic'' in analytic number theory. It is a very deep open problem to characterise the polynomials $P$ for which such a bound holds, but one interesting sufficient condition is that $P$ is a $d$-form with sufficiently large \emph{Schmidt rank} (this parameter is also sometimes called \emph{h-invariant} or \emph{strength}), meaning that $P$ cannot be written as a sum of few products of forms of lower degree~\cite{Sch85}. We suspect that a version of \cref{conj:costello} along these lines should also hold, as follows.
\begin{conjecture}\label{conj:rank}
For any $d$, there is $k\in \mb N$ such that the following holds. In the setting of \cref{conj:costello}, either \cref{B2} holds or 
\begin{itemize}
\item [\textbf{B1'}]For some $r\le k$ there are homogeneous polynomials $g_1,h_1,g_2,h_2,\dots,g_r,h_r\in \mb F[x_1,\dots,x_n]$, each with positive degree, such that $g_1h_1+g_2h_2+\dots+g_rh_r$ has degree at most $d$, and $f-(g_1h_1+g_2h_2+\dots+g_rh_r)$ has at most $\varepsilon n^d$ nonzero coefficients.
\end{itemize}
\end{conjecture}
Note that \cref{thm:rank} almost, but not quite, resolves the $d=2$ case of \cref{conj:rank}. The difference is that in \cref{thm:rank}, as one wishes to get closer and closer to the generic bound $n^{-d/2+o(1)}$, the rank requirement continues to increase, whereas in \cref{conj:rank} there is a particular rank above which one immediately obtains the generic bound $n^{-d/2+o(1)}$.

In the case where $f$ is a $d$-multilinear form, the Schmidt rank of $f$ coincides with the so-called \emph{partition rank} of the tensor of coefficients of $f$. This notion has recently been intensively studied in theoretical computer science and additive combinatorics (see for example \cite{prank10,prank7,prank16,prank2,prank9,prank12,prank15,prank1}), and it would likely be much more tractable (and still very interesting) to prove the special case of \cref{conj:rank} where $f$ is a $d$-multilinear form.
\subsubsection{Further questions}
There are a huge number of other questions related to the results and conjectures in this paper. For example:
\begin{enumerate}[{\bfseries{Q\arabic{enumi}}}]
    \item Could an even stronger version of \cref{conj:stability} hold, where in \cref{B1} one demands that $g$ is divisible by a polynomial of degree 1 (i.e., this degree-1 polynomial is ``directly responsible'' for the poor anticoncentration behaviour)?
    \item\label{item:smallness} Can one consider weaker notions of ``smallness'' of a polynomial than having few nonzero coefficients? In \cite{KS}, the second and third authors show that in the $d=2$ case of \cref{thm:polynomial-LO} one can replace property \cref{A1} (that $f$ has at most $\varepsilon n^2$ nonzero coefficients) with the property that one can fix at most $\varepsilon n$ of the variables to obtain a constant polynomial. It would be very interesting if this notion of ``smallness'' could also be used for theorems along the lines of \cref{thm:multilinear,thm:rank,thm:stability}.
    \item Can one quantify the ``$n^{\varepsilon}$'' error term in \cref{thm:multilinear}? We note that it is not possible to remove $n^{o(1)}$ terms entirely: for example, if we consider the polynomial $f=L_1L_2-L_3 L_4$ in the $d=2$ case of \cref{prop:costello-false}, known results on the ``multiplication table problem'' (counting the integers respresentable in the form $a\cdot b$, where $a,b\in\{1,\dots,N\}$; see \cite{Ford08}) show that
    \[
\sup_{z\in\mathbb{F}}\Pr[f(\xi_{1},\dots,\xi_{n})=z]\ge  \frac{(\log n)^{\alpha+o(1)}}n.
\]%
for some absolute constant $\alpha>0$.
    \item What can we prove in the ``higher-dimensional'' case, where we are interested in the probability that a length-$k$ vector $(f_1(\xi_1,\dots,\xi_n),\dots,f_k(\xi_1,\dots,\xi_n))$ of polynomials of independent random variables is equal to a particular vector $z\in \mb F^k$? In the linear ($d=1$) case, essentially optimal results were famously obtained by Hal\'asz~\cite{Hal77} (see also the results in \cite{HO,FJZ22}, which consider the dependence on $k$), but almost nothing is known for higher degrees.
\end{enumerate}

\subsubsection{Inverse theorems}\label{subsubsec:inverse} From a broader point of view, all the theorems and conjectures in this paper can be interpreted as falling under the umbrella of the \emph{inverse Littlewood--Offord problem}, which is concerned with understanding the structural properties of $f$ that control the anticoncentration behaviour of $f(\xi_1,\dots,\xi_n)$. In this paper we have considered only the \emph{algebraic} aspects of this problem; aspirationally it would be of interest to consider the interplay between algebraic and \emph{arithmetic} aspects.

To give some context: in the linear ($d=1$) case, there is no relevant algebraic structure; the anticoncentration behaviour of $f(\xi_1,\dots,\xi_n)$ only depends on the multiset $\{a_1,\dots,a_n\}$ of (degree-1) coefficients of $f$. It is now understood that, in fact, the anticoncentration behaviour of $f(\xi_1,\dots,\xi_n)= a_1\xi_1+\dots+a_n\xi_n $ is essentially determined by ``how much $\{a_1,\dots,a_n\}$ looks like a generalised arithmetic progression''. A particular theorem along these lines, due to Nguyen and Vu~\cite{NV11}, appears as \cref{theorem: optimal inverse theorem} later in the paper; earlier theorems in this spirit were proved by Hal\'asz~\cite{Hal77}, Tao and Vu~\cite{TV09,TV10} and Rudelson and Vershinyn~\cite{RV08}.

For higher degrees, the only result in the literature considering both arithmetic and algebraic aspects of the Littlewood--Offord problem is a theorem of Nguyen~\cite{Ngu12} in the quadratic case ($d=2$). This theorem had important consequences in random matrix theory, but unfortunately it is very crude, and we are still a long way from ``optimal'' inverse theorems for higher degrees. In particular, Nguyen's theorem is only capable of distinguishing ``polynomial'' anticoncentration from ``sub-polynomial'' anticoncentration (i.e., distinguishing whether there is a point probability at least $n^{-C}$, for some constant $C$, or not). 
For bilinear forms (a subfamily of quadratic polynomials), Costello made a more refined conjecture~\cite[Conjecture~1]{Cos13} but we do not know of a general conjecture for quadratic polynomials.

\subsubsection{Property testing for (tensor) rank} As we will discuss further in the proof outlines in \cref{sec:outline}, the proofs of \cref{thm:multilinear,thm:rank,thm:stability} depend on two ``local-to-global'' lemmas which allow one to study the rank of a matrix or tensor via the ranks of small submatrices or small subtensors. First, \cref{thm:tensor-property-testing} says that if a $d$-dimensional $n\times \dots\times n$ tensor $T$ has the property that a $1-o(1)$ fraction of its $2^{d-1}\times \dots\times 2^{d-1}$ subtensors have partition rank 1, then we can change a $o(1)$-fraction of the entries of $T$ to obtain a tensor with partition rank 1. Second, \cref{lemma: close to symmetric low rank} says that, if an $n\times n$ symmetric matrix $A$ has the property that a $1-o(1)$ fraction of its $r\times r$ submatrices are singular, then we can change a $o(1)$-fraction of the entries of $A$ to obtain a symmetric matrix with rank less than $r$. 

\cref{thm:tensor-property-testing,lemma: close to symmetric low rank} are of independent interest, as they permit efficient \emph{property testing} for matrices and tensors. For example, suppose we have a $d$-dimensional tensor $T$ of enormous size $n\times \dots\times n$, and we want to distinguish between the possibilities that $T$ has partition rank 1 and that $T$ is ``$\varepsilon$-far'' from having partition rank 1. With high probability, we can accomplish this in \emph{constant} time (i.e., with a number of queries that only depends on $\varepsilon$, $d$, and the desired success probability), by randomly sampling many subtensors with side length $2^{d-1}$ and checking\footnote{While computing the partition rank of a tensor is difficult in general, it is easy to check whether a tensor has partition rank $1$ (this amounts to checking whether some matrix of a certain form has rank $1$, see \cref{fact:tensor-to-matrix}).} whether they have partition rank 1. There is an enormous literature on property testing of graphs, Boolean functions, probability distributions, etc.\ (see for example the monographs \cite{Gol17,BY22,Ron07} and the references therein), and in particular there has been some specific attention on property testing for the rank of (not necessarily symmetric) matrices; see for example \cite{BLWZ19,KS03,LWW14}. However, there does not seem to have been any prior work on property testing for any kind of tensor rank, or for matrices constrained to be symmetric. We believe there is quite some scope to investigate these directions further.

For example, while \cref{lemma: close to symmetric low rank} already has near-optimal quantitative aspects, it would be interesting to investigate quantitative aspects for \cref{thm:tensor-property-testing}. It would also be interesting to investigate the possibility of more efficient property testing algorithms that do not simply randomly sample small submatrices/subtensors (for not-necessarily-symmetric matrices, there are rank-testing algorithms that do strictly better than sampling submatrices; see \cite{LWW14}).

It would also be very interesting to consider more general property testing for the partition rank (or other notions of tensor rank), beyond the setting of \cref{sec:tensor-property-testing} (where we are only interested in whether the partition rank is 1 or not). Tensor ranks are surprisingly poorly behaved with respect to subtensors (see e.g.\ the surprising construction of Gowers in \cite[Proposition~3.1]{Kar22}), and a straightforward generalisation of \cref{thm:tensor-property-testing} to higher ranks does not actually seem to be possible.

\subsection{Organisation of the paper} In \cref{sec:outline} we provide outlines of the proofs of \cref{thm:multilinear,thm:rank,thm:stability}, including statements of various key lemmas. Then, in \cref{sec:prelim}, we record some basic preliminary lemmas that will be used throughout the proofs.

The details of the proof of \cref{thm:multilinear} appear in \cref{sec:reducible-variety,sec:tensor-property-testing,sec:multilinear}. First, in \cref{sec:tensor-property-testing} we prove a key lemma allowing us to restrict our attention to small subtensors, and in \cref{sec:reducible-variety} we characterise the variety of reducible multilinear forms (i.e., reducible tensors). In \cref{sec:multilinear} we combine all the relevant ingredients to prove \cref{thm:multilinear}.

Most of the rest of the paper is concerned with the proofs of \cref{thm:rank,thm:stability}. In \cref{sec:matrix-property-testing} we prove a key lemma allowing us to restrict our attention to small submatrices, and in \cref{sec:decoupling} we prove some special-purpose decoupling inequalities that will allow us to efficiently use tools from linear Littlewood--Offord theory to study quadratic anticoncentration. Then, the proof of \cref{thm:rank} is completed in \cref{sec:rank} and the proof of \cref{thm:stability} is completed in \cref{sec:power-saving}.

Finally, we have two appendices: in \cref{sec:FKS} we show how to prove a certain variant of a ``geometric Littlewood--Offord theorem'' of Fox, Spink and the second author~\cite{FKS23}, which is an ingredient in the proof of \cref{thm:multilinear}. Then, \cref{sec:counterexample} (authored by Matthew Kwan, Ashwin Sah and Mehtaab Sawhney) contains the proof of \cref{prop:costello-false} (i.e., the disproof of Costello's conjecture).

\subsection{Notation}\label{subsec:notation} We use standard asymptotic notation throughout, as follows. For functions $f=f(n)$ and $g=g(n)$, we write $f=O(g)$ to mean that there is a constant $C$ such that $|f(n)|\le C|g(n)|$, and we write $f=\Omega(g)$ to mean that there is a constant $c>0$ such that $f(n)\ge c|g(n)|$ for sufficiently large $n$.
We also write $f=o(g)$ to mean that $f(n)/g(n)\to0$ as $n\to\infty$. Subscripts on asymptotic notation indicate quantities that should be treated as constants.

Slightly less standardly, for a matrix $A\in \mb F^{X\times Y}$ and subsets $I\subseteq X,J\subseteq Y$, we write $A[I,J]\in \mb F^{I\times J}$ for the $|I|\times |J|$ submatrix of $A$ consisting of the rows with indices in $I$ and the columns with indices in $J$. In the case that $I=\{i\}$ for some $i\in X$, we slightly abuse notation and write $A[i,J]$ instead of $A[\{i\},J]$. Similarly, for a vector $\vec v\in \mb F^X$ and a subset $I\subseteq X$ we write $\vec v[I]\in \mb F^{I}$ for the vector obtained from $\vec v\in \mb F^I$ by only taking the coordinates with indices in $I$.
In addition, for a matrix $A \in \mb F^{X\times Y}$, we write $\|A\|_0$ for its number of nonzero entries.

We say that a partition $X=I\cup J$ of a set $X$ into two subsets $I$ and $J$ is ``non-trivial'', if both of the subsets $I$ and $J$ are non-empty. For a real number $x$, the floor and ceiling functions are denoted $\lfloor x\rfloor=\max\{i\in \mb Z:i\le x\}$ and $\lceil x\rceil =\min\{i\in\mb Z:i\ge x\}$. We will however sometimes omit floor and ceiling symbols and assume large numbers are integers, wherever divisibility considerations are not important. Finally, some conventions: all logarithms in this paper are in base $e$, unless specified otherwise, and the natural numbers $\mb N$ include zero.

\subsection*{Acknowledgments} We would like to thank Tim Browning and Sarah Peluse for informing us about relevant references in the analytic number theory literature.

\section{Proof outlines\label{sec:outline}}

In this section, we outline the proofs of \cref{thm:multilinear,thm:rank,thm:stability}. The proofs of \cref{thm:rank,thm:stability}
are quite closely related, but the proof strategy for \cref{thm:multilinear}
is rather different.

\subsection{$d$-multilinear forms}\label{subsec:multilinear}
First, we outline the proof of \cref{thm:multilinear}.
Recall that a polynomial $f\in \mb F[x_1,\dots,x_n]$ is a \emph{$d$-multilinear form} if there is a partition $\{1,\dots,n\}=I_1\cup \dots\cup I_d$ such that every term of $f$ is of the form $c_{i_1,\dots,i_d} x_{i_1}\dots x_{i_d}$, where $c_{i_1,\dots,i_d}\in \mb F$ and $i_j\in I_j$ for each $j\in \{1,\dots,d\}$. Such a $d$-multilinear form can be equivalently viewed as a function $\mb F^{I_1}\times \dots \times \mb F^{I_d}\to \mb F$ which is linear in each of its $d$ arguments. It can be encoded by the $|I_1|\times \dots\times |I_d|$ array of coefficients $c_{i_1,\dots,i_d}$, and it can also be viewed as an element of the vector space $\mb F^{I_1}\otimes \dots\otimes \mb F^{I_d}$. This gives rise to three different ways to define the notion of an $|I_1|\times \dots\times |I_d|$ \emph{tensor}:
\begin{itemize}
    \item as a $d$-multilinear form, which can be interpreted as a polynomial in $\mb F[x_1,\dots,x_n]$ or a function $\mb F^{I_1}\times \dots \times \mb F^{I_d}\to \mb F$,
    \item as a function $I_1\times \dots\times I_d\to \mb F$ (i.e., a $I_1\times \dots\times I_d$ array with entries in $\mb F$), or
    \item as an element of the vector space $\mb F^{I_1}\otimes \dots\otimes \mb F^{I_d}$.
\end{itemize}
We will switch between the these three points of view as convenient. Indeed, viewing a tensor as a function $I_1\times \dots\times I_d\to \mb F$ makes it easier to talk about subtensors (i.e., tensors induced by subsets $I_1'\subseteq I_1,\dots,I_d\subseteq I_d'$), while viewing a tensor as an element of $\mb F^{I_1}\otimes \dots\otimes \mb F^{I_d}$ makes it easier to talk about factorisation.
\begin{definition} \label{def: reducible tensors}
    A tensor $T\in \mb F^{I_1}\otimes \dots\otimes \mb F^{I_d}$ is \emph{reducible} or has \emph{partition rank at most 1} if there is a non-trivial partition $\{1,\dots,d\}=J_1\cup J_2$, and tensors $T_1\in \bigotimes_{j\in J_1} \mb F^{I_j}$ and $T_2\in \bigotimes_{j\in J_2}\mb F^{I_j}$, such that $T$ can be factored as a tensor product $T=T_1\otimes T_2$.
    Equivalently, this says that the corresponding $d$-multilinear form can be factorised into lower-degree factors.
\end{definition}

The first ingredient we need is the following theorem, which states that for any tensor $T$, if most of the small subtensors of $T$ are reducible, then $T$ itself is close to being reducible. In our proof, we will use this as a sort of local-to-global principle: we can study small subtensors of $T$ individually, and make conclusions about the whole tensor $T$.

\begin{lemma}\label{thm:tensor-property-testing}
    Fix $d>1$ and $\varepsilon>0$, and any field $\mb F$. 
    Let $\delta=(\varepsilon/2)^{2^{d-1}}$, and let $n_0$ be sufficiently large in terms of $d$ and $\varepsilon$. 
    Consider a $d$-dimensional tensor $T:I_1\times \dots\times I_d\to \mb F$, where $|I_1|,\dots,|I_d|\ge n_0$, and suppose that all but at most a $\delta$-fraction of the $2^{d-1}\times \dots\times 2^{d-1}$ subtensors of $T$ are reducible. Then one can make $T$ reducible by changing up to an $\varepsilon$-fraction of its entries.
\end{lemma}
\begin{remark}
The proof of \cref{thm:tensor-property-testing} is also suitable for other notions of ``tensor rank at most 1'', most notably the so-called \emph{slice rank} (see \cite{srank1,srank2}).
\end{remark}

The (short) proof of \cref{thm:tensor-property-testing} appears in \cref{sec:tensor-property-testing}. We will prove \cref{thm:multilinear} by induction on $d$, using Costello's result for $d=2$ as the base case. Given \cref{thm:tensor-property-testing}, the main ingredient for the induction step is the following lemma.

\begin{lemma}\label{lem:collapse}
    Fix $d\ge 3$ and $r \ge 1$ and $\varepsilon>0$ and $\mb F\in \{\mb R,\mb C\}$.
    Let $n$ be sufficiently large in terms of $d$, $r$ and $\varepsilon$.
    Consider a tensor $T:I_1\times \dots \times I_{d}\to \mb F$, where $|I_1|=\dots=|I_{d-1}|=r$ and $I_d=\{1,\dots,n\}$. 
    For a vector $\vec x\in \mb F^n$, define the $(d-1)$-dimensional tensor $T\vec x:I_1\times\dots\times I_{d-1}\to \mb F$ by
    \[T\vec x(i_1,\dots,i_{d-1})=\sum_{i=1}^n T(i_1,\dots,i_{d-1},i)x_i.\]
    Then at least one of the following holds.
    \begin{enumerate}[{\bfseries{D\arabic{enumi}}}]
        \item\label{D1} All but an $\varepsilon$-fraction of the $r\times \dots\times r$ subtensors of $T$ are reducible, or
        \item\label{D2} letting $\vec \xi=(\xi_{1},\dots,\xi_{n})\in\{-1,1\}^n$ be a vector of i.i.d.\ Rademacher
        random variables, we have
        \[
            \Pr[T\vec \xi\text{ is reducible}]\le n^{-1/2+\varepsilon}.
        \]
    \end{enumerate}
\end{lemma}

Roughly speaking, we will use \cref{lem:collapse} as follows. Given a $d$-multilinear form $f$ (with variables partitioned into $d$ parts), we wish to understand the behaviour of $f$ when the variables take random values. In the induction step, we reveal all the randomness in only the last of the $d$ parts, which allows us to reinterpret $f$ as a $(d-1)$-multilinear form in the remaining unrevealed variables.
One way to visualise this, from the point of view that the tensor of coefficients is a $d$-dimensional array, is to view our $d$-dimensional tensor as a ``stack'' of $(d-1)$-dimensional tensors. We assign a random sign to each of the $(d-1)$-dimensional tensors in the stack, and add them together to obtain the ``collapsed'' $(d-1)$-dimensional tensor. For the induction to work, we need to know that if the original $d$-multilinear form was far from being reducible, then the ``collapsed'' $(d-1)$-multilinear form is likely to also be far from being reducible. By \cref{thm:tensor-property-testing}, it suffices to study this ``collapse'' on tensors whose side-lengths are $2^{d-1}$ (except in the $d$-th dimension), and this is precisely the role of \cref{lem:collapse}.

In the statement of \cref{lem:collapse}, note that $T\vec \xi$ can be viewed as a linear combination of Rademacher random variables, with coefficients in the space of $|I_{d-1}|\times \dots\times |I_{d-1}|$ tensors $\mb F^{I_1}\otimes\dots\otimes \mb F^{I_{d-1}}$.
In \cref{D2} we are interested in the event that this random linear combination falls in the variety of reducible tensors. So, to prove \cref{lem:collapse} we will use the following ``geometric Littlewood--Offord'' theorem (closely related to results of Fox, Spink and the second author~\cite[Theorem~1.5]{FKS23}), providing bounds on the probability that a Rademacher sum falls in a given variety.
\begin{theorem}\label{thm:FKS}
    Fix $d\ge 1$ and $\varepsilon>0$ and $\mb F\in \{\mb R,\mb C\}$ and let $\mathcal Z\subsetneq  \mb F^d$ be a (possibly reducible) affine variety. Let $n$ be sufficiently large (in terms of $\mc Z,\varepsilon$), and consider vectors $\vec a_1,\dots,\vec a_n\in \mb F^d$.
    Then, at least one of the following holds.
    \begin{enumerate}[{\bfseries{E\arabic{enumi}}}]
        \item\label{E1} There is a linear subspace $\mc W\subseteq \mb F^d$ containing all but at most $\varepsilon n$ of the vectors $\vec a_i$, such that $\vec w+\mc W\subseteq \mc Z$ for some vector $\vec w\in \mb F^d$.
        \item\label{E2} letting $\xi_{1},\dots,\xi_{n}\in\{-1,1\}$ be i.i.d.\ Rademacher
random variables, we have
\[
\Pr[\xi_1\vec a_1+\dots+\xi_n\vec a_n\in \mc Z]\le n^{-1/2+\varepsilon}.
\]
\end{enumerate}
\end{theorem}

This result follows from the same proof as in \cite{FKS23} (which considers a more general setting). For completeness, we include a proof of \cref{thm:FKS} in \cref{sec:FKS}, where we deduce \cref{thm:FKS} from the Meka--Nguyen--Vu bound for the polynomial Littlewood--Offord problem (see \cref{thm:polynomial-LO}).

In order to actually apply \cref{thm:FKS}, our final ingredient is a description of the maximal linear subspaces of the variety of reducible tensors.
\begin{lemma}\label{thm:irreducible-variety}
Fix disjoint sets $I_1,\dots,I_d$, fix $\mb F\in \{\mb R,\mb C\}$ and let $r=\prod_{j=1}^d|I_j|$. Let $\mc Z\subseteq \mb F^{I_1}\otimes\dots\otimes \mb F^{I_d}$ be the set of reducible  $I_1\times \dots\times I_d$ tensors. Note that $\mb F^{I_1}\otimes\dots\otimes \mb F^{I_d}\cong \mb F^{r}$, so we can view $\mc Z$ as a subset of $\mb F^r$. Then:
\begin{enumerate}
    \item $\mc Z\subseteq \mb F^{r}$ is a (possibly reducible) affine variety.
    \item The maximal linear subspaces of $\mc Z$ are all sets of tensors of the following form: fix a proper subset $\emptyset\subsetneq J\subsetneq \{1,\dots,d\}$, fix a $|J|$-dimensional tensor $T^\star\in \bigotimes_{j\in J} \mathbb F^{I_j}$, and consider all tensors of the form $T^\star\otimes T'$, where $T'$ ranges over all tensors $T'\in \bigotimes_{j\notin J} \mathbb F^{I_j}$.
\end{enumerate}
\end{lemma}
We prove \cref{thm:irreducible-variety} in \cref{sec:reducible-variety}.

In \cref{sec:multilinear}, we give the details of how to prove \cref{lem:collapse} using \cref{thm:irreducible-variety,thm:FKS}, and deduce \cref{thm:multilinear} from \cref{lem:collapse,thm:tensor-property-testing}.

\subsection{Quadratic polynomials}
In this subsection we outline the proofs of \cref{thm:rank,thm:stability}. We start with some elements common to both proofs, and then split into separate subsections for each of \cref{thm:rank,thm:stability}.

First, recalling the role that \cref{thm:tensor-property-testing} played for the proof of \cref{thm:multilinear}, we need a ``local-to-global'' lemma for the rank of a symmetric matrix (while $d$-multilinear forms are naturally encoded by $d$-dimensional tensors, quadratic forms are naturally encoded by symmetric matrices). Recall from \cref{subsec:notation} that $\|A\|_0$ denotes the number of nonzero entries in a matrix $A$.

\begin{lemma}\label{lemma: close to symmetric low rank}
    Let $n \ge r \ge 1$, let $\alpha \in [0,1]$ and let $\mb{F}$ be a field.
    Consider a symmetric matrix $A \in \mb{F}^{n \times n}$.
    If all but at most an $\alpha$-fraction of the $r\times r$ submatrices of $A$ are singular, then there exists a symmetric matrix $A^\prime \in \mb{F}^{n \times n}$ of rank less than $r$ such that $\|A-A^\prime\|_0 \le O(r^4 \alpha^{1/r} n^2)$.
\end{lemma}
We remark that \cref{lemma: close to symmetric low rank} is very easy to prove if we drop the requirement that the approximating matrix $A'$ is symmetric. The challenge is that we want to modify $A$ to have rank less than $r$ \emph{while maintaining symmetry}. The details of the proof of \cref{lemma: close to symmetric low rank} appear in \cref{sec:matrix-property-testing}, but to give a brief idea: first, we show that we can adjust a small number of entries of $A$ to obtain a matrix $B$ which has low rank (but is not necessarily symmetric). Then, we identify a principal submatrix of $B$ which is symmetric and ``robustly has full rank'', and use this as a ``guide'' to construct $A'$.
\begin{remark}
Throughout the paper we do not pay much attention to quantitative aspects, but it seems worth noting that in \cref{lemma: close to symmetric low rank} the dependence on $\alpha$ is best-possible. Indeed, for any fixed integer $r\ge 1$, let $\ell=\lfloor \alpha^{1/2r} n/2 \rfloor$ and consider a binary matrix $A=(a_{ij})\in \mb R^{n\times n}$ defined by
\[a_{ij}=\begin{cases}1&\text{if }j\ge n-\ell+i\text{ or }i\ge n-\ell+j\\
0&\text{otherwise}\end{cases}\]
(this can be informally described as an all-zero matrix, in which an upper-triangular chunk of the top-right corner, and a lower-triangular chunk of the bottom-left corner, have been changed to ``1''). 
It is not hard to see that all but at most an $\alpha$-fraction of the $r\times r$ submatrices are singular, but to make the rank less than $r$ it is necessary to change an $\Omega_r(\alpha^{1/r})$-fraction of the entries.
\end{remark}

Next, one very important aspect of the proofs of \cref{thm:rank,thm:stability}, which has no counterpart in the proof of \cref{thm:multilinear}, is \emph{decoupling}. Namely, in the last subsection (on the proof of \cref{thm:multilinear}), we crucially used that $d$-multilinear forms can be inductively broken down into \emph{linear} forms in disjoint sets of variables. For example, a bilinear form in the variables $x_1,\dots,x_m$ and $y_1,\dots,y_n$ can be viewed as a linear form in $x_1,\dots,x_m$ whose coefficients are themselves linear forms in $y_1,\dots,y_n$. So, one can study anticoncentration of bilinear forms via tools for \emph{linear} anticoncentration.

In \cref{thm:rank,thm:stability}, we are interested in general (not necessarily bilinear) quadratic polynomials, where this type of recursive linear structure is not available. However, one of the most important techniques for the polynomial Littlewood--Offord problem, called \emph{decoupling}, nonetheless allows one to deduce information about general quadratic polynomials via tools for linear anticoncentration\footnote{The term ``decoupling'' refers more generally to a class of techniques used to reduce from dependent situations to independent situations, see for example the book-length treatment in \cite{dlPG12}. In the context of the polynomial Littlewood--Offord problem, this technique was introduced by Rosi\'nski and Samorodnitsky~\cite{RS96} and Costello, Tao and Vu~\cite{CTV06}.}.

We will need some non-standard decoupling inequalities, which allow us to relate quadratic polynomials to \emph{high-dimensional} linear anticoncentration problems.
We will state and discuss our decoupling inequalities momentarily (when we start to break down the proofs of \cref{thm:rank,thm:stability}), but first we introduce a very powerful tool which gives us control over the high-dimensional linear anticoncentration problems that arise from our decoupling inequalities: namely, the so-called ``optimal inverse theorem'' of Nguyen and Vu~\cite{NV11} (see the discussion in \cref{subsubsec:inverse} for context). This inverse theorem is stated in terms of \emph{generalised arithmetic progressions}, defined as follows.
\begin{definition}\label{def:GAP}
    Let $G$ be an additive group (i.e., an abelian group, additively written). A \emph{symmetric generalised arithmetic progression}, or a \emph{symmetric GAP} for short, is a map of the form
    \[\varphi: \{-N_1,-N_1+1,\dots,N_1\}\times \dots\times \{-N_r,-N_r+1,\dots,N_r\}\to G\]
    for some $r,N_1,\dots,N_r\in \mb N$, such that there exist $v_1,\dots,v_r \in G$ with $\varphi(a_1, \dots ,a_r)=a_1 v_1 + \dots + a_r v_r$ for all $(a_1,\dots,a_r)\in \{-N_1,-N_1+1,\dots,N_1\}\times \dots\times \{-N_r,-N_r+1,\dots,N_r\}$ (in other words, such that $\varphi$ is the restriction of a group homomorphism $\mathbb{Z}^r\to G$). 
    
    We say that $r$ is the \emph{rank} and $\prod_{i=1}^r (2N_i+1)$ is the \emph{volume} of the symmetric GAP. Furthermore, we say that a subset $G'\subseteq G$ is contained in (or ``lies in'') a symmetric GAP $\varphi$, if $G'$ is contained in the image\footnote{It is more common to define GAPs as sets, not maps (i.e., many authors would say that the image of $\varphi$ is itself a GAP). However, this creates some ambiguity around rank and volume (since the image of $\varphi$ may not uniquely determine $\varphi$), and later in this paper it will be rather important to be precise about these notions.} of $\varphi$.
\end{definition}

Note that any symmetric GAP has volume at least $1$. Now, the optimal inverse theorem of Nguyen and Vu~\cite[Theorem 2.5]{NV11}  is as follows.
\begin{theorem}\label{theorem: optimal inverse theorem}
    Fix $C > 0$.
    Let $v_1,\dots,v_n$ be elements of a torsion-free\footnote{Recall that an abelian group is \emph{torsion-free} if there is no nonzero element of finite order.} additive group $G$ and $\xi_1,\dots,\xi_n$ be i.i.d.\ Rademacher random variables.
    If
    \[\rho:=\sup_{z\in G} \Pr[\xi_1v_1+\dots+\xi_nv_n=z] > n^{-C},\]
    then for any $\sqrt{n} \le n' \le n$, there exists a symmetric GAP with rank $r=O_C(1)$ and volume $O_C(\rho^{-1}(n')^{-r/2})$, that contains all but at most $n'$ of $v_1,\dots,v_n$.
\end{theorem}

Note that the symmetric GAP in the conclusion must automatically have rank $r\le 4C$ if $n$ is sufficiently large in terms of $C$ (since otherwise the volume $O_C(\rho^{-1}(n')^{-r/2})\le O_C(n^C\sqrt{n}^{-r/2})$ would be smaller than 1). In fact, if $n'$ is chosen to be linear in $n$ one obtains an even better bound for the rank (namely, $r\le 2C$ if $n'\ge \delta n$ for some fixed $\delta>0$ and $n$ is sufficiently large with respect to $C$ and $\delta$).

If we take $G$ to be a high-dimensional space of the form $\mb F^d$ (for $\mb F\in\{\mb R,\mb C\}$), then \cref{theorem: optimal inverse theorem} gives us anticoncentration bounds for random variables of the form $A\vec \xi$  (where $A\in \mb F^{d\times n}$ is a matrix, and $\vec \xi\in \mb \{-1,1\}^n$ is a vector of i.i.d.\ Rademacher random variables). The quality of the bound depends on how well the columns of $A$ can be ``covered'' by a GAP of low rank and small volume. In particular, one can already obtain useful bounds just using information about the ``robust rank'' of $A$ (without taking into account the volume of the covering GAP); such a bound was previously obtained by Hal\'asz~\cite{Hal77}, and will be used occasionally in this paper (it appears explicitly as \cref{thm:Halasz}).

Now we move into the details of the proofs of \cref{thm:rank,thm:stability}. Both proofs follow a similar proof strategy, but the proof of \cref{thm:rank} is simpler, so we start with that.

\subsubsection{Bounds in terms of the rank}\label{subsubsec:rank} We now outline the proof of \cref{thm:rank}, approaching the ``generic bound'' $1/n$ for quadratic forms of high rank.

First, our decoupling inequality is as follows. (Recall from \cref{subsec:notation} our notation $\vec \xi[X]$, $A[X,Y]$ for subvectors and submatrices). 
\begin{definition}\label{def:shifted-rademacher}
    We say that a random variable $\xi$ is \emph{shifted Rademacher} if $\xi-\mb E \xi$ has the Rademacher distribution.
\end{definition}
\begin{lemma}\label{lem:multiple-decoupling-LO}
    Fix $\mb F\in \{\mb R,\mb C\}$ and a positive integer $k$.
    \begin{itemize}
        \item Let $f\in \mb F[x_1,\dots,x_n]$ be a quadratic polynomial, so there is a unique way to write the quadratic part of $f(\vec x)$ as $\vec x^{\,T} A\vec x$ for some symmetric matrix $A\in \mb F^{n\times n}$.
        \item Let $\vec{\xi}=(\xi_{1},\dots,\xi_{n})$ be a sequence of independent Rademacher random variables.
        \item Consider a partition $\{1,\dots,n\}=X\cup Y$.
    \end{itemize}
    Then,
    \begin{equation}
        \sup_{z \in \mb{F}} \Pr[f(\vec \xi)=z] 
        \le \Pr\Bigl[(\vec{\alpha}^{\,(i)})^T A[X,Y]\,\vec{\xi}[Y]=\psi(\vec{\alpha}^{\,(i)})\text{ for all }i\in\{1,\dots,k\}\Bigr]^{1/(k+1)},\label{eq:decoupling-3}
    \end{equation}
    for some function $\psi:\mb F^X\to \mb F$, and some i.i.d.\ random vectors $\vec \alpha^{(1)},\dots,\vec \alpha^{(k)}\in \mb F^X$ with independent shifted Rademacher entries (all independent of $\vec \xi[Y]$).
\end{lemma}

We give the short proof of \cref{lem:multiple-decoupling-LO} in \cref{sec:decoupling}.

Note that, after conditioning on an outcome of $\vec \alpha^{(i)}$, the expression $(\vec \alpha^{(i)})^T \allowbreak A[X,Y]\,\vec{\xi}[Y]-\psi(\vec{\alpha}^{(i)})$ becomes a linear function of $\vec{\xi}[Y]$. So, the $k=1$ case of \cref{lem:multiple-decoupling-LO} provides a direct way to deduce quadratic anticoncentration bounds from linear anticoncentration bounds. This particular case of \cref{lem:multiple-decoupling-LO} is well-known (in this form, it seems to have been first observed by Costello and Vu~\cite{CV08}, though similar inequalities appeared earlier in \cite{RS96,CTV06,Sid93}).

Although it is near-trivial to generalise the proof of the $k=1$ case of \cref{lem:multiple-decoupling-LO} to general $k$, the possibility of ``decoupling with multiple copies'' (and the advantage of doing so) has been recognised only recently. In particular, a related (though much more delicate) decoupling scheme was recently used by the second and third authors~\cite{KS} to obtain optimal bounds for the quadratic case of \cref{thm:polynomial-LO}. We believe \cref{lem:multiple-decoupling-LO} makes it quite transparent why it is helpful to decouple with more than one copy. Indeed, \cref{eq:decoupling-3} involves $k$ simultaneous equations; if we imagine that each equation is satisfied independently with probability at most about $1/n$, then \cref{lem:multiple-decoupling-LO} provides a bound of about $n^{-k/(k+1)}$, which gets closer and closer to $1/n$ as we increase $k$.

Of course, the $k$ simultaneous equations in \cref{eq:decoupling-3} are not independent. The primary challenge in our proof of \cref{thm:rank} is to demonstrate that, for the purposes of anticoncentration, these equations are ``approximately independent'', if \cref{C1} does not hold (i.e., if our quadratic form robustly has high rank). Specifically, given \cref{lem:multiple-decoupling-LO,lemma: close to symmetric low rank}, the following lemma is the main part of the proof of \cref{thm:rank}. (Recall the definition of a shifted Rademacher random variable from \cref{def:shifted-rademacher}).
\begin{lemma} \label{thm: multi-x-single-y costello}
    Fix any $\varepsilon,\delta \in(0,1]$ and $k \ge 1$, and let $\mb{F} \in \{\mb{R},\mb{C}\}$.
    Let $n$ be sufficiently large (in terms of $\varepsilon,\delta,k$), and let $A \in \mb{F}^{n \times n}$ be a matrix.
    Then, at least one of the following holds.
    \begin{enumerate}[{\bfseries{G\arabic{enumi}}}]
        \item\label{G1} At most a $\delta$-fraction of the $\lceil2k^2/\varepsilon\rceil \times \lceil2k^2/\varepsilon\rceil$ submatrices of $A$ are nonsingular, or
        \item\label{G2} letting $\Xi \in \mb{F}^{k \times n}$ be a random matrix with independent shifted Rademacher entries, and (independently) letting $\vec{\eta} \in \{-1,1\}^n$ be a random column vector with independent Rademacher entries, for any (non-random) function $\vec \varphi: \mb{F}^{k\times n}\to \mb{F}^k$
        we have
        \[
            \Pr\big[\Xi A \vec{\eta} = \vec \varphi(\Xi)  \big]
            \le n^{-k+\varepsilon}.
        \]
    \end{enumerate}
\end{lemma}

To prove \cref{thm: multi-x-single-y costello}, we first use a counting argument, and some anticoncentration estimates using the rank of $A$, to show that it is unlikely that most of the columns of $\Xi A$ are contained in a generalised arithmetic progression with small rank and volume. Then, we can apply \cref{theorem: optimal inverse theorem} (the Nguyen--Vu optimal linear inverse theorem) with the randomness of $\vec \eta$, to prove an upper bound on the probability that $\Xi A \vec{\eta} = \vec \varphi(\Xi)$.

\begin{remark}\label{rem:optimal?}
We suspect that the rank bound in \cref{G1} is far from best possible; indeed, we suspect that the same statement should hold with ``$\lceil2k^2/\varepsilon\rceil$'' replaced by ``$2k$'' (if true, this would be best-possible, as one can see by taking $A$ to be an appropriate block-diagonal matrix). Actually, this would constitute a generalisation of the $d=2$ case of \cref{thm:multilinear}, via \cref{lemma: close to symmetric low rank}. We did not fight too hard to optimise the particular rank bound in \cref{G1}, but our methods do seem to be somewhat too crude to prove an optimal bound.
\end{remark}

In \cref{sec:rank}, we give the details of the proof of \cref{thm: multi-x-single-y costello}, and deduce \cref{thm:rank} from \cref{thm: multi-x-single-y costello,lem:multiple-decoupling-LO,lemma: close to symmetric low rank}.

\subsubsection{A power-saving improvement in the ``robustly irreducible'' case} Now, we outline the proof of \cref{thm:stability}, giving a power-saving improvement for quadratic polynomials which are ``robustly irreducible'' over $\mb C$.

First, the significance of $\mb C$ is that a quadratic polynomial is reducible over $\mb C$ if and only if it has rank at most 2. (On the other hand, reducibility over $\mb R$ cannot be expressed in terms of rank: $x^2+y^2$ and $x^2-y^2$ both have rank two, but the former is irreducible and the latter is reducible).

Given this connection between irreducibility and rank, our proof of \cref{thm:stability} follows the same general scheme as the proof of \cref{thm:rank}, though the details are much more delicate. Indeed, for \cref{thm:rank} we could wastefully assume our coefficient matrix $A$ had high rank, but here we can only assume that the rank is at least 3.

Just as for \cref{thm:rank}, the first ingredient is a decoupling lemma. Instead of \cref{lem:multiple-decoupling-LO}, we use the following more technical lemma.
\begin{definition}\label{def:lazy-rademacher}
 we say that a random variable $\alpha$ is \emph{lazy Rademacher} if it can be expressed as the difference of two independent Rademacher random variables (explicitly, this means $\Pr[\alpha=0]=1/2$ and $\Pr[\alpha=-1]=\Pr[\alpha=1]=1/4$).
\end{definition}
\newcommand{\vx}{\vec{\alpha}}
\newcommand{\vy}{\vec{\beta}}
\newcommand{\vz}{\vec{\gamma}}
\begin{lemma}\label{lem:custom-decoupling-LO}
    Fix $\mb F\in \{\mb R,\mb C\}$.
    \begin{itemize}
        \item Let $f \in \mb{F}[x_1,\dots,x_n]$ be a quadratic polynomial, so there is a unique way to write the quadratic part of $f(\vec{x})$ as $\vec{x}^T A \vec{x}$ for some symmetric matrix $A \in \mb{F}^{n\times n}$.
        \item Let $\vec{\xi}=(\xi_{1},\dots,\xi_{n})$ be a sequence of independent Rademacher random variables.
        \item Consider a partition $\{1,\dots,n\}=X\cup Y\cup Z$.
    \end{itemize}
    Then,
    \[
        \sup_{z\in \mb{F}} \Pr[f(\vec \xi)=z] 
        \le \Pr\Bigl[\vx^{\,T} A[X,Y] \vy=0, \quad\vx^{\,T} A[X,Z] \vz=\varphi(\vec{\alpha}, \vec{\beta}),\quad\vy^{\,T} A[Y,Z] \vz=\psi(\vec{\alpha},\vec{\beta})\Bigr]^{1/4},
    \]
    for some functions $\varphi,\psi:\mb F^X\times \mb F^Y \to \mb F$, and independent random vectors $\vx,\vy,\vz$, where the entries of $\vx,\vy$ are i.i.d.\ lazy Rademacher, and the entries of $\vz$ are i.i.d.\ Rademacher.
\end{lemma}

We prove \cref{lem:custom-decoupling-LO} in \cref{sec:decoupling}. Then, given \cref{lem:custom-decoupling-LO,lemma: close to symmetric low rank}, the proof of \cref{thm:stability} essentially comes down to the following lemma.
\begin{lemma}\label{theorem: three bilinears}
Fix any $\varepsilon\in(0,1]$, and let $\mb{F} \in \{\mb{R},\mb{C}\}$.
    Let $n$ be sufficiently large (in terms of $\varepsilon$), and consider matrices $A_1,A_2,A_3 \in \mb{F}^{n \times n}$.
    Then, at least one of the following holds.
    \begin{enumerate}[{\bfseries{H\arabic{enumi}}}]
        \item\label{H1} 
         For one of the matrices $A_1,A_2,A_3$,
        at most an $\varepsilon$-fraction of its $3\times 3$ submatrices are nonsingular, or
        \item\label{H2}
        letting $\vx,\vy\in \{-1,0,1\}^n$ be vectors of lazy Rademacher random variables, and $\vz\in \{-1,1\}^n$ be a vector of Rademacher random variables (all independent), 
        for any functions $\varphi,\psi: \mb{F}^n\times \mb{F}^n\to \mb{F}$
        we have
        \[
\Pr\left[ \vx^{\,T} A_1 \vy=0, \quad\vx^{\,T} A_2 \vz=\varphi(\vx,\vy), \quad\vy^{\,T} A_3 \vz=\psi(\vx,\vy) \right] \le n^{-2-1/6+\varepsilon}.
        \]
    \end{enumerate}
\end{lemma}

At a high level, the proof of \cref{theorem: three bilinears} can be compared to the proof of \cref{thm: multi-x-single-y costello}. To study anticoncentration of a bilinear form of random variables, we reveal its two sets of variables in two stages; in the first stage we prepare for the application of \cref{theorem: optimal inverse theorem} (the Nguyen--Vu optimal linear inverse theorem), and in the second stage we actually apply \cref{theorem: optimal inverse theorem}. However, the details of \cref{theorem: three bilinears} are much more complicated than the details of \cref{thm: multi-x-single-y costello}, for two main reasons.

First, since \cref{theorem: three bilinears} concerns three different bilinear forms which depend on each other in a ``cyclic'' manner, we must be very careful about the way that we reveal the random variables $\vx,\vy,\vz$. We actually need to consider three different cases, in which we reveal the random variables $\vx,\vy,\vz$ in different ways. (The three cases are defined in terms of the random vectors $\vx^{\,T} A_1$ and $\vx^{\,T} A_2$; specifically, we need to distinguish whether one of these vectors is almost zero, and whether these vectors are almost collinear).

Second, we have much less room to make crude estimates, and we need to be much more careful about the tradeoffs between the stages of random exposure.
In particular, in addition to the simple counting-based estimates featuring in proof of \cref{thm: multi-x-single-y costello}, we also need to use some more technical estimates proved by Costello~\cite{Cos13} via number-theoretic means.

In \cref{sec:power-saving}, we give the details of the proof of \cref{theorem: three bilinears}, and deduce \cref{thm:stability} from \cref{theorem: three bilinears,lem:custom-decoupling-LO,lemma: close to symmetric low rank}.

\section{Preliminaries}\label{sec:prelim}
In this section we collect a few basic tools that will be used throughout the paper. First, the following lemma is related to the ``local-to-global'' point of view in \cref{thm:tensor-property-testing,lemma: close to symmetric low rank}. (Very similar estimates appear for example in \cite[Section~3.4]{Cos13}, but we were not able to find an easily citable reference).

\begin{lemma}\label{lemma: tuple-counting}
    Let $p,\delta\in (0,1)$ and $n \ge r \ge 1$.
    Suppose $I$ is a random subset of $\{1,\dots,n\}$ such that $\Pr[|I|\ge (1-\delta) n] \ge p$.
    Then, for all but a $2r\delta$-fraction of the $r$-element subsets $S\subset \{1,\dots,n\}$, we have $\Pr[S \subset I] \ge p/2.$
\end{lemma}
\begin{proof}
    For $i\in \{1,\dots,n\}$, let $\mathcal{E}_{i}$ be the event that $i\in I$
and let $\mathcal{E}$ be the event that $|I|\ge(1-\delta)n$, so we
are assuming $\Pr[\mathcal{E}]\ge p$. Let $\mathcal{Q}$ be the collection
of sets $\{i_{1},\dots,i_{r}\}\subseteq \{1,\dots,n\}$ of $r$ distinct indices such that
$\Pr[\mathcal{E}_{i_{1}}\cap\dots\cap \mathcal{E}_{i_{r}}]<p/2\le\Pr[\mathcal{E}]/2$.
So, our goal is to prove that $|\mathcal{Q}|\le2r\delta\binom{n}{r}$.

On one hand, we have
\begin{equation}
\mb E\left[\mbm 1_{\mathcal{E}}\binom{|I|}{r}\right]\ge\binom{\lceil (1-\delta)n\rceil}{r}\Pr[\mathcal{E}],\label{eq:double-counting-events-1}
\end{equation}
but on the other hand we have
\begin{align}
\mb E\left[\mbm 1_{\mathcal{E}}\binom{|I|}{r}\right] =\sum_{i_{1}<\dots<i_{r}}\Pr[\mathcal{E}\cap\mathcal{E}_{i_{1}}\cap\dots\cap\mathcal{E}_{i_{r}}] 
 & \le\sum_{\{i_{1},\dots,i_{r}\}\in\mathcal{Q}}\Pr[\mathcal{E}_{i_{1}}\cap\dots\cap\mathcal{E}_{i_{r}}]+\sum_{\{i_{1},\dots,i_{r}\}\notin\mathcal{Q}}\Pr[\mathcal{E}]\nonumber\\
 & <|\mathcal{Q}|\cdot(\Pr[\mathcal{E}]/2)+\left(\binom{n}{r}-|\mathcal{Q}|\right)\cdot\Pr[\mathcal{E}].\label{eq:double-counting-events-2}
\end{align}
Combining \cref{eq:double-counting-events-1,eq:double-counting-events-2}
yields the desired inequality
\[
|\mathcal{Q}|<2\left(\binom{n}{r}-\binom{\lceil (1-\delta)n\rceil}{r}\right)\le2\left(\delta n\cdot\binom{n-1}{r-1}\right)=2r\delta\binom{n}{r}.
\]
(The second inequality has a combinatorial interpretation: given a
set of $n$ elements, and an identification of $\lfloor \delta n\rfloor$ of those
elements as ``special'', we are counting the number of ways to choose
a subset of $r$ elements, at least one of which is special).
\end{proof}

Next, the following high-dimensional Littlewood--Offord theorem is due to Hal\'asz~\cite{Hal77} (it can also be deduced from \cref{theorem: optimal inverse theorem}).
\begin{theorem}\label{thm:Halasz}
    Let $n\ge d\ge 1$ and $\ell\ge 1$ and $\mb{F} \in \{\mb{R},\mb{C}\}$, and suppose $A \in \mb{F}^{d \times n}$ is a matrix containing at least $\ell$ disjoint  nonsingular $d\times d$ submatrices. 
    Let $\vec{\xi}=(\xi_1,\dots,\xi_n)\in \mb{F}^n$ be a vector of independent Rademacher (or shifted Rademacher, or lazy Rademacher\footnote{It is easy to deduce the lazy and shifted Rademacher cases from the Rademacher case. Indeed, recall that a lazy Rademacher random variable is the difference of two independent Rademacher random variables; after conditioning on one of these it is a shifted Rademacher random variable. By linearity, shifting the entries of $\vec \xi$ does not affect anticoncentration.}) random variables.
    Then,
    \[
        \sup_{\vec{z}\in \mb F^d} \Pr\left[A\vec{\xi} = \vec{z} \right] \le O_{d}(\ell^{-d/2}).
    \]
\end{theorem}
Hal\'asz' theorem has the following consequence for the probability that a random variable falls in an affine-linear subspace (see for example \cite[Corollary~3.4]{KS}).
\begin{corollary}\label{cor:Halasz}
    Let $n\ge d\ge 1$ and $\ell\ge 1$ and $\mb{F} \in \{\mb{R},\mb{C}\}$, and suppose a matrix $A \in \mb{F}^{d \times n}$ contains at least $\ell$ disjoint  nonsingular $d\times d$ submatrices. 
    Let $\vec{\xi}=(\xi_1,\dots,\xi_n)\in\mb{F}^n$ be a vector of independent Rademacher (or shifted Rademacher, or lazy Rademacher) random variables.
    Then, for any $k$-dimensional affine-linear subspace $\mc W\subseteq \mb F^d$, we have
    \[
        \Pr\left[A\vec{\xi} \in \mc W \right] \le O_{d}(\ell^{-(d-k)/2}).
    \]
\end{corollary}

Finally, for \cref{thm:rank,thm:stability} we will need some basic observations about matrices. The following fact allows us to translate between the property of having many nonsingular submatrices and the property of having many \emph{disjoint} nonsingular submatrices.
\begin{fact}\label{fact: disjoint and non-disjoint nonsingular submatrices}
    Let $n \ge d \ge 1$ and $m \ge 1$ and $\mb{F} \in \{\mb{R},\mb{C}\}$.
    If a matrix $A \in \mb{F}^{d \times n}$ does not contain more than $m$ disjoint nonsingular $d\times d$ submatrices, then at most an $(md^2/n)$-fraction of the $d\times d$ submatrices in $A$ are nonsingular.
\end{fact}
\begin{proof}
    Fix a maximal collection of disjoint nonsingular $d\times d$ submatrices in $A$; these involve at most $md$ different columns.
    The maximality of the collection implies that any $d\times d$ submatrix that does not use any of these $md$ columns must be singular.
    Hence, the number of nonsingular $d\times d$ submatrices in $A$ is at most $md \binom{n-1}{d-1}=(md^2/n)\binom{n}{d}$.
\end{proof}
The following ``monotonicity'' fact says that if a matrix has many nonsingular $r\times r$ submatrices, then for any $k\leq r$, there are also many nonsingular $k\times k$ submatrices.
\begin{fact} \label{lem:rank-monotonicity}
    Let $\mb{F} \in \{\mb{R}, \mb{C}\}$. Let $A\in \mb F^{n\times n}$ be a matrix with the property that at least an $\varepsilon$-fraction of its $r\times r$ submatrices are nonsingular. Then, for any $k\le r$, at least an $(\varepsilon/\binom r k)$-fraction of the $k\times k$ submatrices of $A$ are nonsingular.
\end{fact}
\begin{proof}
Suppose that $A$ has at least $\varepsilon\binom n r^2$ nonsingular $r\times r$ submatrices. For $k\le r$, each of these nonsingular $r\times r$ submatrices contains at least $\binom r k$ nonsingular $k\times k$ submatrices, and every $k\times k$ submatrix can be extended to exactly $\binom{n-k}{r-k}^2$ different $r\times r$ submatrices. So, there are at least $\varepsilon\binom n r^2\binom r k/\binom{n-k}{r-k}^2=\left(\varepsilon/\binom r k\right)\binom n k^2$ nonsingular $k\times k$ submatrices.
\end{proof}
The following fact says that if a matrix has many nonsingular $r\times r$ submatrices, then it cannot be close to a matrix of rank less than $r$. (It can be viewed as the easy ``converse direction'' of \cref{lemma: close to symmetric low rank}).
\begin{fact}\label{lem:close-easy}
Let $\mb{F} \in \{\mb{R}, \mb{C}\}$. Fix $\varepsilon>0$ and $r\ge 1$, and let $n$ be sufficiently large (in terms of $\varepsilon,r$). Let $A\in \mb F^{n\times n}$ be a matrix with the property that at least an $r^2\varepsilon$-fraction of its $r\times r$ submatrices are nonsingular. Then there is no matrix $B\in \mb F^{n\times n}$ with rank less than $r$ such that $\|A-B\|_0\le \varepsilon n^2$.
\end{fact}
\begin{proof}
If $\|A-B\|_0\le \varepsilon n^2$ then there are at most $\varepsilon n^2\binom {n-1}{r-1}^2=r^2\varepsilon\binom{n}{r}^2$ different $r\times r$ submatrices of $A$ that are not identical to their counterpart in $B$. So, if $B$ has rank less than $r$, then at most an $r^2\varepsilon$-fraction of the $r\times r$ submatrices in $A$ are nonsingular.
\end{proof}

The following lemma says that if a matrix has many nonsingular $r\times r$ submatrices, then we can partition it into $q$ block matrices, each of which individually has many $r\times r$ submatrices.
\begin{lemma} \label{lemma: tri-partition}
    Let $\mb{F} \in \{\mb{R}, \mb{C}\}$. Fix $\varepsilon>0$, and $q,r\ge 1$. Let $n$ be sufficiently large (in terms of $\varepsilon,q,r$), and divisible by $q$.
    Let $A\in \mb F^{n\times n}$ be a matrix with the property that at least an $\varepsilon$-fraction of its $r\times r$ submatrices are nonsingular.
    
    Then, there is a partition of $\{1,\dots,n\}=I_1\cup \dots \cup I_q$ with $|I_1|=\dots=|I_q|=n/q$, such that for any $i,j\in \{1,\dots,q\}$, more than an $\varepsilon/2$-fraction of the $r\times r$ submatrices of the matrix $A[I_i,I_j]$ are nonsingular.
\end{lemma}
\begin{proof}
    This is a routine application of the probabilistic method. 
    Consider a \emph{random} partition of $\{1,\dots,n\}$ into sets $I_1,\dots,I_q$ of size $n/q$. 
    Let $X_{i,j}$ be the number of nonsingular $r\times r$ submatrices of $A[I_i,I_j]$.
    
    There are at least $\varepsilon \binom n r^2-n\binom {n-1}{r-1}^2=(\varepsilon-o(1))\binom n r^2$ nonsingular $r\times r$ submatrices which do not involve a diagonal entry of $A$ (here and in the rest of this proof, the asymptotic notation is for fixed $\varepsilon,q,r$). 
    For any $i,j\in \{1,\dots,q\}$, each of these submatrices is contained in $A[I_i,I_j]$ with probability $(1+o(1))/q^{2r}$. 
    So, \[
        \mb E X_{i,j} \ge \frac{(\varepsilon-o(1))\binom n r^2}{q^{2r}}=(\varepsilon-o(1))\binom {n/q}r^2. 
    \]
    Now, moving an element in or out of $I_i$ or $I_j$ can only change $X_{i,j}$ by at most $O(n^{2r-1})$. 
    So, by the Azuma--Hoeffding inequality (see for example \cite[Lemma 11]{Frieze-Pittel} for a statement adapted to our purposes\footnote{This statement is for random permutations. But note that a random partition into $q$ parts can easily be defined in terms of a random permutation: simply take the first $n/q$ elements as the first part, the next $n/q$ elements as the second part, and so on.}), for each $i,j$ we have \[
        \Pr\left[ X_{i,j}\le (\varepsilon/2)\binom {n/q}r^2\right]
        \le \exp\left(-\Omega\left(\frac{\big(\Ex X_{i,j} \big)^2}{n\cdot O(n^{2r-1})^2}\right)\right)=\exp(-\Omega(n))<\frac1{q^2}.
    \]
    for sufficiently large $n$. 
    It follows that with positive probability we have $X_{i,j}> (\varepsilon/2)\binom {n/q}{r}^2$ for all $q^2$ choices of $i,j$, as desired.
\end{proof}

\section{A local-to-global lemma for reducibility of tensors}\label{sec:tensor-property-testing}
In this section, we prove \cref{thm:tensor-property-testing}: if most of the small subtensors of $T$ are reducible, then $T$ itself is close to being reducible. 
There are different types of reducibility with respect to particular partitions, as follows.
\begin{definition}\label{def:reducible-under}
For a tensor $T\in \mb F^{I_1}\otimes \dots\otimes \mb F^{I_d}$ and a non-trivial partition $\{1,\dots,d\}=J_1\cup J_2$, we say that $T$ is \emph{reducible with respect to $\{J_1,J_2\}$} if there are tensors $T_1\in \bigotimes_{j\in J_1} \mb F^{I_j}$ and $T_2\in \bigotimes_{j\in J_2}\mb F^{I_j}$, such that $T$ can be factored as a tensor product $T=T_1\otimes T_2$.
\end{definition}
We note that reducibility of a tensor $T$ with respect to $\{J_1,J_2\}$ can equivalently be interpreted as the property that a certain matrix associated with $T$ has rank at most $1$, as follows.
\begin{fact}\label{fact:tensor-to-matrix}
    Fix disjoint sets $I_1,\dots,I_d$, a field $\mb F\in \{\mb R,\mb C\}$, and a non-trivial partition $\{1,\dots,d\}=J_1\cup J_2$. For $s\in \{1,2\}$, let $r_s=\prod_{j\in J_s}|I_j|$, so we have a vector space isomorphism $\bigotimes_{j\in J_s}\mb F^{I_j}\cong \mb F^{r_s}$.
    Then, \[\mb F^{I_1}\otimes \dots\otimes \mb F^{I_d}=\Bigg(\bigotimes_{j\in J_1}\mb F^{I_j}\Bigg)\otimes \Bigg(\bigotimes_{j\in J_2}\mb F^{I_j}\Bigg)\cong \mb F^{r_1\times r_2};\]
    that is to say, $\{J_1,J_2\}$ gives rise to a correspondence between $I_1\times \dots\times I_d$ tensors and $r_1\times r_2$ matrices. 
    A tensor in $\mb F^{I_1}\otimes \dots\otimes \mb F^{I_d}$ is reducible with respect to $\{J_1,J_2\}$ if and only if the corresponding $r_1\times r_2$ matrix has rank at most 1.
\end{fact}
We also adopt the following notation for subtensors, similar to that of matrices.
\begin{definition}\label{def:subtensors}
    Consider a tensor $T:I_1\times\dots\times I_d\to \mb F$ (represented as a $d$-dimensional array). For subsets $S_1\subseteq I_1,\dots,S_d\subseteq I_d$, write $T[S_1,\dots,S_d]$ for the corresponding subtensor (formally, this is a restriction of the function $T$ to the set $S_1\times \dots\times S_d$).
\end{definition}

We break the proof of \cref{thm:tensor-property-testing} into three simpler lemmas. First, by a simple averaging argument, we can reduce our consideration to the subtensors containing a particular nonzero entry.

\newcommand{\xs}[1]{{i^\star_{#1}}}
\begin{lemma}\label{claim:tensor-property-testing-1}
    Fix $d,r>1$ and  $\delta>0$ and $c\in [0,1]$, and any field $\mb F$.
    Consider a $d$-dimensional tensor $T:I_1\times \dots\times I_d\to \mb F$ (represented as a $d$-dimensional array). 
    Suppose that at least a $\delta^{c}$-fraction of the entries of $T$ are nonzero and that all but at most a $\delta$-fraction of the $r\times \dots\times r$ subtensors are reducible.
    
    Then, there is a nonzero entry $T(\xs{1},\xs{2},\dots,\xs{d})$ such that all but at most a $\delta^{1-c}$-fraction of the $r\times \dots \times r$ subtensors which contain the entry $T(\xs{1},\xs{2},\dots,\xs{d})$ are reducible.
\end{lemma}
\begin{proof}
    For each $j\in \{1,\dots,d\}$, let $i_j^\star,i_j^1,\dots,i_j^{r-1}\in I_j$ be a uniformly random sequence of $r$ distinct indices from $I_j$. 
    So, by assumption we have
    \begin{align*}
        \delta &\ge \Pr\Big[ T\big[\{i_1^\star,i_1^1,\dots,i_1^{r-1}\}, \dots, \{i_d^\star,i_d^1,\dots,i_d^{r-1}\}\big]\text{ is not reducible}\Big]\\
        &=\mb E\bigg[\Pr\Big[T\big[\{i_1^\star,i_1^1,\dots,i_1^{r-1}\}, \dots, \{i_d^\star,i_d^1,\dots,i_d^{r-1}\}\big]\text{ is not reducible}\;\Big|\; i_1^\star,\dots,i_d^\star\Big]\bigg].
    \end{align*}
    By Markov's inequality, we deduce that for more than a $(1-\delta^c)$-fraction of choices of $(i_1^\star,\dots,i_d^\star)$, we have
    \[
        \Pr\Big[T\big[\{i_1^\star,i_1^1,\dots,i_1^{r-1}\}, \dots, \{i_d^\star,i_d^1,\dots,i_d^{r-1}\}\big]\text{ is not reducible}\;\Big|\; i_1^\star,\dots,i_d^\star\Big]\le \delta^{1-c},
    \]
    meaning that all but at most a $\delta^{1-c}$-fraction of the $r\times \dots \times r$ subtensors which contain the entry $T(\xs{1},\dots,\xs{d})$ are reducible. 
    Since at least a $\delta^c$-fraction of the entries of $T$ are nonzero, for at least one of these choices of $(i_1^\star,\dots,i_d^\star)$, the entry $T(\xs{1},\dots,\xs{d})$ is nonzero.
\end{proof}

The second lemma (arguably the most important) allows us to relate reducibility in general to reducibility with respect to a particular partition.

\begin{lemma}\label{claim:tensor-property-testing-2}
    Fix $d>1$ and $\delta > 0$, and let $\ell=2^{d-1}-1$. Consider a tensor $T:I_1\times\dots\times I_d\to \mb F$ (represented as a $d$-dimensional array) and suppose $T(\xs{1},\dots,\xs{d})\ne 0$ for some indices $\xs{1}\in I_1,\dots,\xs{d}\in I_d$.
    Suppose that all but at most a $\delta$-fraction of the $2^{d-1}\times \dots \times 2^{d-1}$ subtensors of $T$ which contain the entry $T(\xs{1},\dots,\xs{d})$ are reducible.
    
    Then, there is a non-trivial partition $\{1,\dots,d\}=J_1\cup J_2$, such that all but at most a $\delta^{1/\ell}$-fraction of the $2\times \dots\times 2$ subtensors of $T$ which contain the entry $T(\xs{1},\dots,\xs{d})$ are reducible with respect to $\{J_1,J_2\}$.
\end{lemma}
\begin{proof}
Note that $\ell=2^{d-1}-1$ is the number of partitions $\mc P=\{J_1,J_2\}$ of $\{1,\dots,d\}$, such that neither $J_1$ nor $J_2$ is empty. Let $\mathcal P_1,\dots,\mc P_\ell$ be an enumeration of all these partitions.

For each $j\in \{1,\dots,d\}$, let $i_j^1,\dots,i_j^\ell$ be uniformly random indices sampled independently from $I_j\setminus \{i_j^\star\}$. Note that there may be fewer than $\ell$ different indices among $i_j^1,\dots,i_j^\ell$, due to repetitions; let $Q_j$ be a random set of exactly $\ell$ indices in $I_j\setminus \{i_j^\star\}$, obtained by starting from $\{i_j^1,\dots,i_j^\ell\}$ and adding the appropriate number of additional random indices. Do this independently for each $j\in \{1,\dots,d\}$.

Then, we have
\begin{align*}\delta&\ge \Pr\Big[T\big[Q_1\cup \{\xs{1}\}, \dots, Q_d\cup \{\xs{d}\}\big]\text{ is not reducible}\Big]\\
&= \Pr\Bigg[\bigcap_{t=1}^\ell\Big\{T\big[Q_1\cup \{\xs{1}\}, \dots, Q_d\cup \{\xs{d}\}\big]\text{ is not reducible with respect to }\mc P_t\Big\}\Bigg]\\
&\ge \Pr\Bigg[\bigcap_{t=1}^\ell\Big\{T\big[\{i_1^t,\xs{1}\}, \dots, \{i_d^t,\xs{d}\}\big]\text{ is not reducible with respect to }\mc P_t\Big\}\Bigg]\\
&=\prod_{t=1}^\ell\Pr\Big[T\big[\{i_1^t,\xs{1}\}, \dots, \{i_d^t,\xs{d}\}\big]\text{ is not reducible with respect to }\mc P_t\Big].
\end{align*}
Indeed, the first line is by the assumption on $T$ and the fact that each $Q_j$ is an independent and uniformly random subset of $I_j\setminus \{i_j^\star\}$ of size $\ell$,
the second line is by the definition of reducibility, the third line is due to the fact that $Q_j\cup \{\xs{j}\}\supseteq \{i_j^t,\xs{j}\}$ for all $j\in \{1,\dots,d\}$, and the last line is by the independence of $(i_1^t,\dots,i_d^t)$ between different $t$.

We deduce that there is some $t \in \{1,\dots,\ell\}$ such that, with probability at most $\delta^{1/\ell}$, the random subtensor $T\big[\{i_1^t,\xs{1}\}, \dots, \{i_d^t,\xs{d}\}\big]$ is not reducible with respect to $\mathcal P_t$. 
That is to say, all but at most a $\delta^{1/\ell}$-fraction of $2\times \dots\times 2$ subtensors of $T$ containing the entry $T(\xs{1},\dots,\xs{d})$ are reducible with respect to $\mc P_t$, as desired.
\end{proof}

In our final lemma, we observe that if $T$ has many $2\times \dots\times 2$ subtensors involving a single nonzero entry, which are all reducible with respect to the same partition, then we can construct a reducible tensor that agrees with $T$ on many entries.

\begin{lemma}\label{claim:tensor-property-testing-3}
    Fix $d\ge 1$ and $\delta > 0$. Let $n_0$ be sufficiently large in terms of $d$ and $\delta$. 
    Consider a tensor $T:I_1\times\dots\times I_d\to \mb F$ (represented as a $d$-dimensional array) with $|I_1|,\dots,|I_d|\ge n_0$, and a non-trivial partition $\{1,\dots,d\}=J_1\cup J_2$, and suppose $T(\xs{1},\dots,\xs{d})\ne 0$ for some indices $\xs{1}\in I_1,\dots,\xs{d}\in I_d$.
    Suppose that all but at most a $\delta$-fraction of the $2\times \dots\times 2$ subtensors of $T$ which contain the entry $T(\xs{1},\dots,\xs{d})$ are reducible with respect to $\{J_1,J_2\}$.

    Then one can make $T$ reducible by changing up to a $2\delta$-fraction of its entries.
\end{lemma}
\begin{proof}
Suppose without loss of generality that $J_1=\{1,\dots,h\}$ and $J_2=\{h+1,\dots,d\}$.
Let $T':I_1\times \dots\times I_d\to \mb F$ be the tensor defined by
\[T'(i_1,\dots,i_d)=\frac{T(i_1,\dots,i_{h},\xs{h+1},\dots,\xs{d})\,T(\xs{1},\dots,\xs{h},i_{h+1},\dots,i_{d})}{T(\xs{1},\dots,\xs{d})}.
\]
By definition, $T'$ is reducible. Indeed, we have $T'=T_1\otimes  T_2$ for $T_1: I_1\times \dots\times I_h\to \mb F$ given by $T_1(i_1,\dots,i_h)=T(i_1,\dots,i_{h},\xs{h+1},\dots,\xs{d})$ and $T_2: I_{h+1}\times \dots\times I_d\to \mb F$ given by $T_1(i_{h+1},\dots,i_d)=T(\xs{1},\dots,\xs{h},i_{h+1},\dots,i_{d})/T(\xs{1},\dots,\xs{d})$. It remains to show that $T'$ differs from $T$ in at most a $2\delta$-fraction of its entries.

To see this, note that for all $i_1\in I_1\setminus \{\xs{1}\},\dots,i_d\in I_d\setminus \{\xs{d}\}$ such that $T\big[\{i_1,i_1^\star\},\dots,\{i_d,i_d^\star\}\big]$ is reducible with respect to $\{J_1,J_2\}$, we have
\[T(i_1,\dots,i_d)\,T(\xs{1},\dots,\xs{d})-T(i_1,\dots,i_{h},\xs{h+1},\dots,\xs{d})\,T(\xs{1},\dots,\xs{h},i_{h+1},\dots,i_{d})=0\]
(this can be interpreted as a determinant of a $2\times 2$ matrix of rank at most $1$, recalling \cref{fact:tensor-to-matrix}).
Rearranging, we obtain $T(i_1,\dots,i_d)=T'(i_1,\dots,i_d)$. But, by assumption, the number of such choices of $i_1,\dots,i_d$ is at least $(1-\delta)\prod_{i=1}^d (|I_i|-1)\ge (1-2\delta)\prod_{i=1}^d |I_i|$. (Here we are using that $I_1,\dots,I_d$ are sufficiently large with respect to $d$ and $\delta$).
\end{proof}

We now combine \cref{claim:tensor-property-testing-1,claim:tensor-property-testing-2,claim:tensor-property-testing-3} to prove \cref{thm:tensor-property-testing}.
\begin{proof}[Proof of \cref{thm:tensor-property-testing}]
    As in \cref{claim:tensor-property-testing-2}, let $\ell:=2^{d-1}-1$, and note that then $\delta=(\varepsilon/2)^{2^{d-1}}=(\varepsilon/2)^{\ell+1}$.
    If at most a $\delta^{1/(\ell+1)}$-fraction of the entries in $T$ are nonzero, then $T$ can be made zero by changing up to a $\varepsilon$-fraction of the entries; we are done in this case.
    So, assume that $T$ contains more than a $\delta^{1/(\ell+1)}$-fraction of nonzero entries. We apply \cref{claim:tensor-property-testing-1} with $r=2^{d-1}$ and $c=1/(\ell+1)$ to see that $T$ has a nonzero entry $T(\xs{1},\xs{2},\dots,\xs{d})$ such that all but at most a $\delta^{\ell/(\ell+1)}$-fraction of the $2^{d-1}\times \dots \times 2^{d-1}$ subtensors which contain this entry are reducible.
    By \cref{claim:tensor-property-testing-2},
    there is a non-trivial partition $\{1,\dots,d\}=J_1\cup J_2$ such that all but at most a $\delta^{1/(\ell+1)}$-fraction of the $2\times \dots\times 2$ subtensors of $T$ containing the entry $T(\xs{1},\xs{2},\dots,\xs{d})$, are reducible with respect to $\{J_1,J_2\}$. Then, \cref{claim:tensor-property-testing-3}
    guarantees that $T$ can be made reducible by changing a fraction of up to $2\delta^{1/(\ell+1)}=\varepsilon$ of its entries, as desired.
\end{proof}

\section{Linear subspaces of the variety of reducible tensors}\label{sec:reducible-variety}
In this section, we prove \cref{thm:irreducible-variety}, characterising the linear subspaces of the variety of reducible tensors.

First, note that we can break up the set of reducible tensors according to the partition with respect to which they are reducible.
\begin{fact}\label{fact:break-up-by-partitions}
    Fix disjoint sets $I_1,\dots,I_d$, fix $\mb F\in \{\mb R,\mb C\}$ and let $r=\prod_{j=1}^d|I_j|$. Let $\mc Z\subseteq \mb F^{I_1}\otimes\dots\otimes \mb F^{I_d}$ be the set of reducible $I_1\times \dots\times I_d$  tensors, and for a non-trivial partition $\{1,\dots,d\}=J_1\cup J_2$, let $\mc Z_{J_1,J_2}$ be the set of $I_1\times\dots\times I_d$ tensors that are reducible with respect to $\{J_1,J_2\}$. Then
    \[\mc Z=\bigcup_{\{J_1,J_2\}} \mc Z_{J_1,J_2},\]
    where the union is over all $2^{d-1}-1$ unordered non-trivial partitions $\{1,\dots,d\}=J_1\cup J_2$.
\end{fact}

The above fact, together with \cref{fact:tensor-to-matrix}, makes \cref{thm:irreducible-variety}(1) nearly immediate (i.e., the observation that the reducible tensors form a variety).
\begin{proof}[Proof of \cref{thm:irreducible-variety}(1)]
By \cref{fact:break-up-by-partitions}, it suffices to prove that each $\mc Z_{J_1,J_2}$ is an affine variety, and by \cref{fact:tensor-to-matrix}, it suffices to show that the set of $r_1\times r_2$ matrices with rank at most 1 can be interpreted as an affine variety in $\mb F^{r_1r_2}$ (i.e., it is the zero locus of a system of polynomial equations). This is well-known to be true: a matrix has rank at most 1 if and only if every $2\times 2$ submatrix has zero determinant.
\end{proof}

Next, we prove \cref{thm:irreducible-variety}(2), classifying the maximal linear subspaces of the variety of reducible tensors. In this proof, we need the simple algebraic geometry fact that linear subspaces are irreducible as affine varieties. This fact also holds in the case $\mb F=\mb R$, even though the underlying field is not algebraically closed (it holds over any infinite field).

\begin{proof}[Proof of \cref{thm:irreducible-variety}(2)]
    Recall that all linear subspaces of $\mb F^r$ are irreducible as affine varieties. 
    So, recalling \cref{fact:break-up-by-partitions}, any linear subspace of $\mc Z=\bigcup_{\{J_1,J_2\}} \mc Z_{J_1,J_2}$ must in fact be a linear subspace of some particular $\mc Z_{J_1,J_2}$.
    That is to say, it suffices to classify the maximal linear subspaces of $\mc Z_{J_1,J_2}$. 
    By \cref{fact:tensor-to-matrix}, this is really the same as classifying the maximal linear subspaces of the variety of $r_1\times r_2$ matrices with rank at most 1. 
    That is to say, it suffices to prove the lemma in the case $d=2$ (recalling that matrices are 2-dimensional tensors).
    
    So, fix disjoint sets $I_1,I_2$, and let $\mc V$ be a linear subspace of the variety $\mc Z$ of $I_1\times I_2$ tensors (matrices) which have rank at most 1. 
    Our objective is to prove that there is a ``column vector'' $\vec u^{\hspace{0.05em}\star}\in \mb F^{I_1}$ such that
    \[\mc V\subseteq \{\vec u^{\hspace{0.05em}\star}\otimes \vec x:\vec x\in \mb F^{I_2}\}\]
    or there is a ``row vector'' $\vec v^{\hspace{0.05em}\star}\in \mb F^{I_2}$ such that
    \[\mc V\subseteq \{\vec y\otimes \vec v^{\hspace{0.05em}\star}:\vec y\in \mb F^{I_1}\}.\]
    Indeed, this will show that we have $\mc V=\{\vec u^{\hspace{0.05em}\star}\otimes \vec x:\vec x\in \mb F^{I_2}\}$ or $\mc V= \{\vec y\otimes \vec v^{\hspace{0.05em}\star}:\vec y\in \mb F^{I_1}\}$, whenever $\mc V$ is a maximal linear subspace of $\mc Z$.
    
    If $\mc V$ contains only the zero matrix, the statement trivially holds. 
    So, we can assume that $\mc V$ contains a nonzero matrix, i.e., a matrix of rank exactly $1$. This means that there are nonzero vectors $\vec u^{\hspace{0.05em}\star}\in \mb F^{I_1}$ and $\vec v^{\hspace{0.05em}\star}\in \mb F^{I_2}$ such that $\vec u^{\hspace{0.05em}\star} \otimes \vec v^{\hspace{0.05em}\star}\in \mc V$.
    \begin{claim}
        Let $T \in \mc V$. Then, $T$ is of the form $\vec u^{\hspace{0.05em}\star}\otimes \vec x$ for some $\vec x\in \mb F^{I_2}$, or $\vec y\otimes \vec v^{\hspace{0.05em}\star}$ for some $\vec y\in \mb F^{I_1}$.
    \end{claim}
    \begin{claimproof}
        $T \in \mc{V} \subseteq \mc{Z}$ has rank at most 1, so we can write $T=\vec u\otimes \vec v$ for some $\vec u\in \mb F^{I_1}$ and $\vec v\in \mb F^{I_2}$.
        We need to prove that $\vec u$ is a multiple of $\vec u^{\hspace{0.05em}\star}$ or $\vec v$ is a multiple of $\vec v^{\hspace{0.05em}\star}$.
    
        We can interpret $\vec u\otimes \vec v+u^{\hspace{0.05em}\star} \otimes \vec v^{\hspace{0.05em}\star}$ as a matrix product $PQ$, where $P$ is the $|I_1|\times 2$ matrix with columns $\vec u,\vec u^{\hspace{0.05em}\star}$ and $Q$ is the $2\times |I_2|$ matrix with rows $\vec v,\vec v^{\hspace{0.05em}\star}$. Since $\mc V$ is a linear space, we know that $\vec u\otimes \vec v+u^{\hspace{0.05em}\star} \otimes \vec v^{\hspace{0.05em}\star}=T+u^{\hspace{0.05em}\star} \otimes \vec v^{\hspace{0.05em}\star}\in \mc V \subseteq \mc{Z}$ has rank at most 1, so at least one of $P$ or $Q$ has rank at most 1 (since the product of a $|I_1|\times 2$ matrix of rank $2$ with a $2\times |I_2|$ matrix or rank 2 always has rank 2 as well). That is to say, $\vec u$ is a multiple of $\vec u^{\hspace{0.05em}\star}$, or $\vec v$ is a multiple of $\vec v^{\hspace{0.05em}\star}$.
    \end{claimproof}
    Now, let $\mc{V}_1:=\{\vec u^{\hspace{0.05em}\star}\otimes \vec x:\vec x\in \mb F^{I_2}\}$ and $\mc{V}_2:=\{\vec y\otimes \vec v^{\hspace{0.05em}\star}:\vec y\in \mb F^{I_1}\}$; they are both affine varieties (in fact, they are both linear subspaces). 
    The above claim shows that $\mc V \subseteq \mc{V}_1 \cup \mc{V}_2$.
    Since $\mc V$ is a linear space, it is irreducible, so $\mc V \subseteq \mc{V}_1$ or $\mc V\subseteq \mc{V}_2$, as desired.
\end{proof}

\section{$k$-multilinear forms}\label{sec:multilinear}
In this section we prove \cref{lem:collapse} and use it to prove \cref{thm:multilinear}.

\begin{proof}[Proof of \cref{lem:collapse}]
    For each $i\in I_d=\{1,\dots,n\}$, let $T_i:I_1\times I_2\times \dots\times I_{d-1}\to \mb F$ be the tensor defined by $T_i(i_1,\dots,i_{d-1})=T(i_1,\dots,i_{d-1},i)$, so we can write $T\vec \xi=\xi_1T_1+\dots+\xi_nT_n$. 
    Recall that we can interpret an $I_1\times I_2\times \dots\times I_{d-1}$ tensor as a vector in $\mb F^{|I_1|\dotsm |I_{d-1}|}$, and let $\mc Z\subseteq \mb F^{|I_1|\dotsm |I_{d-1}|}$ be the variety of reducible $I_1\times I_2\times \dots\times I_{d-1}$ tensors. 
    Then, the event that $T\vec \xi$ is reducible is precisely the event that $\xi_1T_1+\dots+\xi_nT_n\in \mc Z$. 
    
    Also, note that every maximal affine-linear subspace of $\mc Z$ is in fact a linear subspace. To see this, note that multiplying by a nonzero scalar does not affect reducibility, and therefore we have $\lambda\vec v\in \mc Z$ for all $\vec v\in \mc Z$ and $\lambda\in \mb F$. So, for any affine-linear subspace $\mc W\subseteq \mc Z$, the linear span $\{\lambda \vec w:\vec w\in \mc W,\lambda\in \mb F\}$ of the vectors in $\mc W$ is a \emph{linear} subspace in $\mc Z$ that contains $\mc W$.
    
    So, by \cref{thm:FKS}, $\Pr[T\vec \xi\text{ is reducible}]\le n^{-1/2+\varepsilon/r}\le n^{-1/2+\varepsilon}$ (i.e., \cref{D2} holds), or there is a maximal linear subspace $\mc W\subseteq \mc Z$ such that all but $(\varepsilon/r) n$ of the $T_i$ lie in $\mc W$. 
    We may assume the latter property holds.
    
    By \cref{thm:irreducible-variety}, there is a subset $\emptyset\subsetneq J\subsetneq \{1,2,\dots,d-1\}$ and a tensor $T^\star \in \bigotimes_{j\in J} \mb{F}^{I_j}$ such that 
    \[\mc W=\Bigg\{T^\star\otimes T':\; T' \in \bigotimes_{j\in \{1,\dots,d-1\}\setminus J}\mb{F}^{I_j}\Bigg\}.\]
    Suppose without loss of generality that $J=\{1,\dots,j\}$ for some $j \in \{1,\dots,d-2\}$.
    Now, we know that at least a $(1-\varepsilon/r)$-fraction of the $T_i$ lie in $\mc W$, so for at least a $(1-\varepsilon)$-fraction of the $r$-element subsets $I_d'\subseteq I_d$, we have $T_i\in \mc W$ for each $i\in I_d'$. 
    To prove \cref{D1}, it suffices to prove that for all such subsets $I_d'$, the corresponding $r\times \dots \times r$ subtensor $T[I_1,\dots,I_{d-1},I_d']$ is reducible. 
    In fact, this is nearly immediate: for each $i\in I_d'$ we can write $T_i=T^\star\otimes T'_i$, so defining $T':I_{j+1}\times \dots\times I_d\to \mb F$ by $T'(i_{j+1},\dots,i_d)=T_{i_d}(i_{j+1},\dots,i_{d-1})$, we have $T[I_1,\dots,I_{d-1},I_d']=T^\star\otimes T'$.
\end{proof}

Now, we turn to the proof of \cref{thm:multilinear}. It is a direct consequence of the following slightly more general result.
\begin{theorem}\label{theorem: multilinear probability and tensor reducibility}
    Fix $d\ge 2$ and $\varepsilon>0$, and let $\mathbb{F}\in\{\mathbb{R},\mathbb{C}\}$. 
    Let $n$ be sufficiently large (in terms of $\varepsilon,d$), and consider a tensor $T:I_1\times\dots\times I_d \to \mb{F}$ for some partition $\{1,\dots,n\}=I_1\cup \dots\cup I_d$, and the associated $d$-multilinear form  $f \in \mb{F}[x_1,\dots,x_n]$.
    Then, at least one of the following holds:
    \begin{enumerate}[{\bfseries{F\arabic{enumi}}}]
        \item \label{F1} $T$ can be made reducible by changing up to $\varepsilon n^{d}$ entries, or
        \item \label{F2} letting $\vec{\xi}=(\xi_{1},\dots,\xi_{n}) \in \{-1,1\}^n$ be a vector of i.i.d.\ Rademacher random variables, we have
        \[
            \sup_{z\in\mathbb{F}}
            \Pr[f(\xi_{1},\dots,\xi_{n})=z]\le n^{-1+\varepsilon}.
        \]
    \end{enumerate}
\end{theorem}

Our proof of \cref{theorem: multilinear probability and tensor reducibility} is by induction on $d$.
The base case $d=2$ was essentially proved by Costello~\cite{Cos13} (in fact, even with a stronger\footnote{We remark that Costello's notion of ``close to reducible'' (which we do not state here) seems to be genuinely stronger than \cref{F1}, in the sense that we do not know how to directly deduce \cite[Theorem~5]{Cos13} from the $d=2$ case of \cref{theorem: multilinear probability and tensor reducibility}. As discussed in \cref{item:smallness} in \cref{subsec:further-directions}, it may be interesting to investigate further the various different notions of ``smallness'' or ``closeness'', and the implications between them.} notion of ``close to reducible'' than our property \cref{F1}). Specifically, the following theorem is a direct consequence of \cite[Theorem~5]{Cos13}.
\begin{theorem}
\label{thm:costello}
    Fix $\gamma>0$ and $\mb{F} \in \{\mb{R},\mb{C}\}$, and let $r_0$ be sufficiently large in terms of $\gamma$.
    For any $m,m'\in \mb N$, consider a matrix $A \in \mb{F}^{m \times m'}$ such that every row has at least $r\ge r_0$ nonzero entries.
    Let $\vec{\xi}=(\xi_1,\dots,\xi_m)\in \{-1,1\}^m$ and $\pvec \xi=(\xi_1',\dots,\xi_{m'}')\in \{-1,1\}^{m'}$ be independent (column) vectors with i.i.d.\ Rademacher entries, and suppose that there is a function $\varphi:\mb F^{m'}\to \mb F$
    such that
    \[
    \Pr\big[\vec{\xi}^{\,\,T} A \pvec{\xi}=\varphi(\pvec{\xi})\big] \ge r^{-1+\gamma}.\]
    Then, $A$ can be turned into a matrix of rank at most 1 by changing up to $\gamma r(m+m')$ entries.
\end{theorem}
Compared to \cref{theorem: multilinear probability and tensor reducibility}, note that \cref{thm:costello} has an additional assumption requiring many nonzero entries in every row of the relevant matrix. We take a moment to state a slight strengthening of the $d=2$ case of \cref{theorem: multilinear probability and tensor reducibility} (this strengthening will be useful later in the paper) and explain how to deduce this statement from Costello's work (using an argument as in \cite[Remark~1]{Cos13}).
\begin{theorem}
\label{thm:costello-mod}
    Fix $\eps>0$ and $\mb{F} \in \{\mb{R},\mb{C}\}$. Consider $m,m'\in \mb N$ such that $m+m'$ is sufficiently large in terms of $\eps$, and consider a matrix $A \in \mb{F}^{m \times m'}$. Then, at least one of the following holds:
     \begin{enumerate}[{\bfseries{F\arabic{enumi}}'}]
        \item \label{F1'} $A$ can be made into a rank-1 matrix by changing up to $\varepsilon (m+m')^{2}$ entries, or
        \item \label{F2'} letting $\vec{\xi}=(\xi_1,\dots,\xi_m)\in \{-1,1\}^m$ be a vector of i.i.d.\ Rademacher random variables, and independently letting $\pvec \xi=(\xi_1',\dots,\xi_{m'}')\in \{-1,1\}^{m'}$ be a vector of i.i.d.\ Rademacher or lazy Rademacher random variables,  we have
        \[
            \Pr[\vec{\xi}^{\,\,T} A\pvec{\xi}=\psi(\pvec \xi)]\le (m+m')^{-1+\varepsilon}.
        \]
        for any function $\psi:\mb F^{m'}\to \mb F$.
    \end{enumerate}
\end{theorem}

\begin{proof}[Proof]
    Let $n=m+m'$ and assume that $f$ does not satisfy \cref{F2'}; that is to say, 
    \[
        \Pr\left[\vec{\xi}^{\,\,T}A\,\pvec\xi = \psi(\pvec \xi) \right]> n^{-1+\varepsilon}.
    \]
    for some function $\psi:\mb F^{m'}\to \mb F$. 
    We will show that \cref{F1'} holds.
    In other words, we will show that we can change at most $\varepsilon n^2$ entries of $A$ to obtain a matrix with rank at most 1.
    
    Let $I$ be the set of $i \in \{1,\dots,m\}$ such that row $i$ of $A$ has at least $\varepsilon n$ nonzero entries, and let $I^*=\{1,\dots,m\}\setminus I$. 
    Let $A'\in \mb F^{|I|\times m'}$ be the submatrix of $A$ containing the rows indexed by $I$, and let $A^*\in \mb F^{|I^*|\times m'}$ be the submatrix of $A$ containing the rows indexed by $I^*$. 
    Define the random function $\varphi:\{-1,1\}^{m'}\to \mb{F}$ by $\varphi(\vec x)=\psi(\vec x)- \vec{\xi}[I^*]^TA^*\vec{x}$ (this function depends on the random vector $\vec \xi[I^*]$).
    So, we have \[
         n^{-1+\varepsilon} 
         < \Pr\Big[\vec{\xi}^{T}\! A\, \pvec{\xi}=\psi(\pvec\xi) \Big]
         = \Pr\Big[\vec{\xi}[I]^{T}\! A'\, \pvec{\xi}=\varphi(\pvec{\xi})\Big].
    \]
    Consequently, there exists an outcome of $\vec{\xi}[I^*]$ such that \[
        \Pr\Big[ \vec{\xi}[I]^{T}\! A'\, \pvec{\xi}=\varphi(\pvec{\xi})\;\Big|\;\vec{\xi}[I^*] \Big]
        >n^{-1+\varepsilon} >(\varepsilon n)^{-1+\varepsilon/2}.
    \]
    Applying \cref{thm:costello} (with $r=\varepsilon n$ and $\gamma=\varepsilon/2$) in the conditional probability space where we condition on this particular outcome of $\vec{\xi}[I^*]$, we see that we can change at most $(\varepsilon/2)(\varepsilon n)(|I|+m')\le (\varepsilon n)(|I|+m')$ entries of $A'$ to obtain a matrix of rank at most 1. 
    Also, $A^*$ has at most $(\varepsilon n)|I^*|$ nonzero entries by definition, so we can make $A^*$ into the all-zero matrix by changing at most $(\varepsilon n)|I^*|=(\varepsilon n)(m-|I|)$ entries. 
    Recalling the definitions of $A'$ and $A^*$, we see that we can change at most $(\varepsilon n)(|I|+m')+(\varepsilon n)(m-|I|) =(\varepsilon n)(m+m')=\varepsilon n^2$ entries of $A$ to obtain a matrix with rank at most 1, as desired.    
\end{proof}

We now complete the proof of \cref{theorem: multilinear probability and tensor reducibility} by presenting the inductive step. Briefly speaking, we first expose $\vec{\xi}[I_d]$, and obtain a $(d-1)$-multilinear form in the variables $x_i$ for $i\in I_1\cup\dots\cup I_{d-1}$.
Equivalently, we ``collapse'' $T$, based on $\vec{\xi}[I_d]$, to a $(d-1)$-dimensional tensor $T'$.
Then, we will apply \cref{lem:collapse} if $T'$ tends to be close to reducible and apply the inductive hypothesis to $T'$ otherwise. 
\begin{proof}[Proof of \cref{theorem: multilinear probability and tensor reducibility}]
    We proceed by induction on $d$; the base case $d=2$ is handled by \cref{thm:costello-mod}.
    Fix $d \ge 3$ and suppose the statement holds for all $(d-1)$-multilinear forms. 
    Recall that $f(x_1,\dots,x_n)=\sum_{i_1\in I_1}\dots\sum_{i_d\in I_d} T(i_1,\dots,i_d) \,x_{i_1}\dots x_{i_d}$ for all $\vec{x} \in \mb{F}^n$.
    Our goal is to prove that \cref{F1} or \cref{F2} holds.
    
    First, recall that we are assuming that $n$ is large with respect to $d,\varepsilon$.
    We may also assume that
    \begin{equation}\label{eq:I'-big}
        |I_1|\cdots|I_d|\ge \varepsilon n^d.
    \end{equation}
    Indeed, otherwise, $T$ has at most $\varepsilon n^d$ entries, and \cref{F1} trivially holds (as we can change all the entries of $T$ to zero to make $T$ reducible).
    
    Now, let $I'=I_1\cup \dots\cup I_{d-1}$, and note that for any outcome of $\vec \xi[I_d]$, we can view $f$ as a $(d-1)$-multilinear form $f'\in \mb F[(x_i)_{i\in I'}]$, in the variables $x_i$ for $i\in I'$. Note that the coefficient tensor of $f'$ is precisely $T\vec \xi[I_d]$, in the notation of \cref{lem:collapse}. For convenience we write $T'=T\vec \xi[I_d]$. The high-level idea of the proof is to fix an outcome of $\vec \xi[I_d]$ (thereby fixing $f'$), and to apply the induction hypothesis to $f'$. When $f'$ is far from being reducible, this gives us the desired probability bound, so the main challenge is to carefully study the event that $f'$ is close to reducible. Actually, we need to break down this event further, depending on whether $f'$ is close to the zero polynomial or close to some other (far-from-zero) reducible polynomial. Specifically, let $\delta=(\varepsilon/2)^{2^{d-1}}$ and consider the following two events.
    \begin{itemize}
        \item Let $\varepsilon_1=\varepsilon^3/32$ and let $\mc E_1$ be the event that $T'$ has at most $\varepsilon_1 |I'|^{d-1}$ nonzero entries.
        \item Let $\varepsilon_2=2^{-(d-1)^2}\delta\varepsilon/4$, and let $\mc E_2$ be the event that $T'$ can be made reducible by changing at most $\varepsilon_2 |I'|^{d-1}$ entries.
    \end{itemize}

    \medskip\noindent\textbf{Decomposing the anticoncentration probability.}
    Let $\mc E_1^{\mr c},\mc E_2^{\mr c}$ be the complements of $\mc E_1$ and $\mc E_2$. Recall that our goal is to prove that \cref{F1} or \cref{F2} holds. 
    Assume that \cref{F2} fails, that is to say, for some $z\in \mb F$,
    \[\Pr\big[f(\vec{\xi})=z \big]>n^{-1+\varepsilon}.\]
    The probability $\Pr[f(\vec \xi)=z]$ can be decomposed as
    \[\Pr\big[f(\vec \xi)=z\,\big|\,\mc E_1\big]\cdot \Pr\big[\mc E_1\big]\;+\;\Pr\big[f(\vec \xi)=z\,\big|\,\mc E_1^{\mr c}\cap \mc E_2\big]\cdot \Pr\big[\mc E_1^{\mr c}\cap \mc E_2\big]\;+\;\Pr\big[f(\vec \xi)=z\,\big|\,\mc E_1^{\mr c}\cap \mc E_2^{\mr c}\big]\cdot \Pr\big[\mc E_1^{\mr c}\cap \mc E_2^{\mr c}\big].\]
    and we deduce
    \begin{equation}\label{eq:3-terms}
        \Pr\big[f(\vec \xi)=z\big]\;\le\;\Pr\big[\mc E_1\big]\;+\;
        \Pr\big[f(\vec \xi)=z\,\big|\,\mc E_1^{\mr c}\cap \mc E_2\big]\cdot\Pr\big[\mc E_2\big]\;+\;
        \Pr\big[f(\vec \xi)=z\,\big|\,\mc E_1^{\mr c}\cap \mc E_2^{\mr c}\big].
    \end{equation}
    Now, if $\vec \xi[I_d]$ satisfies $\mc E_1^{\mr c}\cap \mc E_2^{\mr c}\subseteq \mc E_2^{\mr c}$, then $T'$ cannot be made reducible by changing at most $\varepsilon_2 |I'|^{d-1}$ coefficients. So, if we fix an outcome of $\vec \xi[I_d]$ satisfying $\mc E_2^{\mr c}$, then we can condition on this outcome of $\vec \xi[I_d]$ and apply the induction hypothesis to $f'$ and $T'$ in the resulting conditional probability space, to obtain
    \[
         \Pr\big[f(\vec \xi)=z\,\big|\,\vec \xi[I_d]\big]
        =\Pr\big[f'(\vec \xi[I'])=z\,\big|\,\vec \xi[I_d]\big]
        \le |I'|^{-1+\varepsilon_2}
        \le n^{-1+\varepsilon}/3.
    \]
    (Here we used that $|I'|\ge \varepsilon n$, which follows from \cref{eq:I'-big}).
    Averaging over $\vec \xi[I_d]$ satisfying $\mc E_2^{\mr c}$, we deduce that
    \[
        \Pr\big[f(\vec \xi)=z\,\big|\,\mc E_2^{\mr c}\big]\le n^{-1+\varepsilon}/3.
    \]
    Similarly, if we fix an outcome of $\vec \xi[I_d]$ satisfying $\mc E_1^{\mr c}\cap \mc E_2\subseteq \mc E_1^{\mr c}$, then $f'$ has at least $\varepsilon_1 |I'|^{d-1}$ nonzero coefficients, so \cref{thm:polynomial-LO} yields
    \[
        \Pr\big[f(\vec \xi)=z\,\big|\, \vec \xi[I_d]\big] = \Pr\big[f(\vec \xi)=z\,\big|\, \vec \xi[I_d]\big]
        \le |I'|^{-1/2+\varepsilon_1}\le n^{-1/2+\varepsilon/2}.
    \]
    Averaging over all $\vec \xi[I_d]$ satisfying $\mc E_1^{\mr c}\cap \mc E_2$, we obtain
    \[
        \Pr\big[f(\vec \xi)=z\,\big|\,\mc E_1^{\mr c}\cap \mc E_2\big]
        \le |I'|^{-1/2+\varepsilon_1}\le n^{-1/2+\varepsilon/2}.
    \]
    Now, from \cref{eq:3-terms}, it follows that
    \[
        n^{-1+\varepsilon}<\Pr[f(\vec \xi)=z]
        \le\Pr[\mc E_1]+n^{-1/2+\varepsilon/2}\cdot \Pr[\mc E_2]+n^{-1+\varepsilon}/3.
    \]
    Thus, we must have $\Pr[\mc E_1]>n^{-1+\varepsilon}/3$ or $\Pr[\mc E_2]>n^{-1/2+\varepsilon/2}/3$. 
    In both cases, we will prove that $T$ can be made reducible by changing $\varepsilon|I_1|\cdots|I_d|\le \varepsilon n^d$ of its entries; that is, \cref{F1} holds. 
    In the first case we will use \cref{thm:Halasz} (i.e., Hal\'asz' inequality), and in the second case we will use \cref{lem:collapse}.
    
    \medskip\noindent\textbf{Case 1: $f'$ is likely to be close to the zero polynomial.}
    In this case, we suppose that $\Pr[\mc E_1]>n^{-1+\varepsilon}/3$. 
    Let $Q=I_1\times I_2\times\dots\times I_{d-1}$; it follows from \cref{eq:I'-big} that $|Q| > \varepsilon n^{d-1}$. 
    Recalling \cref{fact:tensor-to-matrix}, let us view $T$ as a $Q \times I_d$ matrix, which we refer to as $M$.
    
    Translating from tensor to matrix language, in this case we are assuming that with probability exceeding $n^{-1+\varepsilon}/3$, the number of nonzero entries in the matrix-vector product $M\vec \xi[I_d]$ is at most $\varepsilon_1 |I'|^{d-1}\le (\varepsilon^2/32)|Q|$.
    By \cref{lemma: tuple-counting} (with $r=2$), applied to the set of zero entries of  $M\vec \xi[I_d]$, all but at most a $\varepsilon^2/8$ fraction of the pairs of distinct $\vec i,\pvec i\in Q$ satisfy 
    \[
        \Pr\big[(M\vec{\xi}[I_d])_{\vec i} 
        = (M\vec{\xi}[I_d])_{\pvec i} = 0 \big]
        =\Pr\big[M[\{\vec{i},\pvec{i}\},I_d]\,\vec{\xi}[I_d] = \vec 0 \,\big]
        > (n^{-1+\varepsilon}/3)/2 > |I_d|^{-1+\varepsilon/2}.
    \]
    (Here we write $M[\{\vec{i},\pvec{i}\},I_d]$ for the $2\times |I_d|$ matrix containing just the rows of $M$ indexed by $\vec i,\pvec i$ and we use that $|I_d|>\varepsilon n$, which follows from \cref{eq:I'-big}.)
    For each such $\vec i,\pvec i$, by \cref{thm:Halasz}, $M[\{\vec{i},\pvec{i}\},I_d]$ contains at most $O(|I_d|^{1-\varepsilon/2})$ disjoint nonsingular $2\times 2$ submatrices, so by \cref{fact: disjoint and non-disjoint nonsingular submatrices} it has at most $O(|I_d|^{1-\varepsilon/2})|I_d|<(\varepsilon^2/8)\binom{|I_d|}{2}$ nonsingular $2\times 2$ submatrices in total.
    Summing over all these $\vec i,\pvec i$, the fraction of $2\times 2$ submatrices of $M$ which are nonsingular is less than $\varepsilon^2/8+\varepsilon^2/8\le \varepsilon^2/4$. 
    So, by the $d=2$ case of \cref{thm:tensor-property-testing}, we can change at most an $\varepsilon$-fraction of the entries of $M$ to obtain a matrix with rank at most 1. 
    Translating back into tensor language, $T$ can be made reducible by changing at most an $\varepsilon$-fraction of its entries. This proves \cref{F1}, as desired.
    
    \medskip\noindent\textbf{Case 2: $f'$ is likely to be close to reducible.}
    In this case, we suppose that $\Pr[\mc E_2]>n^{-1/2+\varepsilon/2}/3$.
    This means, with probability exceeding $n^{-1/2+\varepsilon/2}/3$, there is a reducible tensor $\tilde T:I_1\times \cdots\times I_{d-1}\to \mb F$ such that $T'$ differs from $\tilde T$ in at most $\varepsilon_2 |I'|^{d-1}\le \varepsilon_2 n^{d-1}\le 2^{-(d-1)^2}(\delta/4)|I_1|\cdots|I_{d-1}|$ entries (here, we used that $|I_1|\cdots|I_{d-1}| \ge \varepsilon n^{d-1}$, which follows from \cref{eq:I'-big}).
    When this occurs, at least a $(1-\delta/4)$ fraction of the $2^{d-1}\times \cdots \times 2^{d-1}$ subtensors of $T'$ are the same as the corresponding $2^{d-1}\times \dots \times 2^{d-1}$ subtensor of $\tilde T$, so are reducible. (Here we used that a $2^{d-1}\times \dots \times 2^{d-1}$ subtensor has $2^{(d-1)^2}$ entries). 
    By \cref{lemma: tuple-counting} (with $r=1$) applied to the set of reducible $2^{d-1}\times\dots\times 2^{d-1}$ subtensors of $T'$, we see that for at least a $(1-\delta/2)$-fraction of choices of $2^{d-1}$-element subsets $Q_1\subseteq I_1,\dots,Q_{d-1}\subseteq I_{d-1}$, we have
    \[\Pr\big[T'[Q_1,\dots,Q_{d-1}]\text{ is reducible}\big]>(n^{-1/2+\varepsilon}/3)/2\ge|I_d|^{-1/2+\delta/2}.\]
    (here we are using that $|I_d|\ge \varepsilon n$, which follows from \cref{eq:I'-big}).
    For each such $Q_1,\dots,Q_{d-1}$, we apply \cref{lem:collapse}: noting that $T'[Q_1,\dots,Q_{d-1}]=T[Q_1,\dots,Q_{d-1},I_d]\,\vec \xi[I_d]$ in the notation of \cref{lem:collapse}, we see that at least a $(1-\delta/2)$-fraction of the $2^{d-1}\times \dots\times 2^{d-1}$ subtensors of $T[Q_1,\dots,Q_{d-1},I_d]$ are reducible.
    Summing over all these $Q_1,\dots,Q_{d-1}$, we see that the fraction of $2^{d-1}\times\dots\times2^{d-1}$ subtensors of $T$ that are reducible is at least $(1-\delta/2)(1-\delta/2)>1-\delta$.
    In the end, \cref{thm:tensor-property-testing} implies that $T$ can be made reducible by changing at most $\varepsilon|I_1|\cdots|I_d|\le \varepsilon n^d$ of its entries. This proves \cref{F1}, as desired.
\end{proof}

\section{A local-to-global lemma for low-rank symmetric matrices}\label{sec:matrix-property-testing}
In this section we prove \cref{lemma: close to symmetric low rank}:
if almost all $r\times r$ submatrices of a symmetric matrix $A$ are singular, then $A$ can be made to have rank less than $r$ by changing only a few entries, \emph{while maintaining symmetry}. Recall that $\|A\|_0$ denotes the number of nonzero entries in a matrix $A$.

We split the proof of \Cref{lemma: close to symmetric low rank} into the following two lemmas.
Roughly speaking, we first approximate $A$ by a matrix of rank less than $r$, which may not be symmetric.
However, this matrix is necessarily \emph{almost} symmetric. In the second step we adjust it to be symmetric while maintaining the rank. Most of the difficulty lies in the second step.

\begin{lemma}\label{lem:close-non-symmetric}
    Consider integers $n, m \ge r \ge 1$, let $\alpha \in [0,1]$ and let $\mb{F}$ be a field. Let $A\in \mb{F}^{n \times m}$ be a matrix such that all but at most an $\alpha$-fraction of its $r \times r$ submatrices are singular.
    Then, there exists a matrix $B \in \mb{F}^{n \times m}$ of rank less than $r$ such that $\|A-B\|_0 \le \alpha^{1/r} nm$.
\end{lemma}

\begin{lemma}\label{lem:fix-symmetry}
    Consider integers $n \ge q \ge 1$, let $\rho \in [0,1]$ and let $\mb{F}$ be a field.
    Let $A \in \mb{F}^{n \times n}$ be a matrix of rank at most $q$ such that $\|A-A^T\|_0 \le \rho n^2$.
    Then, there exists a symmetric matrix $B\in \mb{F}^{n \times n}$ of rank at most $q$ such that $\|A-B\|_0 \le  O(q^4\rho n^2)$.
\end{lemma}

\Cref{lemma: close to symmetric low rank} follows immediately from \Cref{lem:close-non-symmetric} and \Cref{lem:fix-symmetry}, as follows.
\begin{proof}[Proof of \Cref{lemma: close to symmetric low rank}] 
    By \cref{lem:close-non-symmetric}, we first obtain a matrix $B\in \mb{F}^{n\times n}$ with rank less than $r$ and $\|A-B\|_0 \le \alpha^{1/r}n^2$. 
    Since $A$ is symmetric, it holds that 
    $$
        \|B-B^T\|_0
        \le \|B-A^T\|_0+\|A^T-B^T\|_0
        =\|B-A\|_0+\|A-B\|_0 \le 2\alpha^{1/r}n^2.
    $$
    Then, \cref{lem:fix-symmetry} (with $q=r-1$) guarantees a symmetric matrix $A' \in \mb{F}^{n\times n}$ of rank less than $r$ such that $\|B-A'\|_0\le O(r^4\alpha^{1/r}n^2)$.
    Also, $\|A-A'\|_0\le \|A-B\|_0+\|B-A'\|_0\le O(r^4\alpha^{1/r}n^2)$, as desired.
\end{proof}

The rest of this section is devoted to proving \cref{lem:close-non-symmetric} and \cref{lem:fix-symmetry}.
We first give a short proof of \cref{lem:close-non-symmetric}. For the rest of this section, we use the notation $[n]=\{1,\dots,n\}$.
\begin{proof}[Proof of \cref{lem:close-non-symmetric}]
    First, find the minimum $k \ge 1$ such that all but at most an $\alpha^{k/r}$-fraction of the $k\times k$ submatrices of the matrix $A$ are singular.
    Clearly, $k \le r$ by assumption.
    If $k=1$, then $\|A\|_0\le \alpha^{1/r} nm$, and we are done (by taking $B=0$).
    From now on, we assume that $2\le k\le r$.
    
    Second, we claim that there exist sets $I\subseteq [n], J\subseteq [m]$, both of size $k-1$, such that $A[I,J]$ is nonsingular and $A[I\cup\{i\},J\cup \{j\}]$ is nonsingular for at most $\alpha^{1/r}(n-k+1)(m-k+1)$ choices of $(i,j) \in ([n]\setminus I)\times ([n]\setminus J)$.
    Indeed, suppose not. Then, whenever $A[I,J]$ is nonsingular, more than $\alpha^{1/r}(n-k+1)(m-k+1)$ choices of $(i,j) \in ([n]\setminus I)\times ([n]\setminus J)$ are such that $A[I\cup\{i\},J\cup \{j\}]$ is nonsingular.
    In addition, the minimality of $k$ implies there are more than $\alpha^{(k-1)/r}\binom{n}{k-1}\binom{m}{k-1}$ choices of $I,J$ such that $A[I,J]$ is nonsingular.
    In total, we acquire more than \[
        \alpha^{(k-1)/r}\binom{n}{k-1}\binom{m}{k-1}\cdot \alpha^{1/r}(n-k+1)(m-k+1)
        =k^2\alpha^{k/r}\binom{n}{k}\binom{m}{k}
    \] choices of $I,J,i,j$ such that $A[I\cup\{i\},J\cup \{j\}]$ is a nonsingular $k\times k$ submatrix of $A$. Since each submatrix is counted at most $k^2$ times, the number of nonsingular $k\times k$ submatrices of $A$ is larger than $\alpha^{k/r}\binom{n}{k}\binom{m}{k}$, contradicting our definition of $k$.

    We have established that there is some choice of $I$ and $J$ such that $A[I,J]$ is nonsingular, and $A[I\cup\{i\},J\cup \{j\}]$ is nonsingular for at most $\alpha^{1/r}(n-k+1)(m-k+1)$ choices of $(i,j)$. 
    Fix such $I,J$, and let $B\in  \mb{F}^{n \times m}$ be the unique matrix with the same row space as $A[I,[m]]$, such that $B[[n],J]=A[[n],J]$. 
    (In a bit more detail: since $A[I,J]$ is invertible, for each row $A[i,[m]]$ of $A$ there is a unique vector in the row space of $A[I,[m]]$ which agrees with $A[i,[m]]$ in the coordinates indexed by $J$; we take that to be the corresponding row of $B$.)
    Note that $B[I,[m]]=A[I,[m]]$ and $B[[n],J]=A[[n],J]$, and furthermore 
    $\rank B= \rank (A[I,[m]])= |I|=k-1<r$.
    
    Now, take any $(i,j) \in ([n]\setminus I)\times ([n]\setminus J)$ such that $A[I\cup\{i\},J\cup \{j\}]$ is singular; there are at least $(1-\alpha^{1/r})(n-k+1)(m-k+1)$ such $(i,j)$, thanks to the choice of $(I,J)$.
    Consider the two matrices $A[I\cup\{i\},J\cup \{j\}]$ and $B[I\cup\{i\},J\cup \{j\}]$.
    Their ranks are both less than $k$, i.e.\ $\det(A[I\cup\{i\},J\cup \{j\}])=\det(B[I\cup\{i\},J\cup \{j\}])=0$.
    Hence, $\det(A[I\cup\{i\},J\cup \{j\}])-\det(B[I\cup\{i\},J\cup \{j\}])=0$.
    As discussed above, these two matrices are identical except that, possibly, $A[i,j]\neq B[i,j]$.
    This means that $$0=\big|\det(A[I\cup\{i\},J\cup \{j\}])-\det(B[I\cup\{i\},J\cup \{j\}])\big|=\big|\det(A[I,J])(A[i,j]-B[i,j])\big|.$$
    But $\det(A[I,J])\neq 0$ because $A[I,J]$ is nonsingular.
    So, $A[i,j]=B[i,j]$.
    Recall also that $A[i,j]=B[i,j]$ whenever $i \in I$ or $j \in J$.
    Therefore, $\|A-B\|_0 \le \alpha^{1/r}(n-k+1)(m-k+1)<\alpha^{1/r} nm$.
\end{proof}

\cref{lem:fix-symmetry} is more complicated.
Our strategy is as follows. First, we pass from $A$ to a large submatrix $A[I,I]$, whose rank is ``robust'' (i.e., the rank is maintained whenever we delete a few rows and columns). Then, we show how to use this robustness to approximate $A[I,I]$ with a symmetric matrix with the same row space. Finally,  we show how this approximation can be extended to all of $A$.

The first of these steps (passing to a large submatrix whose rank is robust) is rather simple, as follows.
\begin{lemma}\label{lemma: robust rank in n by n matrices}
    Consider integers $n \ge q \ge 1$, let $\gamma > 0$, and let $\mb{F}$ be a field.
    Let $A \in \mb{F}^{n\times n}$ be a matrix with rank at most $q$.
    Then, there exists $k\in\{0,\dots,q\}$ and $I \subseteq [n]$ of size $|I| \ge (1-q\gamma )n$ such that $\rank(A[I',I'])=\rank(A[I,I])=k$ for all $I' \subseteq I$ of size $|I'| \ge |I| - \gamma n$.
\end{lemma}
\begin{proof}
    Take the minimum $k\in\{0,\dots,q\}$ such that $\rank(A[I,I]) \le k$ holds for some $I \subseteq [n]$ of size $|I| \ge (1-(q-k)\gamma) n$; to see that such $k$ exists, consider $k=q$.
    We claim this $k$ and $I$ satisfy the conclusion of the lemma.
    Clearly, $|I| \ge (1-q\gamma)n$.
    Fix an arbitrary $I^\prime\subseteq I$ with $|I^\prime| \ge |I|-\gamma n$. 
    Then, since $|I'| \ge (1-(q-k+1)\gamma)n$, the minimality of $k$ implies that $\rank(A[I',I']) \ge k$. That is to say, $k \ge \rank(A[I,I]) \ge \rank(A[I',I'])\ge k$, so $\rank(A[I',I'])=\rank(A[I,I])=k$.
\end{proof}

The next lemma shows that when a matrix ``has its rank robustly'', we can approximate it by a symmetric matrix with the same row space. The proof approach is similar to the approach of the second and third authors in \cite[Section 5.7]{KS20}.

\begin{lemma}  \label{lemma: symmetry for robust matrices}
    Consider integers $n \ge 1$ and $r \ge 0$, let $\rho \in [0,1]$ and let $\mb{F}$ be a field. 
    Let $A \in \mb{F}^{n\times n}$ be a matrix with $\|A-A^T\|_0 \le \rho n^2$, such that $\rank(A[I,I])=r$ for all $I\subseteq [n]$ of size $|I| \ge (1-r\sqrt{\rho})n$.
    Then, there exists a symmetric matrix $B \in \mb{F}^{n \times n}$ such that $\|A-B\|_0 \le (r^2+1)\rho n^2$, and such that the row spaces of $A$ and $B$ are the same.
\end{lemma}
\begin{proof}
\newcommand{\Ib}{{I_{\text{bad}}}}
    Note that in the case $r=0$ we must have $A = 0$, and so we can take $B = 0$. So let us from now on assume $r\ge 1$.
    
    For $i,j\in [n]$, say that the pair $(i,j)$ is \emph{bad} if $a_{ij}\neq a_{ji}$.
    For each $1 \le i \le n$, write $N(i):=\{1\le j \le n: a_{ij}\neq a_{ji}\}$.
    Also, define $\Ib$ to be the set of indices $i$ such that $|N(i)|\ge \sqrt{\rho}n$; hence, \[|\Ib| \le \|A-A^T\|_0/\sqrt{\rho} n\le \sqrt{\rho} n.\]
    First, we claim that there exists $V \subseteq [n]\setminus \Ib$ of size $r$ such that $A[V,V]$ is symmetric and $\rank(A[[n],V])=r$.
    To this end, we iteratively find, for each $\ell\in[r]$, a sequence $v_1,\dots,v_\ell\in [n] \setminus \Ib$ such that $\rank(A[[n],\{v_1,\dots,v_\ell\}])=\ell$ and such that no $(v_i,v_j)$ is a bad pair.
    Suppose $\ell\in\{0,1,\dots,r-1\}$ and we have found $v_1,\dots,v_\ell$ with these properties.
    For all $i\in\{1,\dots,\ell\}$, we have $|N(v_i)| \le \sqrt{\rho}n$ because $v_i \notin \Ib$.
    Thus, $\tilde{I}:= [n] \setminus (\Ib\cup \bigcup_{i=1}^\ell N(v_i))$ has size $|\tilde{I}| \ge n-\sqrt{\rho}n-\ell\sqrt{\rho}n \ge (1-r\sqrt{\rho})n$.
    By our assumption on $A$, $\rank(A[[n],\tilde{I}])\ge \rank(A[\tilde{I},\tilde{I}])=r>\ell$, so there exists $v_{\ell+1} \in \tilde{I}$ such that $\rank(A[[n],\{v_1,v_2,\dots,v_{\ell+1}\}])=\ell+1$.
    In addition, for each $i\in [\ell]$, the pair $(v_i,v_{\ell+1})$ is not bad because $v_{\ell+1} \in \tilde{I} \subseteq [n]\setminus N(v_i)$.
    This completes the iteration step. Having found $v_1,\dots,v_r$ at the end of the iterative procedure, we set $V=\{v_1,\dots,v_r\}$, and note that $A[V,V]$ is symmetric since none of the pairs $(v_i,v_j)$ is bad.

    Fix $V$ as above. We next claim that $\rank(A[V,V])=r$, i.e., $A[V,V]$ is invertible.
    To see this, first note that since $\rank(A[[n],V])=r=\rank(A)$, the column vectors of $A[[n],V]$ span those of $A=A[[n],[n]]$.
    Let $I := [n]\setminus \bigcup_{v \in V} N(v)$; it holds that
    \begin{equation} \label{eq:I-size}
        |I|\ge n- \sum_{i=1}^r |N(v_i)|\ge (1-r\sqrt \rho) n.
    \end{equation}
    The above discussion indicates that the column vectors of $A[I,V]$ span those of $A[I,I]$, so $\rank(A[I,V]) \ge \rank(A[I,I])$.
    Similarly, we have $\rank(A[V,V]) \ge \rank(A[V,I])$.
    In addition, the definition of $I$ guarantees that $a_{vi}=a_{iv}$ for all $v \in V$ and all $i \in I$. In short, $A[V,I]=A[I,V]^T$.
    So, $$\rank(A[V,V]) \ge \rank(A[V,I])=\rank(A[I,V])\ge \rank(A[I,I])=r,$$ where the last inequality follows as $|I| \ge (1-r\sqrt{\rho})n$. That is to say, $\rank(A[V,V])=r$, i.e., $A[V,V]$ is invertible.
    
    Now, take $B:=A[V,[n]]^T A[V,V]^{-1} A[V,[n]]$. Note that $B$ is symmetric, and $\rank(B) \le |V|=r$.
    By the preceding discussion, \[B[V,V]=A[V,V]^T A[V,V]^{-1} A[V,V]=A[V,V]\] has rank $r$, so $\rank(B)=r$.
    Also, the row vectors of $B$ are spanned by those of $A$.
    This, plus the fact that $\rank(B)=r=\rank(A)$, demonstrates that the row space of $B$ is the same as that of $A$.
    
    To complete the proof, it remains to show that $\|A-B\|_0 \le (r^2+1)\rho n^2$.
    Since $r = \rank(A)\ge \rank(A[V,[n]])\ge \rank(A[V,V])=r$, we see that $A[V,[n]]$ has the same row space as $A$, 
    i.e., for each $i\in [n]$, we have $A[i,[n]]=\vec{x}_i A[V,[n]]$ for some  $\vec{x}_i \in \mb{F}^{V}$.
    In particular, $A[i,V]=\vec{x}_i A[V,V]$, i.e. $\vec{x}_i=A[i,V]A[V,V]^{-1}$.
    Hence, $A[i,[n]]=A[i,V]A[V,V]^{-1}A[V,[n]]$.
    By considering all $i \in I$, we deduce \begin{equation}
        A[I,[n]]=A[I,V]A[V,V]^{-1}A[V,[n]]=A[V,I]^T A[V,V]^{-1}A[V,[n]]= B[I,[n]].\label{eq:I-n-same}
    \end{equation}
    Writing $I^{\mr c}=[n]\setminus I$, note that for any $i \in I^{\mr c}$ and $j \in I$, the above equality implies that $b_{ij}=b_{ji}=a_{ji}$. So, we have $b_{ij}\neq a_{ij}$ if and only if $a_{ij}\neq a_{ji}$. 
    We deduce that
    \begin{equation}\big\|A[I^{\mr c},I]-B[I^{\mr c},I]\|_0 \le \|A-A^T\|_0\le \rho n^2.\label{eq:Ic-I-close}\end{equation}
    Considering separately the subsets of entries indexed by $I\times [n]$, $I^{\mr c}\times I$ and $I^{\mr c}\times I^{\mr c}$, using \cref{eq:I-size,eq:I-n-same,eq:Ic-I-close}, we conclude
    \[
        \|A-B\|_0
        \le \big\|A[I,[n]]-B[I,[n]]\big\|_0+\big\|A[I^{\mr c},I]-B[I^{\mr c},I]\big\|_0+|I^{\mr c}|^2
        \le (r^2+1)\rho n^2.\qedhere
    \]
\end{proof}
Before completing the proof of \cref{lem:fix-symmetry}, we record a necessary inequality on matrix ranks.
\begin{lemma}\label{lemma: rank inequality of matrix blocks}
    Consider a matrix $A \in \mb{F}^{n \times n}$ over any field $\mb F$ (for any integer $n\ge 1$). Then, for any $I,J\subseteq [n]$, we have $$\rank(A)\ge \rank(A[I,[n]])+\rank(A[[n],J])-\rank(A[I,J]).$$
\end{lemma}
\begin{proof}
    Write $\Delta:=\rank(A[[n],J])-\rank(A[I,J]) \ge 0$.
    There exist distinct $i_1,i_2,\dots,i_\Delta \in [n]\setminus I$ such that for each $k \in \{1,\dots,\Delta\}$, the row vector $A[i_k,J]$ is not spanned by the row vectors of $A[I,J]$ along with $A[i_1,J],\dots,A[i_{k-1},J]$.
    Hence, for each $k\in\{1,\dots,\Delta\}$, the row vector $A[i_k,[n]]$ is not spanned by the row vectors of $A[I,[n]]$ along with $A[i_1,[n]],\dots,A[i_{k-1},[n]]$.
    In other words, $\rank(A[I\cup\{i_1,\dots,i_\Delta\},[n]])\ge \rank(A[I,[n]])+\Delta$.
    This shows $\rank(A) \ge \rank(A[I,[n]])+\Delta$, as desired.
\end{proof}

We now complete the proof of \cref{lem:fix-symmetry}.
\begin{proof}[Proof of \cref{lem:fix-symmetry}]
    We first apply \cref{lemma: robust rank in n by n matrices} with $\gamma=q\sqrt{\rho}$ to obtain  $k\in\{0,\dots,q\}$ and $I \subseteq [n]$ of size $|I| \ge (1-q^2\sqrt{\rho} )n$ such that $\rank(A[I',I'])=\rank(A[I,I])=k$ for all $I' \subseteq I$ of size $|I'| \ge |I| - q\sqrt{\rho} n$.
    Write $\rank(A[[n],I])=k+\Delta_1$ and $\rank(A[I,[n]])=k+\Delta_2$; clearly, $\Delta_1,\Delta_2 \ge 0$.
    According to \cref{lemma: rank inequality of matrix blocks}, $q\ge \rank(A) \ge k+\Delta_1+k+\Delta_2-k=k+\Delta_1+\Delta_2$.
    Without loss of generality, we assume $\Delta_1 \le \Delta_2$ ({otherwise replace $A$ with $A^T$}).
    Hence, $k+2\Delta_1 \le k+\Delta_1+\Delta_2 \le q$.
    
    Now, \cref{lemma: symmetry for robust matrices} applied to $A[I,I]$ (with $r=k$) yields a symmetric matrix $C \in \mb{F}^{I \times I}$ such that 
    \begin{equation}\label{eq:I-I}
        \|A[I,I]-C\|_0 \le (q^2+1)\rho n^2
    \end{equation} and such that the row space of $C$ is the same as that of $A[I,I]$.
    The latter property implies that $A[I,I]=P_1 C$ for some $P_1 \in \mb{F}^{I\times I}$.
    
    We have shown how to approximate $A[I,I]$ with the symmetric matrix $C$; we now need to deal with the other parts of the matrix $A$.
    Recall that $\rank(A[[n],I])=\rank(A[I,I])+\Delta_1$.
    Letting $I^{\mr c}=[n]\setminus I$, this means we can write $A[I^{\mr c},I]=P_2 A[I,I]+Q$, for some $P_2,Q \in \mb{F}^{I^{\mr c} \times I}$ and $\rank(Q)=\Delta_1$.
    Let $P:=P_2P_1 \in \mb{F}^{I^{\mr c} \times I}$, so $A[I^{\mr c},I]=PC+Q$.
   
    Now, let $B\in \mb F^{n\times n}$ be the symmetric matrix defined by \[B[I,I]=C, \quad B[I^{\mr c},I]=B[I,I^{\mr c}]^T=PC+Q,\quad B[I^{\mr c},I^{\mr c}]=PCP^T+PQ^T+QP^T.\]
    By construction, and our assumption on $A$, we have
    \begin{align}
    \|B[I^{\mr c},I]-A[I^{\mr c},I]\|_0&=0,\label{eq:Ic-I}
    \\
    \|B[I,I^{\mr c}]-A[I,I^{\mr c}]\|_0&=\|B[I^{\mr c},I]^T-A[I,I^{\mr c}]\|_0=\|A[I^{\mr c},I]^T-A[I,I^{\mr c}]\|_0\le \rho n^2.\label{eq:I-Ic}
    \end{align}
    Combining \cref{eq:I-Ic,eq:Ic-I} with \cref{eq:I-I} yields
    \begin{align*}\|A-B\|_0&\le \|C-A[I,I]\|_0\;+\;\|B[I^{\mr c},I]-A[I^{\mr c},I]\|_0\;+\;\|B[I,I^{\mr c}]-A[I,I^{\mr c}]\|_0\;+\;|I^{\mr c}|^2\\&\le (q^2+1)\rho n^2 + \rho n^2 + q^4\rho n^2= O(q^4 \rho n^2).\end{align*}
    It remains to show that $\rank(B) \le q$. This follows from a short sequence of elementary row and column operations:
    \begin{align*}
            \rank(B)&= \renewcommand\arraystretch{1.2}
            \rank\left[
                \begin{array}{ c | c }
                    C & CP^T+Q^T \\
                    \hline
                    PC+Q & PCP^T+PQ^T+QP^T
                \end{array}
            \right]
            = \renewcommand\arraystretch{1.2}
            \rank\left[
                \begin{array}{ c | c }
                    C & Q^T \\
                    \hline
                    PC+Q & PQ^T
                \end{array}
            \right] 
            = \renewcommand\arraystretch{1.2}
            \rank\left[
                \begin{array}{ c | c }
                    C & Q^T \\
                    \hline
                    Q & 0
                \end{array}
            \right]\\
            &\le \rank(C) + 2\rank(Q)
            = k+2\Delta_1 \le q\,. \qedhere
            \end{align*}
\end{proof}

\section{Decoupling inequalities}\label{sec:decoupling}
In this section we prove \cref{lem:multiple-decoupling-LO,lem:custom-decoupling-LO}, which are the decoupling inequalities that will be used in the proofs of \cref{thm:rank,thm:stability} respectively.
Both lemmas use the following fact, which is an easy application of Jensen's inequality.
\begin{lemma}[Decoupling with multiple copies]\label{cla:decouple}
    Let $\mathcal{E}(E,F)$ be an event depending on independent random objects  $E,F$. 
    Then, for any $k\in \mathbb{N}$, we have
    \[\Pr[\mathcal{E}(E,F)] \leq \Pr\big[\mathcal{E}(E_0, F)\cap \dots\cap \mathcal{E}(E_k, F)\big]^{1/(k+1)},\]
    where $E_0,\dots, E_k$ are independent copies of $E$.
\end{lemma}
\begin{proof}
    By the law of total probability and the independence of $E_0,\dots,E_k$, we have
    \[
    \Pr\big[\mathcal{E}(E_0, F)\cap \dots\cap \mathcal{E}(E_k, F)\big]=\mb E\Big[\Pr\big[\mathcal{E}(E_0, F)\cap \dots\cap \mathcal{E}(E_k, F)\,\big|F\,\big]\Big]=\mb E\Big[\Pr\big[\mathcal{E}(E, F)\,\big|F\,\big]^{k+1}\Big].
    \]
    By Jensen's inequality, this is at least
    \begin{align*}
    \mb E\Big[\Pr\big[\mathcal{E}(E, F)\,\big|F\,\big]\Big]^{k+1}=\Pr\big[\mathcal{E}(E, F)\big]^{k+1}.
    \end{align*}
    Rearranging yields the desired result.
\end{proof}

We now prove \cref{lem:multiple-decoupling-LO} and \cref{lem:custom-decoupling-LO}. The proof approach is the same for both of these statements, but \cref{lem:custom-decoupling-LO} involves somewhat more complicated manipulations of equations.
\begin{proof}[Proof of \Cref{lem:multiple-decoupling-LO}]
    Fix any $z \in \mb{F}$, and let $\vec \xi^{\,(0)},\dots,\vec \xi^{\,(k)}$ be independent copies of $\vec \xi$.
    Applying \Cref{cla:decouple} with $E = \vec{\xi}[X]$ and $F = \vec{\xi}[Y]$, and with $\mc{E}(E,F)$ being the event that $f(\vec{\xi}[X],\vec{\xi}[Y]) = z$, we see
    \[
        \Pr[f(\vec{\xi}) = z] \leq \Pr[f(\vec{\xi}^{\,(i)}[X], \vec{\xi}[Y]) = z \text{ for all } i \in \{0,\dots, k\}]^{1/(k+1)}\,.
    \]
    Subtracting the equation for $i=0$ from each of the other $k$ equations, and using that $f(\vec{x}) = \vec{x}^T A \vec{x}+\vec b^T \vec x+c$ (for some $\vec b\in \mb F^n$ and $c\in \mb F$), we deduce
    \begin{align*}
        &\sup_{z \in \mb{F}} \Pr[f(\vec{\xi}) = z] \\
        &\qquad\leq \Pr\Bigl[f(\vec{\xi}^{\,(i)}[X],\vec{\xi}[Y])-f(\vec{\xi}^{\,(0)}[X],\vec{\xi}[Y])=0\text{ for all }i\in\{1,\dots,k\}\Bigr]^{1/(k+1)} \\
         &\qquad= \Pr\Bigl[(\vec{\xi}^{\,(i)}[X]-\vec{\xi}^{\,(0)}[X])^T A[X,Y]\,\vec{\xi}[Y]+\psi_0(\vec{\xi}^{\,(0)}[X],\vec{\xi}^{\,(i)}[X])=0\text{ for all }i\in\{1,\dots,k\}\Bigr]^{1/(k+1)},
    \end{align*}
    where
    \[\psi_0(\vec{\xi}^{\,(0)}[X],\vec{\xi}^{\,(i)}[X]) = \frac{1}{2} \Bigl(\vec{\xi}^{\,(i)}[X]^T A[X,X]\vec{\xi}^{\,(i)}[X] -  \vec{\xi}^{\,(0)}[X]^T A[X,X]\vec{\xi}^{\,(0)}[X] + \vec{b}[X]^T(\vec{\xi}^{\,(i)}[X] -\vec{\xi}^{\,(0)}[X])\Bigr).\]
    Let $\vec{x} \in \{-1,1\}^X$ maximise the expression
    \[
        \Pr\Bigl[(\vec{\xi}^{\,(i)}[X]-\vec{x})^T A[X,Y]\,\vec{\xi}[Y]+\psi_0(\vec{x},\vec{\xi}^{\,(i)}[X])=0\text{ for all }i\in\{1,\dots,k\}\Bigr]. 
    \]
    Then, we conclude 
    \[
        \sup_{z \in \mb{F}} \Pr[f(\vec{\xi}) = z] 
        \leq \Pr\Bigl[(\vec{\alpha}^{\,(i)})^T A[X,Y]\,\vec{\xi}[Y]=\psi(\vec{\alpha}^{\,(i)}) \text{ for all }i\in\{1,\dots,k\} \Bigr]^{1/(k+1)},
    \] where $\vec{\alpha}^{\,(i)} = \vec{\xi}^{\,(i)}[X] - \vec{x}$ is a shifted Rademacher random variable and  $\psi(\vec{\alpha}^{\,(i)}) = -\psi_0(\vec{x},\vec{\alpha}^{\,(i)} + \vec{x})$.
\end{proof}

\begin{proof}[Proof of \Cref{lem:custom-decoupling-LO}]
    Fix any $z \in \mb{F}$ and let $\pvec \xi$ be an independent copy of $\vec{\xi}$.
    We apply the $k=1$ case of \Cref{cla:decouple} twice, first with
    \[E = \vec{\xi}[X],\quad F = \vec{\xi}[Y\cup Z],\quad \mc{E}(E,F) = \{f(\vec{\xi}[X],\vec{\xi}[Y], \vec{\xi}[Z]) = z\},\]
    and then with 
    \[E = \vec{\xi}[Y],\quad F = (\vec{\xi}[X],\pvec{\xi}[X],\vec{\xi}[Z]),\quad \mc{E}(E,F) = \left\{\begin{minipage}{100pt}\setlength\abovedisplayskip{0pt}\begin{align*}f(\vec{\xi}[X], \vec{\xi}[Y],\vec{\xi}[Z]) &=z,\\
f(\pvec{\xi}[X], \vec{\xi}[Y], \vec{\xi}[Z]) &= z\end{align*}\vspace{-8pt}\end{minipage}
\right\}.\]
This yields
    \[
        \Pr[f(\vec \xi)=z]\le 
        \Pr\left[\begin{minipage}{140pt}\setlength\abovedisplayskip{0pt}\begin{align}
            f(\vec{\xi}[X],\vec{\xi}[Y],\vec{\xi}[Z])&=z,\label{eq:1}\\
            f(\pvec{\xi}[X],\vec{\xi}[Y],\vec{\xi}[Z])&=z,\label{eq:2}\\
            f(\vec{\xi}[X],\pvec{\xi}[Y],\vec{\xi}[Z])&=z,\label{eq:3}\\
            f(\pvec{\xi}[X],\pvec{\xi}[Y],\vec{\xi}[Z])&=z\label{eq:4}
        \end{align}\vspace{-8pt}\end{minipage}\right]^{1/4}.\notag
    \]
    Now, we write  $f(\vec{x}) = \vec{x}^T A \vec{x}+\vec b^T \vec x+c$ for some $\vec b\in \mb F^n$ and $c\in \mb F$. 
    If we subtract \eqref{eq:2} from \eqref{eq:1}, and \eqref{eq:4} from \eqref{eq:3}, and \eqref{eq:3} from \eqref{eq:1}, we obtain 
    \begin{align}
        &2(\vec{\xi}[X] - \pvec{\xi}[X])^T A[X, Y\cup Z] \vec{\xi}[Y\cup Z]+\varphi_0(\vec\xi[X],\pvec{\xi}[X])=0\label{eq:5}\\
        &2(\vec{\xi}[X] - \pvec{\xi}[X])^T A[X, Y\cup Z] (\pvec{\xi}[Y],\vec{\xi}[Z])+\varphi_0(\vec\xi[X],\pvec{\xi}[X])=0,\label{eq:6}\\
        &2(\vec{\xi}[Y] - \pvec{\xi}[Y])^T A[Y, X\cup Z] \vec{\xi}[X\cup Z]+\psi_0(\vec\xi[Y],\pvec{\xi}[Y]) = 0.\label{eq:7} 
    \end{align}
    where
    \begin{align*}
        \varphi_0(\vec\xi[X],\pvec{\xi}[X])&=\vec{\xi}[X]^T A[X,X]\vec{\xi}[X] - \pvec{\xi}[X] A[X,X] \pvec{\xi}[X] + \vec{b}[X]^T(\vec{\xi}[X] - \pvec{\xi}[X]),\\
        \psi_0(\vec\xi[Y],\pvec{\xi}[Y])&=\vec{\xi}[Y]^T A[Y,Y]\vec{\xi}[Y] - \pvec{\xi}[Y] A[Y,Y] \pvec{\xi}[Y] + \vec{b}[Y]^T(\vec{\xi}[Y] - \pvec{\xi}[Y]).
    \end{align*}
    Now, let
    \[\vec\alpha=\vec \xi[X]-\pvec \xi[X],\qquad \vec\beta=\vec \xi[Y]-\pvec \xi[Y],\qquad \vec \gamma=\vec \xi[Z].\]
    Subtracting \eqref{eq:6} from \eqref{eq:5} yields
    \[\vec{\alpha}^{\,T} A[X,Y] \vec{\beta} = 0.\]
    \eqref{eq:6} and \eqref{eq:7} can be written as 
    \[
        \vec{\alpha}^{\,T} A[X,Z] \vec{\gamma} = \varphi_1(\vec{\xi}[X],\pvec{\xi}[X], \vec{\xi}[Y],\pvec{\xi}[Y])\quad \text{and}\quad \vec{\beta}^{\,T} A[Y,Z] \vec{\gamma} = \psi_1(\vec{\xi}[X], \pvec{\xi}[X], \vec{\xi}[Y],\pvec{\xi}[Y])
    \] respectively, where
    \begin{align*}
        \varphi_1(\vec{\xi}[X],\pvec{\xi}[X], \vec{\xi}[Y],\pvec{\xi}[Y])
            &=-(\vec{\xi}[X] - \pvec{\xi}[X])^T A[X, Y] \pvec{\xi}[Y]-\varphi_0(\vec\xi[X],\pvec{\xi}[X])/2 \\
        \psi_1(\vec{\xi}[X], \pvec{\xi}[X], \vec{\xi}[Y],\pvec{\xi}[Y])
            &=-(\vec{\xi}[Y] - \pvec{\xi}[Y])^T A[Y, X] \vec{\xi}[X]-\psi_0(\vec\xi[Y],\pvec{\xi}[Y])/2.
    \end{align*}
    Hence,
    \[
        \Pr[f(\vec \xi)=z]\le 
        \Pr\left[\begin{minipage}{230pt}\setlength\abovedisplayskip{0pt}\begin{align}
        &\vec{\alpha}^{\,T} A[X,Y] \vec{\beta} = 0\notag ,\\
        &\vec{\alpha}^{\,T} A[X,Z] \vec{\gamma} = \varphi_1(\vec{\xi}[X],\pvec{\xi}[X], \vec{\xi}[Y],\pvec{\xi}[Y]),\label{eq:9}\\
        &\vec{\beta}^{\,T} A[Y,Z] \vec{\gamma} = \psi_1(\vec{\xi}[X], \pvec{\xi}[X], \vec{\xi}[Y],\pvec{\xi}[Y])\label{eq:10}
    \end{align}\vspace{-8pt}\end{minipage}\right]^{1/4}.\notag
    \]
    This is almost the desired conclusion, except that $\varphi_1$ and $\psi_1$ do not only depend on $\vec{\alpha},\vec{\beta}$. To address this issue, for all possible outcomes of $\vec \alpha\in \{-1,0,1\}^X$ and $\vec \beta\in \{-1,0,1\}^Y$, choose $\vec x_{\vec \alpha,\vec \beta}\in \{-1,1\}^X$ and $\vec y_{\vec \alpha,\vec \beta}\in \{-1,1\}^Y$ to maximise the conditional probability 
     \[
        \Pr\Big[\eqref{eq:9}\text{ and }\eqref{eq:10}\text{ hold}\,\Big|\,\pvec{\xi}[X]=\vec x_{\vec \alpha,\vec \beta},\;\pvec{\xi}[Y]=\vec y_{\vec \alpha,\vec \beta},\;\vec \alpha,\vec \beta\Big]
    \]
    (among all choices of $\vec x_{\vec \alpha,\vec \beta}$ and $\vec y_{\vec \alpha,\vec \beta}$ such that this conditional probability is well-defined).
    In addition, let $\varphi(\vec\alpha,\vec \beta)=\varphi_1(\vec x_{\vec \alpha,\vec \beta}+\vec{\alpha}, \vec x_{\vec \alpha,\vec \beta}, \vec y_{\vec \alpha,\vec \beta}+\vec{\beta}, \vec y_{\vec \alpha,\vec \beta})$ 
    and $\psi(\vec\alpha,\vec \beta)=\psi_1(\vec x_{\vec \alpha,\vec \beta}+\vec{\alpha}, \vec x_{\vec \alpha,\vec \beta}, \vec y_{\vec \alpha,\vec \beta}+\vec{\beta}, \vec y_{\vec \alpha,\vec \beta})$. 
    Then we have \[
        \sup_{z \in \mb{F}} \Pr[f(\vec \xi)=z] 
        \le \Pr\Bigl[\vx^{\,T} A[X,Y] \vy=0, \quad\vx^{\,T} A[X,Z] \vz=\varphi(\vec{\alpha}, \vec{\beta}),\quad\vy^{\,T} A[Y,Z] \vz=\psi(\vec{\alpha},\vec{\beta})\Bigr]^{1/4},
    \]
    as desired.
\end{proof}

\section{Bounds in terms of rank, for quadratic polynomials}\label{sec:rank}
In this section we prove \cref{thm: multi-x-single-y costello}, and use it (together with \cref{lem:multiple-decoupling-LO,lemma: close to symmetric low rank}) to deduce \cref{thm:rank}. 
First, we show how to deduce \cref{thm:rank} from \cref{thm: multi-x-single-y costello}.
\begin{proof}[Proof of \cref{thm:rank} assuming \Cref{thm: multi-x-single-y costello}]
    In what follows, we assume $n$ is even\footnote{For the odd-$n$ case, we can add a single dummy variable to reduce to the even-$n$ case.} and sufficiently large. Write the quadratic part of $f(\vec x)$ as $\vec{x}^T A \vec{x}$.

    Suppose \cref{C1} does not hold, i.e., for any symmetric matrix $B \in \mb{F}^{n\times n}$ with rank less than $2k^2$, $f(\vec{x})-\vec{x}^T B\vec{x}$ must have more than $\varepsilon n^2$ nonzero coefficients.
    Our goal is to show that \cref{C2} holds, i.e., $\sup_{z\in \mb F}\Pr[f(\xi_1,\dots,\xi_n)=z] \le n^{-1+2/k}$.

    Since $f$ has at most $n+1$ non-quadratic terms, our assumption that \cref{C1} does not hold implies that for any symmetric matrix $B\in \mb F^{n\times n}$ with rank less than $2k^2$, it must be that $\|A-B\|_0 >  \varepsilon n^2-n-1$. 
    By \cref{lemma: close to symmetric low rank}, this means that, for some $\delta\in (0,1)$ depending on $\varepsilon$ and $k$, at least a $\delta$-fraction of $2k^2\times 2k^2$ submatrices of $A$ are nonsingular.

    By \cref{lemma: tri-partition}, there is a partition of $\{1,\dots,n\}=X\cup Y$ into parts of size $|X|=|Y|=n/2$, such that more than a $(\delta/2)$-fraction of  $2k^2\times 2k^2$ submatrices of $A[X,Y]$ are nonsingular. By \cref{lem:multiple-decoupling-LO} we then have
    \[
        \sup_{z\in \mb F}\Pr[f(\vec \xi)=z] 
        \leq\Pr\Bigl[(\vec{\alpha}^{\,(i)})^T A[X,Y]\,\vec{\xi}[Y]=\psi(\vec{\alpha}^{\,(i)})\text{ for all }i\in\{1,\dots,k\}\Bigr]^{1/(k+1)},
    \]
    for some function $\psi:\mb F^X\to \mb F$, and some i.i.d.\ random vectors $\vec{\alpha}^{(1)},\dots,\vec{\alpha}^{(k)}\in \mb F^X$ with independent shifted Rademacher entries (all independent of $\vec \xi[Y]$). By \cref{thm: multi-x-single-y costello} (with $\Xi$ as the matrix with rows $\vec \alpha^{(1)},\dots,\vec \alpha^{(k)}$, and $\vec \eta=\vec \xi[Y]$, and $\vec\varphi(\Xi)=(\psi(\vec{\alpha}^{\,(i)}))_{i=1}^k$, and $\varepsilon=1$), we have
    \[
        \Pr\Bigl[(\vec{\alpha}^{\,(i)})^T A[X,Y]\,\vec{\xi}[Y]=\psi(\vec{\alpha}^{\,(i)})\text{ for all }i\in\{1,\dots,k\}\Bigr]
        \le (n/2)^{-k+1}\le n^{-(k+1)(1-2/k)}.
    \]
    The desired bound follows.
\end{proof}

In the rest of this section, we prove \cref{thm: multi-x-single-y costello}. Recalling the strategy outlined in \cref{subsubsec:rank}, via the Nguyen--Vu inverse theorem (\cref{theorem: optimal inverse theorem}), the proof boils down to bounding the probability that a given tuple of columns of $\Xi A$ lies in a GAP of ``small'' rank and of ``small'' volume (recall \cref{def:GAP}).
To be precise, we need the following two lemmas.

\begin{lemma}\label{lem:count-1}
    Fix $k \le m$ and $\mb{F} \in \{\mb{R},\mb{C}\}$. For some $\ell\le n$, consider a matrix $A \in \mb{F}^{n \times m}$ which has at least $\ell$ disjoint nonsingular $m\times m$ submatrices, and let $\Xi\in \mb{F}^{k \times n}$ be a matrix with independent shifted Rademacher entries.
    Then, $$\Pr[\rank(\Xi A) < k] = O_{m}(\ell^{-(m-k+1)/2}).$$  
\end{lemma}
\begin{proof}
    This follows from \Cref{cor:Halasz}. Indeed, write $\vec\xi^{\,(i)}\in \mb{F}^{n}$ for the $i$-th row of $\Xi$. 
    If $\rank(\Xi A) < k$, then there is an index $i\in \{1,\dots,k\}$ such that $\vec \xi^{\,(i)} A$ lies in $\on{span}(\vec \xi^{\,(j)} A: j < i)$ (which is a linear subspace with dimension at most $k-1$). Thus, applying \Cref{cor:Halasz} (after transposing, with the matrix $A^T$), we conclude
    \[
        \Pr[\rank(\Xi A) < k] 
        \leq \sum_{i=1}^k \Pr\big[\vec \xi^{\,(i)} A \in \on{span}(\vec \xi^{\,(j)} A: j < i)\big] 
        = k\cdot O_m(\ell^{-(m-k+1)/2})
        = O_m(\ell^{-(m-k+1)/2}).\qedhere
    \]
\end{proof}

\begin{lemma}\label{lem:count-2}
    Fix $k \le r\le m$ and $\mb{F} \in \{\mb{R},\mb{C}\}$. For some $\ell\le n$, consider a matrix $A \in \mb{F}^{n \times m}$ which has at least $\ell$ disjoint nonsingular $m\times m$  submatrices, and let $\Xi\in \mb{F}^{k \times n}$ be a matrix with independent shifted Rademacher (or lazy Rademacher) entries. Then for any $V\ge 1$,
    \begin{align*}
        &\Pr[\text{the columns of }\Xi A\text{ lie in a symmetric GAP of rank at most }r\text{ and volume at most }V]\\
        &\qquad\qquad \le O_{m} ( V^m(1+\log  V)^{r-1} \cdot \ell^{-k(m-r)/2}).
    \end{align*}
\end{lemma}
\begin{proof}
Let $\mathcal Z(N_1,\dots,N_r)$ be the set of integer matrices $Z\in \mb Z^{r\times m}$ such that, for each $i\in \{1,\dots,r\}$, the absolute value of the entries in the $i$th row is at most $N_i$. Then, let $\mc Z(V)$ be the union of all the sets $\mathcal Z(N_1,\dots,N_r)$, among $N_1,\dots,N_r$ satisfying $(2N_i+1)(2N_2+1)\dots(2N_r+1) \le V$.

By definition, if the column vectors of $\Xi A$ lie in a single 
symmetric GAP of rank at most $r$ and volume at most $V$, then $\Xi A = U Z$ for some $U \in \mb{F}^{k \times r}$ and $Z \in \mc{Z}(V)$. We will proceed by studying the probability that $\Xi A = U Z$ for some particular $Z$, and summing over all $Z\in \mc Z(V)$.
    \begin{claim}
        For every fixed $Z\in \mc Z(V)$, we have $\Pr[\Xi A=UZ\text{ for some }U\in \mb F^{k\times r}]\le \ell^{-k(m-r)/2}$.
    \end{claim}
    \begin{claimproof}
        This is a consequence of \cref{cor:Halasz}, noting that we can interpret $\Xi A$ as a (matrix-valued) linear function of the $kn$ entries of $\Xi$, and we can interpret $\{UZ:U\in \mb F^{k\times r}\}$ as a $kr$-dimensional linear subspace of the vector space of $k\times m$ matrices.
        
        In more detail: let $\vec \xi^{\,*}\in \mb F^{kn}$ be the vector obtained by concatenating the $k$ rows of $\Xi$ (so, this is a row vector of $kn$ independent shifted Rademacher random variables or $kn$ independent lazy Rademacher random variables), and let $A^*\in \mb F^{nk\times mk}$ be the block-diagonal matrix obtained by concatenating $k$ copies of $A$ along the diagonal. 
        Then, $\vec \xi^{\,*} A^*$ is the concatenation of the $k$ rows of $\Xi A$. 
        If we let $\mc W\in \mb F^{km}$ be the set of all vectors obtained by concatenating the $k$ rows of a matrix $UZ$, for some $U\in \mb F^{k\times r}$, then the event that $\Xi A=UZ$ for some $U\in \mb F^{k\times r}$ precisely corresponds to the event that $\vec \xi^{\,*} A^*\in \mc W$. 
        Note that $A^*$ has at least $\ell$ disjoint nonsingular $km\times km$ submatrices (each obtained by concatenating the $k$ copies of one of the $\ell$ disjoint nonsingular $m\times m$ matrices in $A$),
        and $\mc W$ is a $kr$-dimensional linear subspace of $\mb F^{km}$, so the desired bound follows from \cref{cor:Halasz} (after transposing, with the matrix $(A^*)^T$).
    \end{claimproof}
    \begin{claim}
        $|\mc{Z}(V)| = O_{m} (V^m (1+\log V))^{r-1})$.
    \end{claim}
    \begin{claimproof}
        Let $d$ be the minimal positive integer such that $V \le 2^d$, so $d=\log_2 V+O(1)$.
        
        Consider $N_1,\dots,N_r\in \mb N$ such that $(2N_i+1)(2N_2+1)\dots(2N_r+1) \le V$. For each $i\in \{1,\dots,r\}$, let $d_i\in \mb N$ be the minimal integer such that $N_i\le 2^{d_i}$, so $2N_i+1\ge 2^{d_i}$. 
        We have $2^{d_1+\dots+d_r} \le V \le 2^d$, i.e., $d_1+\dots+d_r \le d$. 
        It follows that
        \[
            \mc Z(V) \subseteq \bigcup_{0\le d_1+\dots+d_r \le d} \mc Z(2^{d_1},\dots,2^{d_r}),
        \] and hence 
        \[
            |\mc Z(V)|\le \sum_{0\le d_1+\dots+d_r \le d} \!\!\!\! |\mc Z(2^{d_1},\dots,2^{d_r})|.
        \]
        Then, note that 
        $|\mc Z(2^{d_1},\dots,2^{d_r})|=\prod_{i=1}^r(2^{d_i+1}+1)^m=O_m(2^{m(d_1+\dots+d_r)})$.
        We deduce that $|\mc Z(V)|$ is at most (recall $r\le m$)
        \[
            \sum_{\substack{d_1,\dots,d_r\in \mb N\\ d_1+\dots+d_r \le d}} \!\!\!\!\!\! O_m(2^{m(d_1+\dots+d_r)})
            \le \sum_{D=0}^{d} \binom{D+r-1}{r-1}O_m(2^{mD})
            = O_m(d^{r-1} 2^{md})
            = O_{m} (V^m (1+\log V)^{r-1}). \qedhere
        \]
    \end{claimproof}
    The desired bound immediately follows from the above two claims.
\end{proof}

\begin{remark}
    Both \cref{lem:count-1} and \cref{lem:count-2} seem far from the truth. In the setting of \cref{lem:count-1} we anticipate a bound of $\ell^{-m/2+o(1)}$ (see \cite{Kat93,Kat94} for related work in the analytic number theory literature, which can be viewed as special cases of this anticipated bound). If one could prove this, as well as an improved bound of the form $V^{m-r+o(1)}\cdot \ell^{-k(m-r)/2+o(1)}$ in \cref{lem:count-2}, then it would be possible to deduce an optimal version of \cref{thm: multi-x-single-y costello} (see \cref{rem:optimal?}), using some of the ideas we will see shortly in the proof of \cref{thm:stability} (namely, consideration of ``expected inverse volume'' quantities as in \cref{lem:technical-estimates}).
\end{remark}

Equipped with \cref{lem:count-1,lem:count-2}, we now present the proof of \cref{thm: multi-x-single-y costello}.
\begin{proof}[Proof of \cref{thm: multi-x-single-y costello}]
    Suppose \cref{G2} does not hold, i.e. $\Pr\big[ \Xi A \vec{\eta}=\vec \varphi(\Xi) \big] > n^{-k+\varepsilon}$ for some function $\varphi: \mb{F}^{k\times n}\to \mb{F}^k$ and some $\varepsilon \in (0,1]$.
    Taking $m:= \lceil2k^2/\varepsilon \rceil$, the goal is to prove that \cref{G1} holds, namely that at most a $\delta$-fraction of the $m\times m$ submatrices of $A$ are nonsingular (for any $\delta>0$, assuming that $n$ is sufficiently large in terms of $m,\delta$).
       
    Let $\mc{E}$ be the event that $\Xi$ satisfies $\Pr\big[\Xi A \vec{\eta}=\varphi(\Xi)\,\big|\,\Xi\big]>n^{-k+\varepsilon}/2$. 
    We have \[
        n^{-k+\varepsilon}
        <\Pr\big[\Xi A \vec{\eta}=\varphi(\Xi)\big] 
        \le \Pr[\mc{E}] + \Pr\big[ \Xi A \vec{\eta}=\varphi(\Xi)\big | \mc{E}^{\mr c}] 
        \le \Pr[\mc{E}] + n^{-k+\varepsilon}/2,
    \] so we must have $\Pr[\mc{E}] > n^{-k+\varepsilon}/2$ (and so in particular $\Pr[\mc{E}]> n^{-k}$).

    Now, note that $\Xi A\vec \eta$ can be viewed as a linear function of $\vec \eta$ (taking values in $\mb F^k$), where the coefficients of this linear function are the columns of $\Xi A$. So, if $\Xi$ satisfies $\mc{E}$, then \cref{theorem: optimal inverse theorem} tells us that most of the columns of $\Xi A$ are contained in a symmetric GAP of small rank and volume. 
    More precisely, for $\Xi$ satisfying $\mc{E}$, applying \Cref{theorem: optimal inverse theorem} with $n' = (\delta/(4m))n$, we conclude there exists a symmetric GAP $\varphi$ in $\mb F^k$ with rank $r = O_k(1)$ and volume $V= O_{k,\delta,\varepsilon}(n^{-r/2 + k-\varepsilon}) \leq n^{-r/2 + k-7\varepsilon/8}$ such that all but at most $n'$ of the columns of $\Xi A$ lie in $\varphi$. Observe that $r< 2k$, since otherwise $V = o(1)$, which is impossible according to the definition of $V$. 
    
    We define a \emph{suitable GAP} to be a symmetric GAP in $\mathbb F^k$ of rank $r< 2k$ and volume $V\leq n^{-r/2 + k-7\varepsilon/8}$.
    For an arbitrary outcome of $\Xi$, define $I = I(\Xi)\subseteq \{1,\dots, n\}$ to be a maximum-size set of (indices of) columns of $\Xi A$ that are contained in a suitable GAP. 
    Whenever $\Xi$ satisfies $\mc{E}$, we have $|I|\geq n - n' \geq (1-\delta/(4m))n$.
    By \cref{lemma: tuple-counting} (with $m$ as the ``$r$'' in \cref{lemma: tuple-counting}), it follows that for all but at most a $\delta/2$-fraction of $m$-sets $J\subseteq \{1,\dots,n\}$, we have $\Pr[J\subseteq I]\ge \Pr[\mc{E}]/2\ge n^{-k+\varepsilon}/4$. Say such subsets $J$ are \emph{good}.

    If $J$ is good, then the corresponding columns of $\Xi A$ are quite likely to lie in a suitable GAP. We will use \cref{lem:count-1,lem:count-2} to show that this implies that $A$ has many disjoint nonsingular $m\times m$ submatrices in the columns indexed by $J$; summing over good $J$ will then yield the desired result.

    \begin{claim}
        If $J$ is good, then at most a $\delta/2$-fraction of the $m\times m$ submatrices of $A$ in the columns indexed by $J$ are nonsingular.
    \end{claim}
    \begin{claimproof}
        Let $\mc F_r$ be the event that the columns of $\Xi A$ indexed by $J$ lie in a suitable GAP with rank $r$, so the assumption that $J$ is good implies that $\Pr[\mc F_0\cup \dots\cup \mc F_{2k-1}]\ge n^{-k+\varepsilon}/4$.
        
        Suppose for the purpose of contradiction that the columns indexed by $J$ contain at least $(\delta/2)\binom n m$ nonsingular $m\times m$ submatrices. By \cref{fact: disjoint and non-disjoint nonsingular submatrices}, then these columns in fact contain $\ell=\Omega_{m,\delta}(n)$ \emph{disjoint} $m\times m$ submatrices.
        
        Recall that $m=\lceil 2k^2/\varepsilon\rceil$. First, \cref{lem:count-1} implies that
        \[\Pr[\mc F_0\cup \dots\cup \mc F_{k-1}]\le O_m(\ell^{-(m-k+1)/2})=o_{m,\delta}(n^{-k+\varepsilon}),\]
        with plenty of room to spare. Second, if $r\ge k$, \cref{lem:count-2} implies that, for $V= n^{k-7\varepsilon/8 - r/2}$,
        \begin{align*}
            \Pr[\mc F_r]\le O_{m}(V^m(\log  V)^{r-1} \cdot \ell^{-k(m-r)/2})
            &=O_{m,\delta}(n^{m(k-7\varepsilon/8-r/2)+o(1)}\cdot n^{-k(m-r)/2})\\
            &=O_{m,\delta}(n^{km/2-7\varepsilon m/8-r(m/2-k/2)+o(1)})\\
            &\le O_{m,\delta}(n^{km/2-7\varepsilon m/8-k(m/2-k/2)+o(1)})\\
            &\le O_{m,\delta}(n^{-7k^2/4+k^2/2+o(1)}) 
            =o_{m,\delta}(n^{-k+\varepsilon}).
        \end{align*}
        Summing the above bounds yields $\Pr[\mc F_0\cup \dots\cup \mc F_{{2k-1}}]= o_{m,\delta}(n^{-k+\varepsilon})$, which is the desired contradiction.
    \end{claimproof}
Now, using the above claim, the number of nonsingular $m\times m$ submatrices in good sets of columns is at most $(\delta/2)\binom n m^2$. Recalling that all but a $(\delta/2)$-fraction of $m$-sets of columns are good, the number of nonsingular $m\times m$ submatrices which do \emph{not} lie in a good set of columns is also at most $(\delta/2)\binom n m^2$. That is to say, at most a $\delta$-fraction of $m\times m$ submatrices of $A$ are nonsingular, which proves \cref{G1}.
\end{proof}

\section{A power-saving improvement for robustly irreducible quadratic polynomials}\label{sec:power-saving}

\renewcommand{\vx}{\vec{\alpha}}
\renewcommand{\vy}{\vec{\beta}}
\renewcommand{\vz}{\vec{\gamma}}
\newcommand{\vxi}{\vec{\xi}}

In this section we prove \cref{theorem: three bilinears}, and use it (together with \cref{lem:custom-decoupling-LO,lemma: close to symmetric low rank}) to deduce \Cref{thm:stability}. First, we show how to deduce \cref{thm:stability} from \cref{theorem: three bilinears}.

\begin{proof}[Proof of \Cref{thm:stability} assuming \Cref{lem:custom-decoupling-LO} and \Cref{theorem: three bilinears}]
    In what follows, we assume $n$ is divisible\footnote{in general, we can add one or two dummy variables to reduce to the case where $n$ is divisible by 3.} by 3, and sufficiently large. Write the quadratic part of $f(\vec x)$ as $\vec x^T A\vec x$ for some symmetric matrix $A \in \mb{C}^{n\times n}$.

    Suppose \cref{B1} does not hold, i.e., for every reducible polynomial $g \in \mb{C}[x_1,\dots,x_n]$ with degree at most 2, the polynomial $f-g$ must have more than $\varepsilon n^2$ nonzero coefficients.
    Our objective is to prove that \cref{B2'} holds, i.e., $\sup_{z\in \mb C}\Pr[f(\vec \xi)=z] \le n^{-1/2-1/24+\varepsilon}$.

    Note that a quadratic form is reducible over $\mb C$ if and only if its rank is less than 3. Thus, for every symmetric matrix $B \in \mb{C}^{n\times n}$ of rank less than 3, the polynomial $g(\vec x)=\vec x^T B\vec x$ is reducible, and so $f-g$ must have more than $\varepsilon n^2$ nonzero coefficients. Since $f-g$ has at most $n+1$ non-quadratic terms, this implies that we must have $\|A-B\|_0 > \varepsilon n^2-n-1$ for every symmetric matrix $B \in \mb{C}^{n\times n}$ with rank less than 3.
    By \cref{lemma: close to symmetric low rank}, this means that, for some $\delta>0$ depending on $\varepsilon$, at least a $\delta$-fraction of the $3\times 3$ submatrices of $A$ are nonsingular.

    By \cref{lemma: tri-partition}, there is a partition of $\{1,\dots,n\}=X\cup Y\cup Z$ with $|X|=|Y|=|Z|=n/3$ such that for each of the matrices $A[X,Y],A[X,Z],A[Y,Z]$, more than a $(\delta/2)$-fraction of their $3\times 3$ submatrices are nonsingular. By \cref{lem:custom-decoupling-LO} we then have
    \[\sup_{z\in \mb C}\Pr[f(\vec \xi)=z] \le \Pr\Bigl[\vx^{\,T} A[X,Y] \vy=0, \quad\vx^{\,T} A[X,Z] \vz=\varphi(\vec{\alpha}, \vec{\beta}),\quad\vy^{\,T} A[Y,Z] \vz=\psi(\vec{\alpha},\vec{\beta})\Bigr]^{1/4},\]
    for some functions $\varphi,\psi:\mb C^X\times \mb C^Y \to \mb C$, and independent random vectors $\vx,\vy,\vz$, where the entries of $\vx,\vy$ are i.i.d.\ lazy Rademacher, and the entries of $\vz$ are i.i.d.\ Rademacher. By \cref{theorem: three bilinears} (taking the ``$\varepsilon$'' in \cref{theorem: three bilinears} to be $\min(\delta/2,\varepsilon)$, we have
    \begin{align*}
        \Pr\Bigl[\vx^{\,T} A[X,Y] \vy=0, \quad\vx^{\,T} A[X,Z] \vz=\varphi(\vec{\alpha}, \vec{\beta}),\quad\vy^{\,T} A[Y,Z] \vz=\psi(\vec{\alpha},\vec{\beta})\Bigr]
        &\le (n/3)^{-2-1/6+\varepsilon}\\
        &\le n^{-4(1/2+1/24-\varepsilon)},
    \end{align*}
    and the desired bound follows.
\end{proof}

In the rest of this section, we prove \Cref{theorem: three bilinears}. As in the proof of \cref{thm:rank}, we will use the Nguyen--Vu inverse theorem (\cref{theorem: optimal inverse theorem}), but the estimates are much more delicate. It will be convenient to use the following reformulation of \cref{theorem: optimal inverse theorem} (which we will apply with $q=3$).

\begin{definition}\label{def:volume-GAP}
    Let $G$ be a torsion-free additive group. 
    For $r\in \mb N$ and $\delta>0$, and a sequence $\vec v=(v_1,\dots,v_n)\in G^n$ of elements of $G$, let $h_r^\delta(\vec{v})$ be the minimum volume of a symmetric GAP of rank at most $r$ that contains all but at most $n^{1-\delta}$ entries of $\vec{v}$ (let $h_r^\delta(\vec v)=\infty$ if no such GAP exists).
\end{definition}
We remark that the only possible rank-0 symmetric GAPs is the singleton $\{0\}$. So, $h_0^\delta(\vec{v})=1$ if at most $n^{1-\delta}$ elements of $\vec{v}$ are nonzero; otherwise $h_0^\delta(\vec{v})=\infty$. 
\begin{theorem}\label{lemma: prob via volume}
    Fix $q\in \mb N$ and $0<\delta<1/2$ , and let $G$ be a torsion-free additive group. Let $n$ be sufficiently large (in terms of $\delta,q$) and consider some $\vec v=(v_1,\dots,v_n) \in G^n$.
    Let $\xi_1,\dots,\xi_n$ be i.i.d.\ Rademacher (or lazy Rademacher) random variables. We have
    \[
        \sup_{z \in G} \Pr\left[\xi_1v_1+\dots+\xi_nv_n = z \right] \le n^{-q/2+q\delta}+\sum_{r=0}^{q-1}n^{-r/2+q\delta}\cdot \frac1{h_r^\delta(\vec{v})}.
    \]
\end{theorem}

\begin{proof}
    Let $\rho=\sup_{z \in G} \Pr[\xi_1v_1+\dots+\xi_nv_n = z]$ and suppose for the purpose of contradiction that the desired bound does not hold. Then, we have $\rho\ge n^{-q/2+q\delta}$ and by\footnote{ \Cref{theorem: optimal inverse theorem} is only stated for Rademacher random variables, but it is easy to deduce the same statement in the lazy Rademacher case. 
    Indeed, a sequence of independent lazy Rademacher random variables can be expressed as a sum $\vec\xi^{(1)}+\vec\xi^{(2)}$ of two independent sequences of independent Rademacher random variables. 
    We can simply apply \Cref{theorem: optimal inverse theorem} after conditioning on an outcome of $\vec\xi^{(1)}$.} \Cref{theorem: optimal inverse theorem} there is a symmetric GAP that contains all but $n^{1-\delta}$ of $v_1,\dots,v_n$, such that the rank $r$ and the volume $V$ of this GAP satisfy
    \begin{equation}
        V=O_q(\rho^{-1}n^{-(1-\delta)(r/2)}).\label{eq:V-inverse-application}
    \end{equation}
    We are assuming that $\rho> n^{-r/2+q\delta}/h_r^\delta(\vec{v})$ for all $r<q$, 
    so \cref{eq:V-inverse-application} implies that the volume of our GAP satisfies $V=O_q(n^{-q\delta+\delta r/2}h_r^\delta(\vec{v}))<h_r^\delta(\vec v)$ for $r<q$. 
    By the definition of $h_r^\delta(\vec v)$, it follows that we cannot have $r<q$. 
    But if $r\ge q$, then \cref{eq:V-inverse-application} also implies that the volume of our GAP is $V=O_q(n^{-q\delta+q/2-(1-\delta)r/2})=O_q(n^{-q\delta+q/2-(1-\delta)q/2})=O_q(n^{-q\delta/2})<1$ (recalling that $\rho\ge n^{-q/2+q\delta}$), which is impossible.
\end{proof}

\subsection{Reduction to technical estimates}

Recall that we are working towards a proof of \cref{theorem: three bilinears}; namely, we are attempting to upper bound the probability of the event that $\vx^{\,T} A_1\vy=0$, and $\vx^{\,T} A_2 \vz=\varphi(\vec{\alpha}, \vec{\beta})$, and $\vy^{\,T} A_3 \vz=\psi(\vec{\alpha},\vec{\beta})$. We will want to break this event into sub-events, depending on the structure of the vectors $\vx^{\,T} A_1$, $\vx^{\,T} A_2$ and $\vy^{\,T} A_3$; in each case we will reveal our random vectors in slightly different orders, and estimate probabilities in slightly different ways.

For example, it is quite unlikely that almost all the entries of $\vx^{\,T} A_1$ are zero. But if $\vx^{\,T} A_1$ has plenty of nonzero entries, then we can get quite a strong bound on the probability that $\vx^{\,T} A_1\vy=0$ using the randomness of $\vy$. Similarly, it is quite unlikely that  $\vx^{\,T} A_2$ and $\vy^{\,T} A_3$ are nearly collinear. If they are not nearly collinear, then we can get quite a strong bound on the joint probability that $\vx^{\,T} A_2 \vz=\varphi(\vec{\alpha}, \vec{\beta})$, and $\vy^{\,T} A_3 \vz=\psi(\vec{\alpha},\vec{\beta})$, using the randomness of $\vz$.

In this subsection, we collect various technical probabilistic estimates of this kind (some of whose proofs will be deferred until the next subsection), and show how to piece together these estimates to prove \cref{theorem: three bilinears}. To state all these estimates, it will also be convenient to introduce some notation for the property that a matrix robustly has rank at least 3.

\begin{definition}\label{definition: robust rank 3}
We say a matrix $A\in \mb{C}^{n\times n}$ is said to be {\em $\varepsilon$-robust} if more than an $\varepsilon$-fraction of
    its $3\times 3$ submatrices are nonsingular.
\end{definition}

First, the following lemma, which is a simple corollary of Hal\'asz' inequality (\cref{thm:Halasz}), will be used to bound the probability that almost all of the entries of $\vx^{\,T} A_1$ or $\vx^{\,T} A_2$ are zero.

\begin{lemma}\label{lemma: many nonzero entries}
    Fix $\delta>0$ and let $n\in \mb N$ be sufficiently large (in terms of $\delta$). Consider a $\delta$-robust matrix $A\in \mb{C}^{n \times n}$.
    Let $\vx \in \{-1,0,1\}^n$ be a vector of i.i.d.\ lazy Rademacher random variables.
    Then \[\Pr\big[\Vert \vx^{\,T} A\Vert_0 \le n^{1-\delta}\big] < n^{-3/2+\delta}.\]
\end{lemma}
In the proof of \cref{lemma: many nonzero entries}, and some other lemmas later in this section, it is convenient to interpret a $m\times n$ matrix $A\in \mb F^{m\times n}$ as a sequence of $n$ column vectors in $\mb F^m$, so in particular, for a subset $J\subseteq \{1,\dots,n\}$ we write $A[J]$ to denote the $m\times |J|$ submatrix of $A$ consisting of just the columns indexed by $J$.
\begin{proof}[Proof of \cref{lemma: many nonzero entries}]
    Suppose for the purpose of contradiction that $\Pr[\Vert \vx^{\,T} A\Vert_0 \le n^{1-\delta}] \ge n^{-3/2+\delta}$.
    By \cref{lemma: tuple-counting}, for all but a $6n^{-\delta}$-fraction of the 3-element subsets $J$ of $\{1,\dots,n\}$, we have $\Pr[\vx^{\,T} A[J] = \vec 0\,] \ge \frac{1}{2}n^{-3/2+\delta}$. Say such a $J$ is \emph{good}.
    
    By \cref{thm:Halasz} (after transposing), if $J$ is good, then $A[J]$ contains at most $O(n^{1-2\delta/3})$ disjoint nonsingular $3\times 3$ submatrices. By \cref{fact: disjoint and non-disjoint nonsingular submatrices}, it follows that at most a $O(n^{-2\delta/3})$-fraction of the $3\times 3$ submatrices of $A[J]$ are nonsingular.

    Summing over both good and non-good triples of columns, the total number of nonsingular $3\times 3$ submatrices in $A$ is at most $6n^{-\delta} \binom n3^2+O(n^{-2\delta/3})\cdot \binom n3^2<\delta\binom n 3^2$, contradicting $\delta$-robustness of $A$.
\end{proof}
Next, given an outcome of $\vx^{\,T} A_2$ with many nonzero entries, the following lemma will be used to bound the probability that $\vx^{\,T} A_2$ and $\vy^{\,T} A_3$ are nearly collinear, while $\vx^{\,T} A_1\vy=0$ (when applying the lemma, we will take $\vec{u}$ and $\vec{v}$ to be $\vx^{\,T} A_1$ and $\vx^{\,T} A_2$, respectively). 
It is proved in a similar way to \cref{lemma: many nonzero entries}, using \cref{lemma: prob via volume} instead of Hal\'asz' inequality. (Recall the notation $h_r^\delta$ from \cref{def:volume-GAP}).
\begin{lemma}\label{lemma: dependent-event}
    Fix $\varepsilon,\delta>0$ and let $n\in \mb N$ be sufficiently large (in terms of $\varepsilon,\delta$). 
    Consider (row) vectors $\vec u,\vec v$ with $\|\vec v\|_0> n^{1-\delta}$, and consider an $\varepsilon$-robust matrix $A\in \mb{C}^{n \times n}$.
    Let $\vy\in \{-1,0,1\}^n$ be a (column) vector of i.i.d.\ lazy Rademacher random variables.
    Then 
    \[\Pr\Bigl[\vec u\cdot \vy=0\text{, and }\|a\vec v+b\vy^{\,T} A\|_0\le n^{1-\delta}\text{ for some }a,b\in \mb C\text{, not both zero}\Bigr] \le 2n^{-1+3\delta}\cdot \frac{1}{h_2^\delta(\vec{u})}+2n^{-3/2+3\delta}.\]
\end{lemma}
\begin{proof}
    First note that if $b=0$ and $a\ne 0$ then $\|a\vec v+b\vy^{\,T} A\|_0=\|\vec v\|_0> n^{1-\delta}$, so it suffices to consider $b\ne 0$. 
    
    Suppose for the purpose of contradiction that the desired inequality does not hold. Then \cref{lemma: tuple-counting} tells us that for all but a $6n^{-\delta}$-fraction of the triples of columns $J$, we have
    \[\Pr\big[\vec u\cdot \vy=0\text{, and }a\vec v[J]+b\vy^{\,T} A[J] = \vec 0\text{ for some }a\in \mb C\text{ and }b\in \mb C\setminus\{0\}\big]>n^{-1+3\delta}\frac{1}{h_2^\delta(\vec {u})}+n^{-3/2+3\delta}.\] Say that such $J$ are \emph{good}.

    For each good triple of columns $J$, let $P_J\in \mb C^{3\times 2}$ be a rank-2 matrix whose columns are both in the orthogonal complement of $\on{span}(\vec{v}[J])$,
    and let $((A[J] P_J)^T;\vec u)$ be the $3\times n$ matrix obtained by attaching the row $\vec u$ to the $2\times n$ matrix $(A[J] P_J)^T$. 
    So, if $\vec u\cdot \vy=0$ and $a\vec v[J]+b\vy^{\,T} A[J] = \vec 0$ for some $a\in \mb C$ and $b\in \mb C\setminus\{0\}$, then we have $((A[J] P_J)^T;\vec u)\vy=\vec 0$.
    
    Associating $((A[J] P_J)^T;\vec u)$ with its sequence of $n$ column vectors,
    we have $h_2^\delta((A[J] P_J)^T;\vec u)\ge h_2^\delta(\vec u)$, so we deduce
    \[
        \Pr[((A[J] P_J)^T;\vec u)\vy=\vec 0]
        >n^{-1+3\delta}\frac{1}{h_2^\delta((A[J] P_J)^T;\vec u)}+n^{-3/2+3\delta}.
    \]
    By \cref{lemma: prob via volume}, it follows that $h_r^\delta((A[J] P_J)^T;\vec u)<\infty$ for some $r\le 1$. 
    This means that all but at most $n^{1-\delta}$ of the rows of the $n\times 2$ matrix $A[J] P_J$ lie in a single symmetric GAP of rank at most 1. 
    So, every $3\times 2$ submatrix of $A[J] P_J$ which has rank 2 must involve at least one of the rows outside this GAP, meaning that $A[J]$ has at most $n^{1-\delta}\cdot \binom n2$ nonsingular $3\times 3$ submatrices.
    
    Summing over both good and non-good triples of columns, the total number of nonsingular $3\times 3$ submatrices in $A$ is at most $6n^{-\delta}\binom n3^2+n^{1-\delta}\binom n2\cdot \binom n3<\varepsilon\binom n 3^2$, contradicting the $\varepsilon$-robustness of $A$.
\end{proof}

We also need some estimates on the quantities $h_r^\delta$ from \cref{def:volume-GAP}.
\newcommand{\Ezero}{\mc{E}_0}
\begin{lemma}\label{lem:technical-estimates}
    Fix any $\delta>0$, and let $n$ be sufficiently large (in terms of $\delta$). Let $\vx\in \mb{C}^n$ be a vector of independent lazy Rademacher random variables.
    \begin{enumerate}
        \item [(1)] For a $\delta$-robust matrix $A \in \mb{C}^{n \times n}$, we have
        \[\mb E\left[\frac{1}{h_2^\delta(\vx^{\,T} A)} \right]\le n^{-1/6+\delta}.\]
        \item[(2)] For $\delta$-robust matrices $A_1,A_2 \in \mb{C}^{n \times n}$, we have
        \[\mb E\left[\frac{1}{h_1^\delta(\vx^{\,T} A_1)h_2^\delta(\vx^{\,T} A_2)} \right]\le n^{-2/3+\delta}.\]
    \end{enumerate}
\end{lemma}
\begin{remark}
    The bounds in \cref{lem:technical-estimates} are likely far from being optimal; we made no attempt to optimise them.
\end{remark}

We defer the proof of \cref{lem:technical-estimates} to the next subsection. First, we show how to deduce \cref{theorem: three bilinears}.

\newcommand{\Edep}{\mc{E}_{1}}
\begin{proof}[Proof of \cref{theorem: three bilinears}]
    Let \[
        \mc{E}:= \bigl\{\vx^{\,T} A_1 \vy=0\bigr\}\cap \bigl\{\vx^{\,T} A_2 \vz=\varphi(\vx,\vy)\bigr\} \cap \bigl\{\vy^{\,T} A_3 \vz=\psi(\vx,\vy)\bigr\}.
    \]
    We will prove that $\Pr[\mc{E}] \le n^{-2-1/6+\varepsilon}$ (i.e., \cref{H2}), assuming that $A_1,A_2,A_3$ are $\varepsilon$-robust (i.e., assuming that \cref{H1} does not hold). Throughout this proof, we assume $n$ is sufficiently large whenever necessary.
    
    Let $\delta=\varepsilon/100$, and define two auxiliary events:
    \begin{itemize}
        \item Let $\Ezero$ be the event that $\|\vx^{\,T} A_1\|_0 \leq n^{1-\delta}$ or $\|\vx^{\,T} A_2\|_0 \leq n^{1-\delta}$. 
        Hence, $\Ezero$ fails to hold if and only if $h_0^\delta(\vec{\alpha}^{\, T} A_1)=h_0^\delta(\vec{\alpha}^{\,T}A_2)=\infty$. 
        \item Let $\Edep$ be the event that there are $a,b\in \mb C$, not both zero, such that $\|a\vx^{\,T} A_2+b\vy^{\,T} A_3\|_0\le n^{1-\delta}$.
    \end{itemize}
    Note that if the event $\Edep$ does not hold, then there  is no symmetric GAP of rank at most 1 containing all but $n^{1-\delta}$ of the columns in the $2\times n$ matrix whose rows are $\vx^{\,T} A_2$ and $\vy^{\,T} A_3$.
    
    Now, $\Pr[\mc{E}]$ is at most
    \begin{equation} \label{eq: dividing dependency}
    \Pr[\mc{E} \cap \Ezero] \;+\; \Pr[\mc{E} \cap \Ezero^{\mr c} \cap \Edep] \;+\; \Pr[\mc{E} \cap \Ezero^{\mr c} \cap \Edep^{\mr c}].
    \end{equation}
    In the rest of the proof, we will show that each of the three terms above is at most $n^{-2-1/6+\varepsilon}/3$. 
    Throughout the proof, we will put subscripts on probabilistic notation to remind the reader whether we are using the randomness of  $\vx$, $\vy$ or $\vz$. Also, we simply write $h_0,h_1,h_2$ instead of $h_0^\delta,h_1^\delta,h_2^\delta$.

    \medskip\noindent\textbf{The first term.}
    By \cref{lemma: many nonzero entries}, we have $\Pr_{\vx}[\Ezero] \le 2n^{-3/2+\delta}$. 
    By \cref{lem:close-easy}, $A_3$ differs in more than $(\varepsilon/9)n^2$ entries from every matrix of rank at most 2 (in particular, this holds for every matrix of rank at most 1). So by \cref{thm:costello-mod}, for any outcome of $\vx$, we have 
    \[
        \Pr_{\vy,\vz}[\mc E\,|\,\vx]
        \le \Pr_{\vy,\vz}\bigl[ \vy^{\,T} A_3 \vz=\psi(\vx,\vy) \,\big|\, \vx\bigr]\le n^{-1+\varepsilon/9}.
    \]
     We deduce \[
        \Pr[\mc{E} \cap \Ezero]
        \le 2n^{-3/2+\delta}\cdot n^{-1+\varepsilon/9}
        < n^{-2-1/6+\varepsilon}/3
    \]
    (with plenty of room to spare in the final inequality).

    \medskip\noindent\textbf{The second term.} For any outcome of $\vx$ for which $\Ezero$ does not hold (in particular, $\|\vx^{\,T} A_2\|_0> n^{1-\delta}$), \cref{lemma: dependent-event} tells us that
    \[
        \Pr_{\vy}[\vx^{\,T}A_1\vy=0\text{, and }\Edep\text{ occurs}\,|\, \vx\,] 
        \le 2 \cdot n^{-1+3\delta}\cdot \frac{1}{h_2(\vx^{\,T} \!A_1)}+ 2\cdot n^{-3/2+3\delta}
        \le  n^{-1+4\delta}\cdot \frac{1}{h_2(\vx^{\,T} \!A_1)}+  n^{-3/2+4\delta}.
    \]
    Then, for any outcomes of $\vx,\vy$ for which $\Ezero$ does not hold (in particular, $\|\vx^{\,T} A_2\|_0> n^{1-\delta}$), we can apply \cref{lemma: prob via volume} (with $q=2$, using that $h_0(\vx^{\,T}\!A_2)=\infty$) to see that
    \[
        \Pr_{\vz}[\vx^{\,T} A_2 \vz = \varphi(\vx,\vy)\, |\, \vx,\vy\,] 
        \le n^{-1/2+2\delta}\frac{1}{h_1(\vx^{\,T} \!A_2)}+n^{-1+2\delta}.
    \]
    So, we have
    \begin{align*}
        \Pr[\mc{E} \cap \Ezero^{\mr c} \cap \Edep] 
        &\le \mb{E}_{\vx}\left[ \left(n^{-1+4\delta}\cdot \frac{1}{h_2(\vx^{\,T} \!A_1)}+n^{-3/2+4\delta}\right)\cdot \left(n^{-1/2+2\delta}\frac{1}{h_1(\vx^{\,T} \!A_2)}+n^{-1+2\delta}\right)\right]\\
        &\le n^{-3/2+6\delta}\,\mb{E}_{\vx}\left[ \frac{1}{h_1(\vx^{\,T} \!A_2) h_2(\vx^{\,T} \!A_1)}\right]+ n^{-2+6\delta}\,\mb{E}_{\vx}\left[\frac{1}{h_2(\vx^{\,T} \!A_1)}\right]\\
        &\qquad\qquad+n^{-2+6\delta}\,\mb{E}_{\vx}\left[\frac{1}{h_2(\vx^{\,T} \!A_2)}\right]+n^{-5/2+6\delta}
    \end{align*}
    (Here we used that $h_1(\vx^TA_2)\ge h_2(\vx^{\,T} A_2)$). Using \cref{lem:technical-estimates}, we deduce
    \[
        \Pr[\mc{E} \cap \Ezero^{\mr c} \cap \Edep]\le n^{-2-1/6+7\delta}+2n^{-2-1/6+7\delta}+n^{-5/2+6\delta}\le n^{-2-{1/6}+\varepsilon}/3.
    \]

    \medskip\noindent\textbf{The third term.}
    For any outcome of $\vx$ for which $\Ezero$ does not hold (in particular, $\|\vx^{\,T} A_1\|_0> n^{1-\delta}$), \cref{lemma: prob via volume} (with $q=2$, using that $h_0(\vx^{\,T}\!A_1)=\infty$) tells us that
    \[\Pr_{\vy}[\vx^{\,T} A_1 \vy=0 \,|\, \vx\,] \le \frac{1}{h_1(\vx^{\,T} \!A_1)}n^{-1/2+2\delta}+ n^{-1+2\delta}.\]
    Then, write $M$ for the $2\times n$ matrix whose rows are $\vx^{\,T} A_2$ and $\vy^{\,T} A_3$ (which we may view as a sequence of $n$ vectors in $\mb F^2$). For any outcomes of $\vx,\vy$ for which $\Edep$ does not hold, another application of \cref{lemma: prob via volume} (with $q=3$, using that $h_0(M)=h_1(M)=\infty$, and $h_2(M)\ge h_2(\vx^{\,T} \!A_2)$) tells us that
    \[\Pr_{\vz}[\vx^{\,T} A_2 \vz = \varphi(\vx,\vy)\text{ and }\vy^{\,T} A_3 \vz = \psi(\vx,\vy)\, |\, \vx,\vy\,] 
                \le \frac{1}{h_2(\vx^{\,T} \!A_2)}n^{-1+3\delta}+n^{-3/2+3\delta}.\]
    So, we have
    \begin{align*}
        \Pr[\mc{E} \cap \Ezero^{\mr c} \cap \Edep^{\mr c}] 
        &\le \mb{E}_{\vx}\left[ \left( \frac{1}{h_1(\vx^{\,T} \!A_1)}n^{-1/2+2\delta} + n^{-1+2\delta} \right) \cdot \left(\frac{1}{h_2(\vx^{\,T} \!A_2)}n^{-1+3\delta}+n^{-3/2+3\delta}\right)\right]\\
        &\le n^{-3/2+5\delta}\,\mb{E}_{\vx}\left[ \frac{1}{h_1(\vx^{\,T} \!A_1) h_2(\vx^{\,T} \!A_2)}\right]+ n^{-2+5\delta}\,\mb{E}_{\vx}\left[\frac{1}{h_2(\vx^{\,T} \!A_1)}\right]\\
        &\qquad\qquad+n^{-2+5\delta}\,\mb{E}_{\vx}\left[ \frac{1}{h_2(\vx^{\,T} \!A_2)}\right]+n^{-5/2+5\delta}
    \end{align*}
    (Here we used that $h_1(\vx^TA_1)\ge h_2(\vx^{\,T} A_1)$). Proceeding in the same way as for the second term, it follows from \cref{lem:technical-estimates} that $\Pr[\mc{E} \cap \Ezero^{\mr c} \cap \Edep^{\mr c}]\le n^{-2-{1/6}+\varepsilon}/3$, as desired.
\end{proof}

\subsection{Proofs of technical estimates}
In this subsection we prove \cref{lem:technical-estimates}, which is the remaining ingredient for the proof of \cref{theorem: three bilinears}.
It will be convenient to introduce some additional notation.
\begin{definition}
For a matrix $A\in\mb C^{n\times n}$ and a subset $J\subseteq\{1,\dots,n\}$,
recall that $A[J]$ denotes the submatrix of $A$ containing just the columns indexed
by $J$. For $r\in \mb N$ and a vector $\vec{u}\in\mb C^{m}$, let
$h^*_{r}(\vec{u})$ be the minimum volume of a symmetric GAP of rank at most $r$
that contains \emph{all} the entries of $\vec{u}$.
\end{definition}

We start by proving \cref{lem:technical-estimates}(1), which is a fairly simple consequence of \cref{lem:count-2}.

\begin{proof}
[Proof of \cref{lem:technical-estimates}(1)]Since $1/h_{2}^{\delta}(\vx^{\,T}A)$ only takes values
in the range $[0,1]$, we have
\begin{equation}
\mb E\left[\frac{1}{h_{2}^{\delta}(\vx^{\,T}A)}\right]=\int_{0}^{1}\Pr\left[\frac1{h_{2}^{\delta}(\vx^{\,T}A)}>u\right]\,du=\int_{1}^{\infty}V^{-2}\,\Pr[h_{2}^{\delta}(\vx^{\,T}A)<V]\,dV.\label{eq:h1-identity}
\end{equation}

To apply this identity, we need to estimate probabilities of the form
$\Pr[h_{2}^{\delta}(\vx^{\,T}A)<V]$, as follows.
\begin{claim}
\label{claim:h2}For any $V\ge1$, we have $\Pr[h_{2}^{\delta}(\vx^{\,T}A)<V]\le O_{\delta}(V^{3}(1+\log V)n^{-1/2})$.
\end{claim}
\begin{claimproof}
    We will apply \cref{lemma: tuple-counting} to reduce studying $h_2^\delta(\vec{\alpha}^{\,T}A)$ to studying $h_2^\delta(\vec{\alpha}^{\,T}A[J])$ for 3-element subsets $J\subseteq \{1,\dots,n\}$.
    Similar arguments appeared in the proof of \Cref{thm: multi-x-single-y costello}, so we will be quite brief with details.
    Suppose that $\Pr[h_{2}^{\delta}(\vx^{\,T}A)<V]>C_{\delta}V^{3}(1+\log V)n^{-1/2}$.
    We will show that this is a contradiction for large $C_{\delta}$.
    
    First, \cref{lemma: tuple-counting} tells us that for all but an $6n^{-\delta}$-fraction of
    3-element subsets $J\subseteq\{1,\dots,n\}$ we have $\Pr[h^*_{2}(\vx^{\,T}A[J])<V]>C_{\delta}V^{3}(1+\log V)n^{-1/2}/2$.
    Say that such $J$ are \emph{good}.

    For each good $J$, \cref{lem:count-2} (with $m=3$ and $k=1$ and $r=2$) tells us
    that $A[J]$ has at most $O(n/C_{\delta}^{2})$ disjoint nonsingular
    $3\times3$ submatrices, so by \cref{fact: disjoint and non-disjoint nonsingular submatrices} at most a $O(C_{\delta}^{-2})$-fraction
    of all the $3\times3$ submatrices of $A[J]$ are nonsingular.
    
    Recalling that all but a $6n^{-\delta}$-fraction of $J$ are good,
    we see that the total number of nonsingular $3\times3$ submatrices
    in $A$ is at most $6n^{-\delta}\binom{n}{3}^{2}+O(C_{\delta}^{-2})\binom{n}{3}^2$,
    which contradicts the $\delta$-robustness of $A$ for sufficiently
    large $C_{\delta}$.
\end{claimproof}
    Now, given \cref{eq:h1-identity} and \cref{claim:h2}, we compute 
    \[
        \mb E\left[\frac{1}{h_{2}^{\delta}(\vx^{\,T}A)}\right]\le\int_{1}^{n^{1/6}}V^{-2}\cdot O_{\delta}(V^{3}(1+\log V)n^{-1/2})\,dV+\int_{n^{1/6}}^{\infty}V^{-2}\,dV=n^{-1/6+o(1)},
    \]
    as desired.
\end{proof}
The proof of \cref{lem:technical-estimates}(2) is similar, but \cref{lem:count-2} is not quite strong enough to get a non-trivial bound. Instead, we need the following more refined estimate of Costello~\cite[Lemma 8]{Cos13}\footnote{
\cite[Lemma 8]{Cos13} is stated for Rademacher random variables, not lazy Rademacher random variables.
This does not have any impact on the proof (also, one can deduce the lazy Rademacher case from the Rademacher case by first revealing which of $\xi_1,\dots,\xi_n$ are zero; conditionally, the nonzero $\xi_i$ are Rademacher).}. He proved this via (elementary) number-theoretic considerations.
\begin{lemma}\label{lemma: pair volume by Costello0}
    Consider any vectors $\vec u,\vec v\in \mb C^m$, such that for each $i \in \{1,\dots,m\}$, at least one of $u_i$ and $v_i$ is nonzero.
    Let $\vec \alpha\in \mb C^{m}$ be a vector of i.i.d.\ lazy Rademacher random variables. For any real number $q\ge 1$, let $\mc E(q)$ be the event that there are integers $x,y\in \mb Z$ such that
    \begin{itemize}
        \item $|x|,|y|\le q$,
        \item $x\vec u-y\vec v$ has at least $m/10$ nonzero entries,
        \item $x(\vec u\cdot \vec \alpha)=y(\vec v\cdot \vec \alpha)$.
    \end{itemize}
    Then for any $1\le q\le \sqrt{m}$ we have
    \[\Pr[\mathcal E(q)]\le q m^{-1/2+o(1)}.\]
\end{lemma}
\cref{lemma: pair volume by Costello0} has the following corollary in terms of rank-1 GAPs.
\begin{corollary}\label{lemma: pair volume by Costello}
    Suppose $A\in \mb C^{n\times 2}$ contains at least $\ell$ disjoint nonsingular $2\times 2$ submatrices.
    Then, for every $1 \le q \le \sqrt{2\ell}$, 
    \[\Pr\big[h_1^*(\vx^{\,T} A)< q \big]\le q \cdot n^{-1/2+o(1)}.\]
\end{corollary}
\begin{proof}
    Write $\vec u$ and $\vec v$ for the two columns of $A$.
    Let $I \subseteq \{1,\dots,n\}$ be a set of $2\ell$ row indices corresponding to $\ell$ disjoint nonsingular $2\times 2$ submatrices, so for each $i \in I$ at least one of $u_i$ and $v_i$ is nonzero.
    Note that for any $x,y\in \mb C$ which are not both zero, the vector $x\vec u[I]-y\vec v[I]$ has at least $\ell$ nonzero entries.

    Now, if the two entries $\vec u\cdot \vec \alpha$ and $\vec v\cdot \vec \alpha$ of $\vx^{\,T} A$ lie in a symmetric rank-1 GAP with volume $2q+1$, then there are integers $x,y\in \mb Z$ with $|x|,|y|\le q$ and $x(\vec u\cdot \vec \alpha)=y(\vec v\cdot \vec \alpha)$, such that $x$ and $y$ are not both zero. We also know that $x(\vec u\cdot \vec \alpha)=x(\vec u[I]\cdot \vec \alpha[I])$ and $y(\vec v\cdot \vec \beta) = y(\vec v[I]\cdot \vec \beta[I])$. 
    We then apply \cref{lemma: pair volume by Costello0} (with $m=2\ell$) to $\vec{u}[I]$ and $\vec{v}[I]$.
\end{proof}
Now we prove \cref{lem:technical-estimates}(2).
\begin{proof}
[Proof of \cref{lem:technical-estimates}(2)]
Here, our starting point is the inequality
\begin{align}
\mb E\left[\frac{1}{h_{1}^{\delta}(\vx^{\,T}A_{1})h_{2}^{\delta}(\vx^{\,T}A_{2})}\right] &=\int_{1}^{\infty}t^{-2}\,\Pr[h_{1}^{\delta}(\vx^TA_1)h_{2}^{\delta}(\vx^{\,T}A_{2})<t]\,dt\nonumber\\
 &\le \int_{1}^{n^{2/3}}t^{-2}\,\Pr[h_{1}^{\delta}(\vx^TA_1)<t^{3/4}\text{ or }h_{2}^{\delta}(\vx^{\,T}A_{2})<t^{1/4}]\,dt+\int_{n^{2/3}}^\infty t^{-2}\,dt\nonumber\\
 & \le \int_{1}^{n^{2/3}}t^{-2}\left(\Pr[h_{1}^{\delta}(\vx^TA_1)<t^{3/4}]+\Pr[h_{2}^{\delta}(\vx^{\,T}A_{2})<t^{1/4}]\right)\,dt+n^{-2/3}.
 \label{eq:h1h2-identity}
\end{align}

To apply \cref{eq:h1h2-identity}, we need \cref{claim:h2}, together
with the following additional estimate.
\begin{claim}
\label{claim:h1h2}For any $1\le V\le\sqrt{2\delta n}$, we have $\Pr[h_{1}^{\delta}(\vx^{\,T}A_{1})<V]\le V\cdot n^{-1/2+o(1)}$.
\end{claim}

\cref{claim:h1h2} can be proved using \cref{lemma: pair volume by Costello}, in a very similar way to the proof of \cref{claim:h2} using the $r=2$ case of \cref{lem:count-2}, but now studying $h_1^\delta(\vec{\alpha}^{\,T}A[J])$ for 2-element subsets $J\subseteq \{1,\dots,n\}$ and counting nonsingular $2\times 2$ submatrices (recalling \cref{lem:rank-monotonicity} to obtain the final contradiction).
Now, substituting \cref{claim:h2,claim:h1h2} into \cref{eq:h1h2-identity}, we obtain
\begin{align*}
    \mb E\left[\frac{1}{h_{1}^{\delta}(\vx^{\,T}A_{1})h_{2}^{\delta}(\vx^{\,T}A_{2})}\right] 
    &\le \int_{1}^{n^{2/3}}t^{-2}\left( t^{3/4}n^{-1/2+o(1)}+ t^{3/4}\cdot n^{-1/2+o(1)}\right)\,dt+n^{-2/3}\\
    &\le n^{-1/2+o(1)}\cdot \int_1^{n^{2/3}} t^{-5/4}\,dt+n^{-2/3}\\
    &= n^{-1/2+o(1)}\cdot (n^{2/3})^{-1/4}+n^{-2/3}=n^{-2/3+o(1)}. \qedhere
\end{align*}
\end{proof}

\bibliographystyle{amsplain_initials_nobysame_nomr}
\bibliography{main}

 \appendix
\section{A Littlewood--Offord theorem for varieties}\label{sec:FKS}
    In this appendix we prove \cref{thm:FKS}, giving a bound of the form $\Pr[\xi_{1}\vec{a}_{1}+\dots+\xi_{n}\vec a_{n}\in\mathcal{Z}]\le n^{-1/2+o(1)}$
    unless almost all the vectors $\vec{a}_{1},\dots,\vec{a}_{n}$ lie
    in a linear subspace $\mc W$ such that some translate $\vec{w}+\mc W$ of $\mc W$ is contained in $\mathcal{Z}$.

\begin{proof}[Proof of \cref{thm:FKS}]
    Assume \cref{E1} does not hold; we will prove the bound in \cref{E2}. We may also assume that $\mc Z\ne \emptyset$ (otherwise \cref{E2} holds trivially). Let $P_1,\dots,P_z$ be nonzero polynomials with $\mc Z=\{\vec y\in \mb F^d: P_1(\vec y)=\dots=P_z(\vec y)=0\}$, and let $q^*=\max(\deg P_1,\dots,\deg P_z)$.

    We first find a subspace $\mc W\subseteq \mb{F}^d$, for which one can find ``many'' disjoint basis among the vectors $\vec{a}_1,\dots,\vec{a}_n$. Writing $n_{b}=n-\varepsilon n(d-b)/d$, let $b\in\{0,\dots,d\}$ be the minimum integer such that at least $n_{b}$ of the vectors $\vec{a}_{i}$ lie in a common linear subspace $\mathcal{W}\subseteq\mb F^{d}$ of dimension $b$
    (certainly such a $b$ exists, considering $b=d$).
    We cannot have $b=0$, as otherwise \ref{E1} would hold with $\mc W = \{\vec{0}\}$ and with $\vec{w}$ as an arbitrary element of $\mc Z$.
    In addition, there is no proper subspace $\mathcal{W}'\subsetneq\mathcal{W}$ containing at least $n_{b-1}$ of the vectors $\vec{a}_{i}$, meaning that among the vectors $\vec{a}_{i}$ we can find at least $(n_{b}-n_{b-1})/b \ge \varepsilon n/d^{2}$ disjoint bases of $\mathcal{W}$.

    Upon relabelling the indices, without loss of generality we may assume that $\vec{a}_{jb+1},\vec{a}_{jb+2},\dots,\vec{a}_{(j+1)b}$ form a basis of $\mc{W}$ for all integers $0\le j< \lfloor \varepsilon n/d^{2}\rfloor$.
    Let $m= \lfloor\varepsilon n/d^{2}\rfloor\cdot b$, then we have $\vec{a}_{1},\dots,\vec{a}_{m}\in \mc W$.
    
    Now, condition on any outcome of $(\xi_{i})_{i>m}$; 
    we will prove the desired bound conditional on this outcome (all probabilistic
    notation for the rest of the proof is implicitly with respect to the
    corresponding conditional probability space). 

    Let $\vec{w}=\sum_{i>m}\xi_{i}\vec{a}_{i}$ (which we no longer view as a random variable). 
    Fix a linear isomorphism $\varphi:\mathcal{W}\to\mb F^{b}$ and let $\mathcal{Z}'=\varphi(\mathcal{W}\cap(\mathcal{Z}-\vec{w}))\subseteq \mb{F}^b$. Recall that at least $n_{b}\ge(1-\varepsilon)n$ of the vectors $\vec{a}_{i}$ lie in $\mathcal{W}$, and we are assuming \cref{E1} does not hold, so $\vec{w}+\mathcal{W}$ is not fully contained in $\mathcal{Z}$, meaning that $\mathcal{Z}'\subsetneq\mb F^{b}$.
    
    Let $P\in \mb F[x_1,\dots,x_b]$ be a nonzero polynomial of degree $q:=\deg P\le q^*$ vanishing on $\mathcal{Z}'$. 
    To see that such a polynomial exists, note that $\mathcal{W}\cap(\mathcal{Z}-\vec{w})=\{\vec y\in \mathcal{W}:P_1(\vec y+\vec w)=\dots=P_z(\vec y+\vec w)=0\}$, and hence $\mathcal{Z}'\subset \mb{F}^b$ can be described as the vanishing set of up to $z$ polynomials with degrees at most $\deg P_1,\dots,\deg P_z$, respectively (at least one of these polynomials is nonzero, since $\mathcal{Z}'\ne \mb{F}^b$).

    Now, for $i=1,\dots,m$, let $\pvec a_{i}=\varphi(\vec{a}_{i})$. 
    Then for each $j=0,\dots,\lfloor \varepsilon n/d^{2}\rfloor-1$, the vectors $\pvec{a}_{jb+1},\dots,\pvec{a}_{(j+1)b}$ form a basis of $\mb F^b$. 
    Furthermore, we have
    \begin{align*}
        \Pr[\xi_{1}\vec{a}_{1}+\dots+\xi_{n}\vec{a}_{n}\in\mathcal{Z}]
        &=\Pr[\xi_{1}\vec{a}_{1}+\dots+\xi_{m}\vec{a}_{m}\in\mathcal{Z}-\vec{w}]
        =\Pr[\xi_{1}\vec{a}_{1}+\dots+\xi_{m}\vec{a}_{m}\in\mathcal{W}\cap (\mathcal{Z}-\vec{w})]\\
        &=\Pr[\xi_{1}\pvec{a}_{1}+\dots+\xi_{m}\pvec{a}_{m}\in\mathcal{Z}']
        \le \Pr[P(\xi_{1}\pvec{a}_{1}+\dots+\xi_{m}\pvec{a}_{m})=0].
    \end{align*}
    Note that we can interpret $P(\xi_{1}\pvec{a}_{1}+\dots+\xi_{m}\pvec{a}_{m})$ as a polynomial of degree (at most) $q$ in the variables $\xi_{1},\dots,\xi_{m}$. 
    It suffices to show that this polynomial has at least $\varepsilon'm^{q}$ nonzero coefficients for $\varepsilon'=\min(\varepsilon/2, 1/(2d)^{q^*})$. Indeed, then the Meka--Nguyen--Vu bound for the polynomial Littlewood--Offord problem (see \cref{thm:polynomial-LO}) implies 
    \[\Pr[\xi_{1}\vec{a}_{1}+\dots+\xi_{n}\vec{a}_{n}\in\mathcal{Z}]\le \Pr[P(\xi_{1}\pvec{a}_{1}+\dots+\xi_{m}\pvec{a}_{m})=0]\le m^{-1/2+\varepsilon'}\le n^{-1/2+\eps},\]
    establishing \cref{E2}.

    For any distinct indices $i(1),\dots,i(q)\in \{1,\dots,m\}$, the coefficient of $\xi_{i(1)}\dotsm\xi_{i(q)}$ in the polynomial $P(\xi_{1}\pvec{a}_{1}+\dots+\xi_{m}\pvec{a}_{m})$ is the same as the coefficient of $\xi_{i(1)}\dotsm\xi_{i(q)}$ in $P(\xi_{i(1)}\pvec{a}_{i(1)}+\dots+\xi_{i(q)}\pvec{a}_{i(q)})$ and hence the same as the coefficient of $t_1\dotsm t_q$ in $P(t_1\pvec{a}_{i(1)}+\dots+t_q\pvec{a}_{i(q)})$. But note that for any vectors $\vec{v}_1,\dots,\vec{v}_q\in \mb F^b$ the coefficient of $t_1\dotsm t_q$ in $P(t_1 \vec{v}_1+\dots+t_q \vec{v}_q)$ is a multilinear function of $\vec{v}_1,\dots,\vec{v}_q$ (more formally, the function assigning a $q$-tuple $(\vec{v}_1,\dots,\vec{v}_q)$ the value of this coefficient is a multilinear function $\mb F^b\times \dots\times \mb F^b\to \mb F$). Since $P$ has degree $q$, this multilinear function is nonzero (indeed, the homogeneous degree-$q$ part $P_q$ of the polynomial $P$ is non-zero, so we can find a vector $\vec{w} \in \mb{F}^b$ with $P_q(\vec{w}) \neq 0$, and observe that the coefficient of $t_1\dotsm t_q$ in $P(t_1 \vec{w}+\dots+t_q \vec{w})$ is the same as in $P_q(t_1 \vec{w}+\dots+t_q \vec{w})=(t_1+\dots+t_q)^q P_q(\vec{w})$ and therefore equal to $q!P_q(\vec{w}) \neq 0$). 
    This means that, given any bases $B_1,\dots,B_q$ of $\mb F^b$, we can choose vectors $\vec{v}_1\in B_1,\dots,\vec{v}_q\in B_q$ such that the coefficient of $t_1\dotsm t_q$ in $P(t_1 \vec{v}_1+\dots+t_q \vec{v}_q)$ is nonzero. 
    In particular, for any distinct $j(1),\dots,j(q)\in \{0,\dots,\lfloor \varepsilon n/d^{2}\rfloor-1\}$, we can find $i(1)\in \{j(1)b+1,\dots,(j(1)+1)b\},\dots,i(q)\in \{j(q)b+1,\dots,(j(q)+1)b\}$ such that the coefficient of $t_1\dotsm t_q$ in $P(t_1 \pvec{a}_{i(1)}+\dots+t_q \pvec{a}_{i(q)})$ is nonzero and hence the coefficient of $\xi_{i(1)}\dotsm\xi_{i(q)}$ in $P(\xi_{i(1)}\pvec{a}_{i(1)}+\dots+\xi_{i(q)}\pvec{a}_{i(q)})$ is nonzero. 
    This means that we can find at least $\binom{\lfloor \varepsilon n/d^{2}\rfloor}{q}=\binom{m/b}{q}> (m/(2d))^q\ge m^q/(2d)^{q^*}\ge \eps' m^q$ distinct $q$-element sets $\{i(1),\dots, i(q)\}\subset \{1,\dots,m\}$, such that the coefficient of $\xi_{i(1)}\dotsm\xi_{i(q)}$ in the polynomial $P(\xi_{1}\pvec{a}_{1}+\dots+\xi_{m}\pvec{a}_{m})$ is nonzero. Thus, the polynomial $P(\xi_{1}\pvec{a}_{1}+\dots+\xi_{m}\pvec{a}_{m})$ has indeed at least $\varepsilon'm^{q}$ nonzero coefficients. 
    This completes the proof.
    \end{proof}

\section{A counterexample to some conjectures of Costello \\(By Matthew Kwan, Ashwin Sah and Mehtaab Sawhney)}\label{sec:counterexample}

Here we prove \cref{prop:costello-false}. First, the second bullet point is straightforward: note that
\[\Pr[f(\xi_{1},\dots,\xi_{n})=0]\ge\Pr\big[L_{1}(\xi_{1},\dots,\xi_{n})=0\text{ and }L_{2d}(\xi_{1},\dots,\xi_{n})=0\big]=\Omega((1/\sqrt{n})^{2})\ge \varepsilon/n,\]for sufficiently small $\varepsilon$.

In the rest of this appendix we prove that \cref{B1} does not hold for this polynomial $f$, showing that there are no non-constant polynomials $g_1,g_2$ for which $f-g_1g_2$ has fewer than $\varepsilon n^d$ nonzero terms. The idea is that if $f-g_1g_2$ had fewer than $\varepsilon n^d$ nonzero terms (i.e., if $f$ had an ``approximate factorisation''), this would lead to an \emph{exact} factorisation of the irreducible polynomial $y_1\cdots y_d+y_{d+1}\cdots y_{2d}$, which is impossible.

To execute this idea, we start with a random sampling argument to strengthen our ``approximate factorisation'', showing that the only obstruction comes from non-multilinear terms. We then use a Ramsey-theoretic argument (related to ideas of Alon and Beigel~\cite{AB01}) to ``clean up'' $g_1$ and $g_2$. We start with some preparations.

\begin{definition}
Fix a vector $\vec{t}=(t_{1},\dots,t_{r})\in\mb N^{r}$, with entries
summing to $k=\|\vec{t}\|_{1}$. Let $I_{1},\dots,I_{r}$ be disjoint
sets of size $m$, and let $K_{m}(\vec{t})$ be the $k$-uniform hypergraph
on the vertex set $I_{1}\cup \dots\cup I_{r}$, containing every possible
edge $e$ for which $|e_{1}\cap I_{1}|=t_{1},\dots,|e_{r}\cap I_{r}|=t_{r}$.
\end{definition}

For example, $K_{m}((k))$ is the complete $k$-uniform hypergraph
on $m$ vertices, and $K_{m}((1,1))$ is the complete bipartite graph
on $m+m$ vertices. We will need the following Ramsey-type theorem.
\begin{lemma} \label{lem:ramsey}
    For any $r,s,b\ge 1$ and $\vec t\in \mb N^r$, there is $M\in \mb N$ such that the following holds. 
    Let $m\ge M$ and consider any colouring of the hyperedges of $K_{m}(\vec{t})$ with $b$ different colours.
    Then there are subsets $I_{1}'\subseteq I_{1},\dots,I_{r}'\subseteq I_{r}'$ of size at least $s$ such that all the edges of $K_{m}(\vec{t})$ inside $I_{1}'\cup\dots\cup I_{r}'$ have the same colour.
\end{lemma}
\begin{proof}
    Put $k:=\|\vec{t}\|_1$ and let $G$ be the complete $k$-uniform hypergraph on the vertex set $\{1,\dots,m\}$. 
    We will use the edge-colouring of $K_{m}(\vec{t})$ to define an edge-colouring of $G$, to which we will apply Ramsey's theorem.
    
    Let $j(0)=0$, and for $\ell\in\{1,\dots,k\}$, define $j(\ell)$ to satisfy
    $t_{1}+\dots+t_{j(\ell)-1}<\ell\le t_{1}+\dots+t_{j(\ell)}$. 
    Writing $I_{j}=\{i_{1}^{j},\dots,i_{m}^{j}\}$
    for each $j\in\{1,\dots,r\}$, we define a mapping $\varphi$ from edges
    of $G$ to edges of $K_{m}(\vec{t})$ as follows. For an edge $e=\{q(1),\dots,q(k)\}$
    of $G$, where $q(1)<q(2)<\dots<q(k)$, let $\varphi(e)=\left\{i_{q(1)}^{j(1)},\dots,i_{q(k)}^{j(k)}\right\}$.
    
    Via this mapping, our edge-colouring of $K_{m}(\vec{t})$ induces
    an edge-colouring of $G$. So, by Ramsey's theorem, assuming $m$
    is sufficiently large there is a subset $Q$ of $sr$ vertices of
    $G$ such that all edges of $G$ inside $Q$ have the same colour.
    Order the elements of $Q$ as
    \[
    q(1,1)<\dots<q(1,s)<\dots<q(r,1)<\dots<q(r,s),
    \]
    and for each $j\in\{1,\dots,r\}$ let $I_{j}'=\{i_{q(j,1)}^{j},\dots,i_{q(j,s)}^{j}\}.$
\end{proof}

We also need a simple combinatorial fact about of sums of vectors.

\begin{definition}
For a vector $\vec v\in \mb N^r$ let $\vec{v}_{\downarrow}$ be the \emph{decreasing rearrangement} of $\vec{v}$
(obtained by sorting the entries of $\vec{v}$ in decreasing order).
Recall that in the \emph{lexicographic order} $\preceq$
on $\mb N^r$, we write $\vec{p}\preceq\vec{q}$ if $\vec{p}=\vec{q}$
or if $p_{i}<q_{i}$ for the first $i$ where $p_{i}$ and $q_{i}$
differ.
\end{definition}

\begin{lemma} \label{lem:colexicographic}
    Let $\vec{s}_{1},\vec{s}_{2},\vec{t}_{1},\vec{t}_{2}\in\mb N^{r}$ be vectors with $\vec{s}_{1}+\vec{s}_{2}=\vec{t}_{1}+\vec{t}_{2}=\vec v$.
    If $(\vec{s}_{1},\vec s_2)\ne(\vec{t}_{1},\vec t_2)$ then 
    \[
        (\vec{s}_{1}+\vec{t}_{2})_{\downarrow}\succ\vec v_\downarrow\text{ or }(\vec{t}_{1}+\vec{s}_{2})_{\downarrow}\succ\vec v_\downarrow.
    \]
\end{lemma}
\begin{proof}
    Without loss of generality, suppose that $\vec v_\downarrow=\vec v$.
    Note that the condition $(\vec{s}_{1},\vec s_2)\ne(\vec{t}_{1},\vec t_2)$ is equivalent to the condition $\vec s_1\ne \vec t_1$ (since $\vec s_1+\vec s_2=\vec t_1+\vec t_2$). 
    Let $i\in\{1,\dots,r\}$ be minimal such that $s_{1,i}\ne t_{1,i}$. If $s_{1,i}> t_{1,i}$ then $s_{2,i}< t_{2,i}$ and 
    $(\vec{s}_{1}+\vec{t}_{2})_\downarrow \succeq \vec{s}_1+\vec{t}_2 \succ \vec{v}$. 
    On the other hand, if $s_{1,i}<t_{1,i}$ then $s_{2,i}>t_{2,i}$ and $(\vec{t}_{1}+\vec{s}_{2})_\downarrow\succeq\vec{t}_{1}+\vec{s}_{2}\succ\vec v$.
\end{proof}

Now we prove \cref{prop:costello-false}.
\begin{proof}[Proof of \cref{prop:costello-false}]
Fix $d\ge 2$. As discussed at the start of this appendix, we need to show that for our particular polynomial $f$, \cref{B1} does not hold. Suppose for the purpose of contradiction that \begin{equation}f=g_{1}g_{2}+p,\label{eq:contradiction-factorisation}\end{equation} for
some polynomials $g_{1}g_{2},p\in\mathbb{C}[x_{1},\dots,x_{n}]$ such
that $\deg(g_{1}),\deg(g_{2})\ge1$ and $\deg(g_{1}g_{2}),\deg(p)\le d$,
and $p$ has at most $\varepsilon n^{d}$ nonzero coefficients. Let
$d_{1}=\deg(g_{1})$ and $d_{2}=d-d_1\ge \deg(g_{2})$.

Let $N$ be a large integer (we will need it to be sufficiently large in terms of $d$, at a later point in the proof).
\begin{claim}\label{claim:random-sampling}
    If $\varepsilon>0$
is sufficiently small (in terms of $N,d$), then there are subsets $I_{1}'\subseteq I_{1},\dots,I_{2d}'\subseteq I_{2d}$
of size $N$, such that for every set $\{i_{1},\dots,i_{d}\}\in I_{1}'\cup\dots\cup I_{2d}'$
of $d$ distinct indices, the coefficient of $x_{i_{1}}\dots x_{i_{d}}$
in $p$ is zero.
\end{claim}
\begin{claimproof}
    Independently for each $j\in\{1,\dots,2d\}$, let $I_j'\subseteq I_j$ be a uniformly random subset of size $N$. 
    We will show that the sets $I_j'$ satisfy the desired property with positive probability.
    
    Indeed, note that for any particular set $Q\subseteq I_j$, we have $\Pr[Q\subseteq I_j']\le (N/|I_j|)^{|Q|}$. 
    Thus, for each monomial $x_{i_1}\dots x_{i_d}$ of $p$ with nonzero coefficient and distinct indices $i_{1},\dots,i_{d}$ (there are at most $\varepsilon n^d$ of them), the probability of having $\{i_1,\dots,i_d\}\subseteq I_1'\cup \dots\cup I_{2d}'$ is at most $(N/(2\lfloor n/(4d)\rfloor))^d<1/(\varepsilon n^d)$ for sufficiently small $\varepsilon$.
    The claim follows from a union bound over all such monomials.
\end{claimproof}

\cref{claim:random-sampling} tells us that if we restrict our polynomials $f,g_1,g_2,p$ to the variables in $I_{1}'\cup\dots\cup I_{2d}'$, then we have a slightly ``cleaner'' version of \cref{eq:contradiction-factorisation} in which we have eliminated all the degree-$d$ multilinear terms in $p$. Unfortunately, the non-multilinear degree-$d$ terms in $p$ can still cause problems, so we need to refine the situation further.

For $S\subset I_{1}'\cup\dots\cup I_{2d}'$, let $\on{type}(S)=(|S\cap I_{1}'|,\dots,|S\cap I_{2d}'|)$
be a vector encoding the number of elements of $S$ that lie in each $I_{j}'$.
Slightly abusively, we conflate the set $S \subset I_1'\cup\dots\cup I_{2d}'$ with the multilinear monomial $\prod_{i\in S}x_{i}$ 
(so we can also talk about the type of a multilinear monomial, or the type of a multilinear term of a polynomial). 
Say a\emph{ $k$-type} is a vector $\vec{t}\in\mb N^{2d}$ with entries
summing to $k$. 
So, the degree-$d_1$ multilinear terms of $g_{1}$ can be categorised by $d_{1}$-type, the degree-$d_2$ multilinear terms of $g_{2}$ can be categorised by $d_{2}$-type, and the terms of $f$ (all of which are degree-$d$, and multilinear) can be categorised by $d$-type.

For $z\in\mb C$, with argument $\arg(z)\in[0,2\pi)$, define the \emph{direction}
\[
\on{dir}(z)=\begin{cases}
* & \text{if }z=0\\
\pi/2 & \text{if }\arg(z)\in[\pi/4,3\pi/4)\\
\pi & \text{if }\arg(z)\in[3\pi/4,5\pi/4)\\
3\pi/2 & \text{if }\arg(z)\in[5\pi/4,7\pi/4)\\
0 & \text{otherwise}.
\end{cases}
\]
In other words, if $z\in\mb C$ is nonzero, then $\on{dir}(z)$ is chosen such that $\arg(z)$ differs from $\on{dir}(z)$ by an angle of at most $\pi/4$. 
Slightly abusively, we conflate each term of a polynomial with its coefficient (so we can talk about the directions of terms of $g_{1}$ and $g_{2}$). 
We will need the following (easy) fact about directions.
\begin{fact} \label{fact:no-cancellation-direction}
    For any $x_{1},\dots,x_{k},y_{1},\dots,y_{k}\in\mb C \setminus \{0\}$, if 
    \[
        \on{dir}(x_{1})=\dots=\on{dir}(x_{k}),\quad\on{dir}(y_{1})=\dots=\on{dir}(y_{k}),
    \]
    then $x_{1}y_{1}+\dots+x_{k}y_{k}\ne 0$.
\end{fact}

Now, by iteratively applying \cref{lem:ramsey} (once for each $d_{1}$-type and once for each $d_{2}$-type), assuming
$N$ is sufficiently large (in terms of $d$), we can find $I_{1}''\subseteq I_{1}',\dots,I_{2d}''\subseteq I_{2d}'$,
each of size $d$, such that, among multilinear degree-$d_{1}$ terms
of $g_{1}$ (respectively, degree-$d_{2}$ terms of $g_{2}$) containing
variables indexed by $I_{1}''\cup\dots\cup I_{2d}''$, the direction
of a term depends only on its type.

Now, let $h_{1},h_{2},q,R_{1},\dots,R_{2d}\in\mb C[x_{i}:i\in I_{1}''\cup\dots\cup I_{2d}'']$
be the $2d^{2}$-variable polynomials obtained from $g_{1},g_{2},q,L_{1},\dots,L_{2d}$
by setting all variables not indexed by $I_{1}''\cup\dots\cup I_{2d}''$
to zero. (So, in particular, we have $R_j=\sum_{i\in I_j''} x_i$ for $j=1,\dots,2d$.)
Then, \cref{eq:contradiction-factorisation} gives rise to the ``cleaner'' identity
\begin{equation}
R_{1}\cdots R_{d}+R_{d+1}\cdots R_{2d}=h_{1}h_{2}+q,\label{eq:cleaner-identity}
\end{equation}
where $q$ has no multilinear degree-$d$ terms (by \cref{claim:random-sampling}), and in $h_{1}$ (respectively,
$h_{2}$), the direction of a degree-$d_{1}$ (respectively, degree
$d_{2}$) multilinear term only depends on its type. So, we can talk about the direction of a $d_{1}$-type in $h_{1}$ or the direction of a $d_{2}$-type in $h_{2}$.

Now, say that a type is \emph{simple} if all of its entries are at
most 1.
\begin{claim}
\label{claim:disjoint-terms}Let $\vec{t}_{1}$ be a $d_{1}$-type
and let $\vec{t}_{2}$ be a $d_{2}$-type, such that $\vec{t}_{1}+\vec{t}_{2}$
is not simple. Then $\vec{t}_{1}$ has direction $*$ in $h_{1}$
or $\vec{t}_{2}$ has direction $*$ in $h_{2}$.
\end{claim}

\begin{claimproof}
    Suppose for the purpose of contradiction that the statement of this claim is false, so there are $\vec{t}_{1},\vec{t}_{2}$ such that $\vec{t}_{1}+\vec{t}_{2}$
    is not simple but $\vec{t}_{1}$ does not have direction $*$ in $h_{1}$
    and $\vec{t}_{2}$ does not have direction $*$ in $h_{2}$. Choose
    such a pair $\vec{t}_{1},\vec{t}_{2}$ such that $(\vec{t}_{1}+\vec{t}_{2})_{\downarrow}$
    is lexicographically maximal.
    
    Now, let $\{i_{1},\dots,i_{d}\}\subseteq I_{1}''\cup\dots\cup I_{2d}''$
    be a set of distinct indices with type $\vec{t}_{1}+\vec{t}_{2}$
    (such a set exists, since each $I_{j}''$ has size $d$).
    Since $\vec{t}_{1}+\vec{t}_{2}$ is not simple, the coefficient of
    $x_{i_{1}}\cdots x_{i_{d}}$ in $R_{1}\cdots R_{d}+R_{d+1}\dots R_{2d}$
    is zero. By \cref{eq:cleaner-identity}, this coefficient can also
    be written as 
    \begin{equation}
    0=\sum_{A,B}\alpha_{A}\beta_{B},\label{eq:ST-sum}
    \end{equation}
    where we write $\alpha_{A}$ for the coefficient of $\prod_{i\in A}x_{i}$
    in $h_{1}$, and we write $\beta_{B}$ for the coefficient of $\prod_{i\in B}x_{i}$
    in $h_{2}$, and the sum is over all partitions of $\{i_{1},\dots,i_{d}\}$
    into a set of $A$ of size $d_{1}$ and a set $B$ of size $d_{2}$.
    Note that for all such $A,B$ we always have $\on{type}(A)+\on{type}(B)=\vec{t}_{1}+\vec{t}_{2}$.
    
    By \cref{fact:no-cancellation-direction}, and the choice of $\vec{t}_{1},\vec{t}_{2}$ (and the meaning of the direction $*$), there is a nonzero contribution to \cref{eq:ST-sum} from pairs $A,B$ with $\on{type}(A)=\vec{t}_{1}$ and $\on{type}(B)=\vec{t}_{2}$.
    So, there must be an additional $d_{1}$-type $\vec{s}_{1}$ and an additional $d_{2}$-type $\vec{s}_{2}$, such that $(\vec{s}_{1},\vec s_2)\ne(\vec{t}_{1},\vec t_2)$ and such that $\vec{s}_{1}+\vec{s}_{2}=\vec{t}_{1}+\vec{t}_{2}$,
    but $\vec{s}_{1}$ does not have direction $*$ in $h_{1}$ and $\vec{s}_{2}$ does not have direction $*$ in $h_{2}$. 
    But by \cref{lem:colexicographic}, at least one of $(\vec{t}_{1}+\vec{s}_{2})_\downarrow$ or
    $(\vec{s}_{1}+\vec{t}_{2})_\downarrow$ is lexicographically greater than $(\vec{t}_{1}+\vec{t}_{2})_\downarrow$ (and therefore not simple), which contradicts the choice of $\vec{t}_{1},\vec{t}_{2}$.
\end{claimproof}
Now, say a term of $h_{1}$ (respectively, of $h_{2}$) is \emph{good}
if it has degree $d_{1}$ (respectively, degree $d_{2}$) and is multilinear.
If two terms with types $\vec t_1,\vec t_2$ share a variable, then $\vec t_1+\vec t_2$ is not simple. So, \cref{claim:disjoint-terms} tells us that no variable appears in nonzero good terms of both $h_{1}$ and $h_{2}$.
Now, arbitrarily choose $i_{1}\in I_{1}'',\dots,i_{2d}\in I_{2d}''$.
Let $h_{1}^{*},h_{2}^{*}\in\mb C[x_{i_1},\dots,x_{i_{2d}}]$ be the $2d$-variable
polynomials obtained from $h_{1},h_{2}$ by setting all variables
other than $x_{i_1},\dots,x_{i_{2d}}$ to zero, and deleting all terms which
are not good. By the above discussion, $h_{1}^{*}$ and $h_{2}^{*}$ involve disjoint sets of variables (so each term of $h_{1}^{*}h_{2}^{*}$ is multilinear, so there can be no cancellation between $h_{1}^{*}h_{2}^{*}$ and $q$). Also, there can be no cancellation between $h_{1}^{*}h_{2}^{*}$ and terms arising from the non-good terms omitted in $h_{1}^{*}$ and $h_{2}^{*}$ (since the terms arising this way cannot be degree-$d$ multilinear). So, \cref{eq:cleaner-identity} gives rise to the identity
\[
x_{i_1}\cdots x_{i_d}+x_{i_{d+1}}\cdots x_{i_{2d}}=h_{1}^{*}h_{2}^{*}.
\]
(Specifically, on both sides of \cref{eq:cleaner-identity}, we have set all variables other than $x_{i_1},\dots,x_{i_{2d}}$ to zero, and restricted to degree-$d$ multilinear terms).
But it is easy to see that this factorisation is impossible.
\end{proof}

\end{document}